\documentclass[12pt,twoside]{article}
\usepackage{latexsym}
\usepackage{amssymb,amsbsy,amsmath,amsfonts,amssymb,amscd,mathtools,dsfont}
\usepackage{graphicx}
\usepackage{hyperref}
\usepackage{pstricks}
\newcommand{\commentout}[1]{}
\newcommand{\R}{\mathbb{R}}

\newcommand {\e}  {\varepsilon}

\newcommand {\Chi} {{\bf \raise 2pt \hbox{$\chi$}} }

\newcommand {\f}   {\frac}
\newcommand {\p}   {\partial}
\newcommand{\fer}{\eqref}
\newcommand{\beq}{\begin{equation}}
\newcommand{\eeq}{\end{equation}}
\newcommand{\bea} {\begin{array}{rl}}
\newcommand{\eea} {\end{array}}
\newcommand{\bepa}{\left\{ \begin{array}{l}}
\newcommand{\eepa} {\end{array}\right.}
\newtheorem{theorem}{Theorem}[section]
\newtheorem{lemma}[theorem]{Lemma}
\newtheorem{definition}[theorem]{Definition}
\newtheorem{remark}[theorem]{Remark}
\newtheorem{proposition}[theorem]{Proposition}

\newcommand{\qed}{{ \hfill
                       {\unskip\kern 6pt\penalty 500 \raise -2pt\hbox{\vrule\vbox to 6pt{\hrule width 6pt
                       \vfill\hrule}\vrule} \par}   }}
\title{Analysis of  the steady solutions of the Fisher's infinitesimal model; a Hilbertian approach
 }
\author{M. Hillairet\thanks{
Institut Montpelli\'erain Alexander Grothendieck, ANGUS, Univ Montpellier, Inria, CNRS, Montpellier, France; E-mail: \texttt{matthieu.hillairet@umontpellier.fr}} \and S. Mirrahimi\thanks{Sepideh Mirrahimi. Univ Toulouse, INSA Toulouse, CNRS, IMT, Toulouse, France.
		E-mail:  \texttt{sepideh.mirrahimi@math.univ-toulouse.fr}  }}

\date{\today}

\begin{document}
\maketitle
\pagestyle{plain}
\pagenumbering{arabic}

\abstract{We provide an asymptotic analysis of a nonlinear integro-differential equation which describes the evolutionary dynamics of a population which reproduces sexually and which is subject to selection and competition. The sexual reproduction is modeled via a nonlinear integral term, known as the Fisher's ’infinitesimal model’.  We consider a small segregational variance regime, where a parameter in the infinitesimal model, which measures the deviation between the trait of the offspring and the mean parental trait, is small with respect to the selection variance. In this regime, we characterize the steady states of the problem and  analyze their stability. Our method relies on a spectral analysis involving Hermite polynomials, highlighting the specific structure of the nonlinear reproduction term. We expect that the framework developed in this article will contribute to progress on several related problems that were out of reach with previous methods. }
\medskip

\noindent {\em 2020 M.S.C. :} 35B40, 35Q92, 92D15, 47G20.

\medskip


\section{Introduction}

\subsection{Model and question}

The purpose of this article is to study the steady solutions, and   their stability, of the following equation
\begin{equation} \label{eq_main}
\begin{cases}
  \partial_t   n  (t,x)= r B_{\alpha}   [ n ]  (t,x)
-    m (x) n (t,x) -\kappa \rho(t) n(t,x),\\
n (0,x)=n_{0}(x),\qquad \rho(t)=\int_\R n(t,x),
\end{cases}
\end{equation}
where $(t,x) \in (0,\infty) \times \mathbb R$ and
\[
B_{\alpha} [{n (t,\cdot)}](x) =\int_{\mathbb R} \int_{\mathbb R} \Gamma_{\alpha}  \left(x-\frac{(y+y')}{2}\right)  n (t,y) \f{n (t,y')}{\rho(t)}  dydy',
\]
\[
\Gamma_{\alpha} (x )= \f{1}{ \sqrt{\pi}\alpha} \exp\left(-\f{x^2}{\alpha^2} \right).
\]
This equation describes the evolutionary dynamics of the phenotypic density of a population subject to sexual reproduction, selection and competition. The unknown  $n(t,x)$ stands for the  density of  individuals of phenotypic trait $x$ at time   $t$. The function $m(x)$ represents the intrinsic mortality rate of individuals of trait $x$ and the nonlocal term $\kappa \rho$ corresponds to a mortality rate due to uniform  competition between individuals. The parameter $r$ scales   reproduction in the population and the   operator $B_{\alpha}$  models the sexual reproduction, assuming that the trait of the offsprings are distributed following a Gaussian profile with variance $\alpha >0$ centered around the mean parental traits. This reproduction model, which is known as the infinitesimal model, was introduced by Fisher in \cite{RF:19}  and it is widely used in 
the biological literature \cite{MB:80,Lynch-Walsh,MT:17}. Such a model is valid under the assumption that the trait $x$ is coded by infinitely many alleles with small additive effects (see \cite{NB.AE.AV:17} for a recent justification of such a model).

 Here, we are interested in a particular regime where the phenotypic variance induced by each reproductive event is small.  We characterize the steady solutions of \fer{eq_main} and analyze their stability. Our analysis combines the computation of the moments of the phenotypic distribution, following our previous work in \cite{JG.MH.SM:23}, with a spectral analysis using Hermite polynomials. With these new tools, we simplify considerably the previous approaches in \cite{VC.JG.FP:19,FP:23} and we extend their results on the asymptotic analysis of the infinitesimal model. We believe that our approach will facilitate the analysis of several open problems, which seemed out of reach with previous methods, as for instance the asymptotic analysis of models involving the infinitesimal model but accounting for  spatial or temporal heterogeneity of the environment, or the development of asymptotic preserving numerical schemes. 

\subsection{Biological motivation}

   The infinitesimal model has been widely used in evolutionary biology and in plant and animal breeding since the pioneer work of Fisher \cite{RF:19}. This work reconciled  Galton's observations \cite{FG:89} on the distribution and the inheritance of continuously varying phenotypes, as human's height, with Mendelian genetics.   This model is interested in  the traits that are coded by a large number of  genes with additive affects \cite{RF:19,MB:80,Lynch-Walsh}. A central limit theorem type result then implies that the trait of offsprings has a normal distribution centered around the mean parental traits. This property which can be referred to as the "Gaussian descendants" approximation \cite{MT:17}, was proved rigorously in \cite{NB.AE.AV:17}.  Many works in theoretical biology make however a stronger assumption that not only the offspring's distribution has a Gaussian profile, but also the population distribution is Gaussian \cite{MK.NB:97,MK.SM:14,RL.SS:96,OR.MK:01}. This can be referred to as the "Gaussian population" approximation \cite{MT:17}. This approximation also seems to provide robust results,  however its framework of validity is not yet completely understood \cite{MT:17}. Our work falls in line with studies that aim both to understand the validity framework of  such Gaussian approximations and to improve the accuracy of approximations in theoretical biology (see also \cite{VC.JG.FP:19,FP:23,GR:17} for other works in this direction). This article follows our earlier work in \cite{JG.MH.SM:23} where we provided an analysis of the time dependent problem \fer{eq_main}. The model considered in these articles considers a homogeneous environment with no space or time variation of the environment. We expect however that our methods would facilitate the analysis of more complex models with space or time heterogeneity.

\subsection{Expected shape of solutions} 
An underlying structure of \fer{eq_main} enables us to simplify the computation of solutions. Indeed, 
any solution $n(t,x)$ can be split into a mass $\rho(t) \in (0,\infty)$ and a probability density $x \mapsto q(t,x) \in (0,\infty)$:
\[
n(t,x)=\rho(t) q(t,x).
\]
Standard manipulations show that $n$ is a solution to \fer{eq_main} if and only if $(q,\rho)$
solve:
{
\begin{align}
\label{eq:qe}
&\begin{cases}
 \p_t q(t,x)=r\big( T_{\alpha}[q](t,x)-q(t,x)\big) -\big(m(  x)-\int_\R m(  y)q(t,y)dy\big) q(t,x),\\
q(0,x)=q_{0}(x):=n_{0}(x)/\rho(0),
\end{cases}
\\
\label{eq:rhoe}
&\begin{cases}
 \p_t \rho(t) = \rho(t) \left( r - \int_{\mathbb R} m(x) q(t,x) dx\right) - \kappa \rho(t)^2,\\
\rho(0) = \int_\R n(0,y)dy,
\end{cases}
\end{align}
}
with
\[
T_{\alpha} [{\overline q( \cdot)}](x) =\int_{\mathbb R} \int_{\mathbb R} \Gamma_{\alpha}  \left(x-\frac{(y+y')}{2}\right)\overline q (y)  \overline q (y')   dydy'.
\]
We point out that {system \fer{eq:qe}-\fer{eq:rhoe}} is "triangular" in the sense that the first equation involves the density unknown $q$ only and can thus be solved independently. The mass $\rho$ is then recovered by integrating the second equation. Elementary properties of the kernel $\Gamma_{\alpha}$ ensure that the property:
\[
\int_{\mathbb R} q(t,y) dy = 1
\]
is propagated when solving \eqref{eq:qe}.

\medskip

In this paper, we tackle existence/uniqueness properties of steady solutions $\overline{n}$ to \eqref{eq_main} that is, with the same factorization as above:
\beq
\label{eq:n-from-q}
\begin{cases}
\overline n(x)=\overline \rho \, \overline q(x),\qquad x\in \R,\\
\overline \rho=(r-\int_\R m(y)\overline q(y)dy)/\kappa,
\end{cases}
\eeq
where $\overline q$ solves the following equation 
\begin{equation} \label{eq_qepsform}
\begin{cases}
0=r\big( T_{\alpha}[\overline q](x)-\overline q(x)\big) -\big(m(  x)-\int_\R m(  y)\overline q(y)dy\big) \overline q(x),\\
 \int_\R \overline q(y)dy=1.
\end{cases}
\end{equation}
We also investigate the stability properties of these steady-states within the dynamical system \eqref{eq_main}. Obviously, the main issue here is to compute and analyze the probability-density part $\bar{q}$ in the factorization of $\overline{n}$.
 
\medskip
 
From the analytical standpoint, our problem is highly nonlinear and requires a subtle understanding of the interactions between the two operators at stake in the right-hand side of \eqref{eq_qepsform} : the (normalized) sexual reproduction (or recombination) operator $T_{\alpha}[q] -q$ and the (normalized) mortality operator  $(m - \int_{\mathbb R} mq)q$. Indeed, on the one-hand, it is by now well-documented that probability distributions cancelling the normalized recombination operator are gaussian distribution with variance $2\alpha$ (and arbitrary center) {(see e.g. \cite{GR:17})}. {
Moreover, considering the problem
\[
\begin{cases}
\p_t q(t,x)=T_\alpha[q]-q,\\
q(0,x)=q_0(x),\qquad \overline x_0=\int_\R xq_0(x)dx,
\end{cases}
\]
one can prove \cite{PM.GR:25}, using a contraction property of the Wasserstein distance, that, as $t\to+\infty$, $W_2(q(t,x),\Gamma_{\sqrt{2}\alpha}(x-\overline x_0))\to 0$.\\
 Let's consider now the evolution problem with the morality operator
\[
\begin{cases}
\partial_t q(t,x) =  -\left(m(x)-\int_\R m(  y) q(t,y)dy\right) q(x),\\
q(0,x)=q_0(x),
\end{cases}
\]
assuming that
\[
\arg\min m=\{x_1,\cdots, x_N\}, \qquad N\geq 1,\qquad \arg \min m\cap \mathrm{supp} \,q_0\neq \emptyset.
\]
Then, one can   prove that \cite{TL.CP:20}, as $t\to 0$, $q(t,x)\to \sum_{i=1}^n \alpha_i \delta(x-x_i)$, with $\alpha_i\geq 0$ and $\sum_{i=1}^N \alpha_i=1$, and    $\alpha_i=0$ for all $x_i\not\in \mathrm{supp} \,q_0$. \\
To summarize, the reproduction operator will tend to make the solution gaussian around one center and the mortality operator will tend to make the solution to concentrate around its minimum points. What would be the dynamics of the solution when we combine both operators? Let's assume that the segregational variance $\alpha$ is small and that $q_0$ is initially concentrated around a trait $x_0$ which is in the convexity zone around a local minimum point $x_m$ of $m$.  It was proved in \cite{FP:23,JG.MH.SM:23} that, when $x_m$ satisfies an admissibility condition ensuring that $m(x_m)$ is not too far from the global minimum value of $m$, the solution $q(t,x)$ remains concentrated for all times, with an approximately Gaussian distribution   centered around an evolving point $\bar x(t)$. Moreover, the point $\bar x(t)$   solves an ordinary differential equation indicating that it moves towards $x_m$ as $t$ grows. This indicates that, for any admissible local minimum point $x_m$, there should exist a steady solution, with an approximately Gaussian shape, concentrated around $x_m$. We prove indeed this property in this article and show that such steady solutions are stable.\\ Let us now consider the following gedankenexperiment. Let us imagine that $m$ is even and has a local maximum in $0$ and two (local) minimas $-x_1$ and $x_1$. For a symmetric initial data, the evolution problem \eqref{eq:qe} will yield a symmetric solution. If we expect this solution to tend to a steady state for large times, we face a novel difficulty. Indeed, the recombination operator will tend to make the solution gaussian around one center (and thus in $0$), while the mortality operator will want to make the solution localized around the two local minima of $m$ (and thus outside $0$). It is then interesting to discuss also what happens around local maximums of $m.$ We prove indeed that when the local maximum points of $m$ satisfy an admissibility condition, there also exists a steady solution concentrated around such maximum points. We don't expect however that these steady states would be stable (they can be reached only for some particular initial data).
}
 
\subsection{A dimensionless parametrization of the model}

In what remains of this paper, we fix $x_0$ a local extremum for $m.$ As we may expect from the previous discussion, the respective flatness of $m$ around $x_0$  (encoded by the value of $|m''(x_0)|$) and variance $\alpha$ of the Gaussian kernel involved in the recombination operator will have a decisive impact on the properties of steady-states. To fix ideas, we propose to non-dimensionalize the equations and fix our framework with one well-chosen parameter $\varepsilon$ encoding the respective amplitude of both quantities. Let $q$ be a solution to \eqref{eq:qe} and set
\[
\tilde{q}(\tau,y) ={ \sqrt{\dfrac{r}{|m''(x_0)|}}}q\left(\dfrac{\tau}{r},\sqrt{\dfrac{r}{|m''(x_0)|}}y+x_0\right).
\]
We obtain that $\tilde{q}$ solves
\[
 \p_{\tau} \tilde{q}(\tau,y)= \left( T_{\e}[\tilde{q}](\tau,y)-\tilde{q}(\tau,y)\right) -\left(\tilde{m}( y)-\int_\R \tilde{m}(z)\tilde{q}(\tau,z)dz\right) \tilde{q}(\tau,y),\\
\]
where:
\[
\e = \alpha \sqrt{\dfrac{|m''(x_0)|}{r}} \qquad 
\tilde{m}(y) = \dfrac{1}{r} m \left( \sqrt{\dfrac{r}{|m''(x_0)|}} y + x_0\right) - {\frac{1}{r}}m(x_0).
\]
With this change of unknown, we have shifted the local extremum of $\tilde{m}$ in $x=0$ and made it $\tilde{m}(0)= 0.$ Consequently, $\tilde m$ is not necessarily nonnegative. We have also normalized the local expansion around $0$ since
\[
\tilde{m}(y) = {\pm}\dfrac{y^2}{2} + lot.
\]
Below, we will consider the natural framework in which $\varepsilon << 1$ that is
\beq
\alpha |m''(0)|^2 << r,
\eeq
so that $m$ is nearly constant on the support of steady solutions for the recombination operator.  This amounts to study system  \eqref{eq:qe} or its stationary version \eqref{eq_qepsform} with $r=1,$ $\alpha = \varepsilon,$
and $|m''(0)|=1.$ We will thus analyze this case in the paper, {replacing $\tilde m$ by $m$ in what follows}. Other important assumptions for the analysis will be made precise and discussed below.

{ To {end} this section, we recall that we focus in what follows on the denstiy equation \eqref{eq:qe} (resp. \eqref{eq_qepsform}). But the full problem aslo involves the mass equation \eqref{eq:rhoe} (resp. \eqref{eq:n-from-q}). A possible corresponding non-dimensionalization for $\rho$ could be to write :
\[
\tilde{\rho}(\tau) = \dfrac{\kappa}{r} \rho \left(\dfrac{\tau}{r}\right).
\]
With this particular choice we transform equation \eqref{eq:rhoe} (resp. \eqref{eq:n-from-q}) into a similar equation with a new parameter  
\[
\tilde{r} =  1 - \dfrac{m(x_0)}{r} 
\]
replacing $r,$ the chosen $\tilde{m}$ and $\kappa$ replaced with $1.$
}

 \subsection{Assumptions and notations}
 
We will consider two sets of assumptions. The first set of assumptions is weaker and is designed to study solutions which are close to the concentrated steady solutions and allow several local extrema of $m$. 
It reads 
\beq
\label{as:m-extremum}\tag{H1}
m'(0)=0,\qquad m(0)=0<\min_\R m(x)+1,\qquad m''(0)\in \{\pm 1\},
\eeq
\beq
\label{as:D3m-polynom}
\tag{H2}
 |m^{'''}(x)| \leq A_m  (1+|x|^{p}), \quad \text{$\forall \,x\in \R,$ for some $p \in \mathbb N^*$ and $A_m \in (0,\infty).$}
\eeq
{The second condition in \fer{as:m-extremum} is an admissibility condition on the extremum point which requires that the value of $m$ at the extremum point $x_m$ is not too far from its global minimum value. We prove in Appendix \ref{sec:nonexistence} that if this condition is not satisfied, then there is no   concentrated steady solution around such a point.}
{ We will use several times that \eqref{as:D3m-polynom} with the condition on $m^{''}(0)$ appearing in \eqref{as:m-extremum} entail that 
\beq
\label{as:D2m-polynom}
\tag{H2'}
 |m^{''}(x)| \leq 2 {A}_m  (1+|x|^{p+1}), \quad \text{$\forall \,x\in \R.$}
\eeq
}

The second set of assumptions enforces that we have a unique admissible minimizer of $m$ and is designed to justify that any steady solution to \fer{eq_qepsform} is concentrated. Under these assumptions we will be able to achieve the uniqueness result in a wider class of solutions.  It reads
\beq
\label{As:m}\tag{H3}
 \text{$m(x)\geq m(0)=0$,  and the value of any extremum point of $m$ is greater than $1  $}.
\eeq
Notice that the assumption \fer{As:m} implies that for any $v\in (0, 1]$, there exists $a_+> 0$ (possibly equal to $+\infty$) and 
$a_-< 0$ (possibly equal to $-\infty$) such that  
 \beq
 \label{monotony-m}
 \tag{H3'}
 m(x)> v,\quad \text{for all $x\in (-\infty,a_-)\cup(a_+,+\infty) $ and }\quad  m (z)\leq v ,\quad \text{for all $z\in [a_-,a_+)$.}
\eeq
We  define then
$$
x_-=\inf \{x<0 \, | \,\forall y\in (x,0),\; m(y) \leq  1\},
$$
$$
x_+=\sup\{x>0 \, | \,\forall y\in (0,x),\; m(y) \leq  1\},
$$
and    we  assume also that there exists positive constant  $c_m,C_m$   such that
 \beq
 \label{As:m2}
 \tag{H4}
 \left\{
 \begin{aligned}
 & |m''(x)|\leq C_m ,&& \forall \, x\in \R \text{ and }\quad m''(0)=1, \\
 & m'(x)\neq 0, && \forall \,x\in(x_-,0)\cup(0,x_+), \\
 & m(x) \geq c_m x^2 && \forall \, x \in \R.
 \end{aligned}
 \right.
 \eeq

 Since Gaussian distributions have a central role in our results,  we define
\[
 G_\e(x) =\Gamma_{\sqrt{2}\,\e}= \dfrac{1}{\sqrt{2\pi}\e}\exp(-\f{x^2}{2\e^2}).
\]
An important part of our analysis will involve the probability density $G_1$:
\[
 G(x) :=G_1(x)= \dfrac{1}{\sqrt{2\pi}}\exp(-x^2/2).
\]
Let also $\sigma_k$, for all $k\geq 0$,  be the $k$-th order  moment of $G(x)$, that is
\begin{equation} \label{eq_sigmak}
\int_\R x^k G(x)=\sigma_k.
\end{equation}
Considering $q_\e\in L^2(\R,G_\e(x)dx)$ we define
\beq
\label{def-moments}
 M_{\e,0}=\int_\R q_\e(y)dy=1,\qquad  M_{\e,1}=\int_\R yq_\e(y)dy,\qquad   M_{\e,k} =\int_R (y-M_{\e,1})^k q_\e(y)dy.
\eeq

\subsection{Main results}

We first consider the case where $0$ is a global minimum point of $m$, with no other extremum point satisfying the admissibility condition \eqref{as:m-extremum}. We   then prove that any steady solution to \fer{eq_qepsform} is necessarily concentrated around the point $0$, {\em i.e.} it has small central moments. More precisely, we have the following result.

\begin{theorem}
\label{thm:concentration}
Assume   \fer{As:m} and\fer{As:m2} and fix $\delta \in (0,1)$. 
Let $\e >0$ and $\overline q_\e \in L^1(\mathbb R,(1+x^2)dx)$  be a steady solution to \fer{eq_qepsform} satisfying: \\
 \beq
 \label{As:q}
 \tag{H5}
 \int_\R m(x)\overline{q}_\e(x)dx<1-\delta.
 \eeq
  (i)  We have $\overline{q}_{\e} \in L^1((1+x^{2})^{\ell}dx)$ for all $\ell \in \mathbb N$ and there  exists positive constants $C_k$, for $k=1,2,...$, which may depend on $\delta$ but not on $\e$ such that for $\e\leq \e_0$ small enough, we have
\beq
\label{est-Mek}
  M_{\e,0}=1,\qquad |  M_{\e,1}|\leq C_1\e^2,\qquad |  M_{\e,k}^c-\e^k  \sigma_k|\leq C_k\e^{2+k}  .
\eeq
(ii) Assume additionally that for some positive constants {$\delta'$} and $A_1$ independent of $\e$,
\beq \tag{H6}
\label{cond-q/G-bound}
\int_{\mathbb R} \overline q_\e(y) e^{\f{{\delta'} y^2}{\e^2}}dy\leq A_1.
\eeq
Then, there exists a constant $C$, independent of $\e$, such that for $\e$ small enough,
\beq
\label{est-qe-Ge-L2}
\left\| \overline q_\e(\cdot)-G_\e(\cdot) (1+\f{  M_{\e,1}}{\e^2}x)\right\|_{L^2(\R,G_\e^{-1}(x)dx)}\leq C\e^2.
\eeq
\end{theorem}

We point out that \eqref{As:q} is not a very restrictive condition. One can indeed verify  by evaluating  \fer{eq_qepsform} at $x=0$ and using the positivity of $T_\e[q]$ that
\beq
\label{bound-int-mq}
\int_\R m(y)q_\e(y)dy< 1.
\eeq
{This assumption is also biologically relevant since it ensures that the population size is positive. The second equality in \fer{eq:n-from-q}  indeed indicates  that this is a necessary assumption for $\overline \rho$ to be positive. }

A key ingredient in the proof of this theorem is the control of the variance $M_{\e,2}$ of the steady distribution. The proof of this property relies strongly on the assumption that $0$ is the only admissible extremum point of $m$ satisfying \eqref{as:m-extremum}. The control of the other central moments then follows using the equations satisfied by these moments which involve a dissipative term, and using the fact that the terms involving the selection term $m(x)$ have small contributions. The proof of \eqref{est-qe-Ge-L2} relies on a spectral decomposition of the solution in terms of Hermite polynomials, which yield an orthonormal basis of $L^2(\mathbb R,G(x){\rm d}x)$.

We then study the general case where $m$ might have several admissible extremum points satisfying \fer{as:m-extremum}. We prove that there exists a unique concentrated steady solution around any admissible extremum point, satisfying a similar property to \fer{est-qe-Ge-L2}.
 \begin{theorem}
 \label{thm:local-steady}
Assume \eqref{as:m-extremum}-\eqref{as:D3m-polynom}. Then, there exists a   constant {$\overline C$} such that for all $C\geq {\overline C}$ and for all $\e \leq \e_0(C)$ small enough, there exists a  steady solution $ \overline q_\e\in L^2(\R,G_\e^{-1}(x)dx)$ to \fer{eq_qepsform} which satisfies
\beq
\label{estimate-qG}
\left\| \overline q_\e(\cdot)-G_\e(\cdot)\big(1+\f{M_{\e,1}}{\e^2}y\big) \right\|_{L^2(\R,G_\e^{-1}(x)dx)}\leq C\e^2,\qquad  | M_{\e,1}|\leq C\e^2.
\eeq
Moreover, {the steady solution satisfying the above conditions is unique}. {Such a steady solution is positive if $0$ is a minimum point of $m$, or if $m$ is even and   $0$ is a maximum point of $m$.}
 \end{theorem}
\begin{remark}
Notice that, combining the result above with Theorem \ref{thm:concentration}, we obtain that under assumptions  \fer{As:m}--\fer{As:m2}, for $\e$ small enough, the steady solution is unique in the set of functions satisfying condition \fer{cond-q/G-bound}. Moreover, we will show later in the article that condition \eqref{estimate-qG} implies that, for all $k\geq 2$, there exists a constant $C_k$ such that
\[
 |  M_{\e,k}^c-C_k\e^k|\leq C_k\e^{k+2}.
 \]
\end{remark}

Theorem \ref{thm:local-steady} implies that, when there exist several admissible extremum points of $m$, then there are several steady solutions to problem  \fer{eq_qepsform}. One could wonder whether   the steady solution is unique when there is only one admissible extremum point. 
In this respect, Theorem \ref{thm:concentration} implies that any steady solution is concentrated with small central moments. Although these properties significantly narrow the range of potential steady solutions, they are still insufficient to establish a  uniqueness result in the most general space, for instance the larger space $L^1((1+|m(x)|)dx,\R)$ that guarantees all terms in \eqref{eq_qepsform} are well defined.

 One can also wonder whether there might exist non-concentrated steady solutions to problem \fer{eq_qepsform}. Theorem \ref{thm:concentration} implies that this is not the case when $0$ is the only admissible extremum point of $m$. However, we believe that when $m$ has several extremum points there might exist non-concentrated steady solutions. Again, let's consider a situation where $m$ is even and it has one local maximum at $0$ {and} two minimum points at $-x_1$ and $x_1$. We also assume that $m(0)=0>m(x_1)+1$ and hence $0$ is not an admissible extremum point and there does not exist any concentrated steady solution at $0$ (see Appendix \ref{sec:nonexistence}). Therefore there exist only two concentrated steady solutions which are non-symmetric, being concentrated around one of the minimum points,  $-x_1$ or $x_1$.  We next assume that the initial condition $q_{\e,0}$ in  \fer{eq_qepsform} is even. Since $m$ is even, the problem \fer{eq_qepsform} preserves the symmetry of the solution and hence the solution remains symmetric for all times. If we expect that in long time $q_\e$ converges to a steady solution, then this steady solution has to be symmetric and hence non-concentrated. We expect however that such a steady solution would be unstable.

The proof of this result also relies on a spectral decomposition of the problem using Hermite polynomials and the fact that the terms involving the selection term $m$ have small contributions. This decomposition allows us to reduce considerably the nonlinearity of the problem. We then write the problem as the sum of a linear operator, with bounded inverse, and a nonlinear operator, satisfying a contraction property. This allows us to prove the existence of a steady solution satisfying the properties above. The positivity of such a steady solution then follows from the stability result provided in the next theorem:
 \begin{theorem} \label{thm_longtime_q}
Assume \eqref{as:m-extremum}-\eqref{as:D3m-polynom}. Then we have two alternatives:\\
(i) {Assume that  $m^{''}(0) >0$. Then, for any $C$ and $\varepsilon\leq \e_0(C)$ sufficiently small, any $q_0$ satisfying} 
\beq
\label{cond-q0}
\left\|  q_0(x)-G_\e(x)\big(1+\f{M_1}{\e^2}x\big) \right\|_{L^2(\R,G_\e^{-1}(x)dx)}\leq C\e^2,\quad   M_{1} =\int {x}q_0(x)dx, \quad |M_1| \leq C\e^2,
\eeq
  yields a unique global solution  $q_{\varepsilon} \in C([0,\infty);L^2(\mathbb R;G_{\e}^{-1}(x){dx}))$  to \eqref{eq:qe}  which satisfies, for some constants $A_1$ and $A_2$,
  \beq
  \label{eq-exp-conv}
 \left\| q_{\varepsilon}(t,\cdot) -  \overline q_{\varepsilon}(\cdot)  \right\|_{L^2(\mathbb R,G_\e^{-1}(x){\rm d}x)}\leq A_1\exp(-\e^2A_2 t) \left\| q_{\varepsilon,0}(\cdot) -  \overline q_{\varepsilon}(\cdot)  \right\|_{L^2(\mathbb R,G_{\e}^{-1}(x){\rm d}x)}.
\eeq
(ii)   {Assume that $m^{''}(0) <0$ and that $m$ is even. Then, for any $C>0$ and $\varepsilon\leq \e_0(C)$ sufficiently small, any even  initial data $q_0 \in L^2(\mathbb R;G_\e^{-1}(x){\rm dx})$ satisfying \eqref{cond-q0},} yields a unique global solution  $q_{\varepsilon} \in C([0,\infty);L^2(\mathbb R;G^{-1}_{\e}(x){dx}))$ to \eqref{eq:qe}  which satisfies, for some constants $A_1$ and $A_2$,
 \[
 \left\| q_{\varepsilon}(t,\cdot) -  \overline q_{\varepsilon}(\cdot)  \right\|_{L^2(\mathbb R,G_\e^{-1}(x){\rm d}x)}\leq A_1\exp(-A_2 t) \left\| q_{\varepsilon,0}(\cdot) -  \overline q_{\varepsilon}(\cdot)  \right\|_{L^2(\mathbb R,G_{\e}^{-1}(x){\rm d}x)}.
\]
 \end{theorem} 
 
 Notice that the exponential convergence in case (ii) is faster than case (i). The $\e^2$ in front of $A_2$ appears in \fer{eq-exp-conv} because of the coefficient $\alpha_1$ which converges slowly in the general case. If we assume that $m$ is even also in the stable case, with an even initial condition, then the coefficient $\alpha_1$ will remain equal to zero for all times and we will have a similar rate of exponential convergence as in case (ii).

We also point out that in Theorem \ref{thm:local-steady}  we obtain existence of a steady solution whatever the sign of $m^{''}(0)$,  we do not consider necessarily a local minimum. However, since the positivity result relies on the stability result above, we obtain the positivity of the steady solutions which are concentrated around maximum points, only when $m$ is even. Indeed, when $m$ is even we can still show that the steady solution is    stable when considering only symmetric initial conditions. Since the operator keeps the positivity of the solution, then starting from positive and symmetric initial conditions, we obtain that the solution converges to a positive steady solution. We believe that when $m$ is not even, there would still exist some positive initial conditions which would yield long time convergence to the steady solution and hence implying its positivity. However, it is harder in the general case to identify the appropriate initial conditions with this property.

   \subsection{State of the art}

 Several works in the mathematical community have already addressed the question of justifying the "Gaussian population" approximation. Our work is closely related to \cite{VC.JG.FP:19,FP:23} where a similar model with a similar scaling was studied using a perturbative analysis based on the Hopf-Cole transformation.  The method in  \cite{VC.JG.FP:19,FP:23} was inspired by the analysis of models with asexual reproduction leading to Hamilton-Jacobi equations \cite{OD.PJ.SM.BP:05,GB.BP:08,GB.SM.BP:09}. \cite{VC.JG.FP:19} provided an approximation of locally concentrated steady states of the problem and \cite{FP:23} provided an asymptotic analysis of the time dependent problem, characterizing solutions, with almost Gaussian shapes, which are concentrated around an evolving dominant trait. 
  Our results  go beyond the work in \cite{VC.JG.FP:19} in two ways. Firstly, we introduce a new method based on the analysis of moments following  \cite{JG.MH.SM:23} and a new Hermite expansion of the solution. We believe that the methods introduced in \cite{JG.MH.SM:23} and this article will facilitate the analysis of more complex models with spatial or temporal heterogeneity. While \cite{VC.JG.FP:19,FP:23} provide very interesting results on the homogenous problem and a robust method to understand the behavior of solutions to more complex models with heterogeneity \cite{LD:20,JG.OC.EB:23}, the rigorous application of the perturbative approach to more complex models seems hard, if not out of reach. Secondly, we   provide a local stability result for the concentrated steady solutions. Moreover, in a particular framework where we expect to have a unique steady solution, we prove that any steady solution is necessarily concentrated and provide a uniqueness result in a certain class of solutions.   
 
 More precisely, our results above  provide the existence of steady solutions close to a Gaussian distribution with a multiplicative correction of type $G_\e(1+\f{M_1}{\e^2}x)$ in the space of $L^2(\mathbb R,G_\e^{-1}(x){\rm d}x)$. The work in \cite{VC.JG.FP:19} also provides a multiplicative correction from the Gaussian distribution. However, their method relies on a Hopf-Cole transformation. They look  for approximations of type $e^{-\f{(x-\bar x)^2+\e^2v(x)}{2\e^2}}$, with an unknown $v\in C^3(\R)$  such that $(1+|z|)^\alpha |D^k v(z)|\in L^{\infty}(\R)$, for $k=1,2,3$. They obtain then the existence of concentrated steady solutions using a perturbative analysis in this space.
Theorem \ref{thm:concentration} implies however that the space of  $L^2(\mathbb R,G_\e^{-1}(x){\rm d}x)$ is also a natural candidate for  the study of the   solutions. Working with such a space allows us to reduce the nonlinearity of the problem and leads to a more direct proof.

 A closely related model with a spatial heterogeneity and  with a different scaling was studied in \cite{GR:17} (see also \cite{SM.GR:13,PM.GR:25}) using the analysis of the moments and a contraction property of the reproduction operator in the Wasserstein distance. The method in  \cite{JG.MH.SM:23} extended this method {to the analysis of the time-dependent version of the model   under consideration in this paper}. Yet,   the analysis of the solutions in $L^2(\mathbb R,G_\e^{-1}(x){\rm d}x)$ in the present article is completely new.

 Considering a model  involving an infinitesimal model and a quadratic selection function but considering discrete time, in \cite{VC.TL.DP:23} the authors proved the uniqueness and the stability of the steady solution. Later in  \cite{VC.DP.FS:23}, the authors studied again a model with discrete time but with a strongly convex selection function. Using a contraction property of a certain Fisher information, they prove the uniqueness of the steady solution, under a decay assumption on the tails of the solution and they prove the stability of such a steady solution under some condition on the initial condition.  
 
 Other models involving a sexual reproduction have been studied in \cite{NF.BP:21,BP.MS.CT:22,LD.SM:22}. In \cite{BP.MS.CT:22} the authors provided an asymptotic analysis on a model with asymmetric reproduction term, where  the trait is mostly inherited from the female. \cite{LD.SM:22} studied a model where the trait is coded by quantitative alleles at two loci in a haploid sexually reproducing population. \cite{NF.BP:21} provided a non-expanding transport distance that allowed them to study a model with sexual reproduction but considering constant birth and death rates.
 
 Notice that the infinitesimal model is also related to models within kinetic theory involving a collision operator (see for instance   \cite{PD.AF.GR:14} on an alignment model). A specificity of the models in evolutionary biology is that the collision operator is combined with a multiplicative operator which is not very common in the kinetic theory. 
 
 Hermite polynomials have been introduced in the kinetic
theory in \cite{HG:49}, leading to various consequences in both numerical and theoretical aspects. They have been for instance used to study the solution of the Boltzmann equation (see e.g. \cite[Section 8]{AB:88}). They have also been used to provide numerical schemes for different equations in kinetic theory  (see e.g. \cite{MB.FF:22,MB.FF:23}). Our work is more related to \cite{FF.FG:25} where a multiplicative operator is also involved in the model. To the best of our
knowledge, Hermite polynomials have not been used previously in the study of the infinitesimal model \fer{eq_main}.

\subsection{Plan of the article}
The outline of the paper is as follows. In the next section, we introduce the Hermite polynomials and highlight their relevance for our analysis. We rephrase our existence/stability result ({\bf Theorem \ref{thm:local-steady}} and {\bf Theorem \ref{thm_longtime_q}}) in this setting (see {\bf Theorem \ref{thm_stat}} and {\bf Theorem \ref{thm_longtime}}) and conclude with some properties of Hermite coefficients of the selection rate $m.$ 

We split what remains of the paper in terms of the underlying mathematical appraoch. The {\bf Section \ref{sec:conc1}} and the {\bf Section \ref{sec:conc2}} contain proofs for {\bf Theorem \ref{thm:concentration}} combining mainly moments arguments possibly linked with their interpretation in terms of Hermite coefficients. The Hilbert structure of the $L^2(\mathbb R,G(x) dx)$ space is deeply used in the two last sections that contain proofs for {\bf Theorem \ref{thm_stat}} and {\bf Theorem \ref{thm_longtime}}. We postpone technical remarks to the appendix as well as a complementary computation illustrating the optimality of assumption \ref{as:m-extremum}.

\section{{The Hermite-polynomials framework}}
\label{sec:Hermite}

In some parts of the article, it will be more convenient to work with $T_1$ instead of $T_\e$. We will hence use the following change of variables
\[
N_\e(t,x)= \e q_\e(t,\e x), \qquad \overline N_\e(x)=\e \overline q_\e(\e x),
\]
which leads  to the following problems  written in terms of $T_1$:
\begin{equation} \label{eq_Nepsform}
\begin{cases}
\p_t N_\e(t,x)=\big( T_1[N_\e](t,x)-N_\e(t,x)\big) -\big(m(\e x)-\int_\R m(\e y)N_\e(t,y)dy\big) N_\e(t,x),\\
N_\e (0,x)=N_{\e,0}(x)=\e q_{\e,0}(\e x),\qquad \int_\R N_\e(0,y)dy=1.
\end{cases}
\end{equation}
\begin{equation} \label{eq_Nepsform_steady}
\begin{cases}
0=\big( T_1[\overline N_\e](x)-\overline N_\e(x)\big) -\big(m(\e x)-\int_\R m(\e y)\overline N_\e(y)dy\big) \overline N_\e(x),\\
  \int_\R \overline N_\e(y)dy=1.
\end{cases}
\end{equation}
The probability density  $G$  then plays a crucial role in our analysis since it satisfies $T_1[G]=G$
and it is centered in the local minimum of $m.$ As a consequence, straightforward computations entail that:
\[
T_1[\partial_x^k G,\partial_x^l G] = \dfrac{1}{2^{k+l}} \partial_x^{k+l} G
\] 
for arbitrary $(k,l) \in \mathbb N^2.$ This is a first cornerstone that indicate a good way of stating our system is to look for solutions to \eqref{eq_Nepsform} in the form $N_{\varepsilon} = hG.$ Indeed, we recall that the family 
\[
H_k(x) = \dfrac{(-1)^k}{\sqrt{k!}} \dfrac{\partial_x^{k} G(x)}{G(x)}\,, \quad \forall \, k \in \mathbb N 
\]
yields an orthonormal basis of $L^2(\mathbb R,G(x){\rm d}x)$.
The above computations indicate a good rephrasing of the bilinear mapping $T_1$ is: 
\[
\tilde{T}_1[H_{k},H_{\ell}] = G^{-1} T_1[H_{k}G,H_{\ell}G].
\]
Indeed, we have that:
\beq
\label{tildeT1-Hermite}
\tilde{T}_1[H_k,H_{\ell}] =   \dfrac{\sqrt{(k+l)!} }{2^{k+l}} \dfrac{H_{k+l}}{\sqrt{k!}\sqrt{l!}}.
\eeq
 for arbitrary $(k,\ell) \in \mathbb N^2.$  
We can then extend the formula by bilinearity and continuity {(see Appendix \ref{sec:appendixT})} to define a bilinear continuous mappping $\tilde{T}_1 : L^2(\mathbb R,G(x)dx)^2 \to L^2(\mathbb R,G(x)dx)$ with the formula:
\[
\tilde{T}_1[h_1,h_2] = G^{-1} T_1[h_1G,h_2G] \qquad \forall \, (h_1,h_2) \in (L^2(\mathbb R,G(x)dx))^2.
\]

{Let $\overline N_\e=\sum_{k=0}^\infty \alpha_k H_kG$, with $\overline N_\e$ solution of \fer{eq_Nepsform_steady}. One can notice that, since $H_0=1$, $\alpha_0=\int_\R \overline N_\e(x)dx=1$. }
{Moreover, the problem \fer{eq_Nepsform_steady}  reads as follows in terms of the Hermite coefficients $(\alpha_k)_{k \in \mathbb N}$:}
\begin{multline}
\label{eq-alpha}
\left( \dfrac{1}{2^{k-1}}  - 1  +(H_0,  m_{\varepsilon} H_0)_G   \right) \alpha_k 
 - \left( \sum_{l=2}^{\infty} \alpha_l H_l ,  m_{\varepsilon}  H_k\right)_G  \\
= m_{k}^{(\varepsilon)}- \left[ \dfrac{\sqrt{k!}}{2^k} \sum_{l=1}^{k-1} \dfrac{\alpha_l \alpha_{k-l}}{\sqrt{l! (k-l)!}}  + \sum_{l=1}^{\infty} \alpha_k  \alpha_l  m_l^{(\varepsilon)}\right]+ \left(   H_1 ,  m_{\varepsilon}  H_k\right)_G \alpha_1 ,
 \end{multline}
  for all $k \geq 1,$ with $m_\e(x)=m(\e x)$ and $m_k^{(\varepsilon)} := (m_{\varepsilon}, H_k)_G$. 

\medskip

Below, we will use the Hilbert structure of $L^2(\mathbb R,G(x)dx)$ endowed with the canonical 
scalar product $(\cdot,\cdot)_G.$  We use the identification with $\ell^2(\mathbb N_0)$  via the decomposition on the orthonormal basis $(H_k)_{k\in  \mathbb N_0}.$
In technical parts, we will have to treat   separately indices $k=0$ and $k=1$. Hence, we introduce $\overline{\alpha}_k = (\alpha_l)_{{l \geq k}}$  when $k=1,2.$  Correspondingly, we shall write $l \in \mathbb N_1$ (or $l \in \mathbb N$) for $l \geq 1$, respectively $l \in \mathbb N_2$ for $l\geq 2.$ We keep the classical $\mathbb N_0$ to include the origin $k=0.$

\subsection{Hermite-polynomial expansions vs moments analysis.}

  In summary, we have three formulations for the analysis of \eqref{eq:qe} and its stationary variant \eqref{eq_qepsform}. Firstly, we can work with the unknown $q_{\e}$ solution to the nondimensional system. Below, we denote $M_{\e,k}$ the moments and $M^{c}_{\e,k}$ the centered moments of $q_{\e}:$ 
\[
M_{\e,k} = \int_{\mathbb R} x^{k}q_{\e}(x)dx, \quad M_{k}^{c} = \int_{\mathbb R} (x-M_1)^{k} q_{\e}(x) dx \quad 
\forall \, k \in \mathbb N. 
\]
Secondly, we have the rescaled unknown $N_{\e}.$  Its moments are denoted with tildas:
\[
\widetilde{M}_{\e,k} = \int_{\mathbb R} y^{k}N(x)dx, \quad \widetilde{M}_{\e,k}^{c} = \int_{\mathbb R} (y-\widetilde{M}_{\e,1})^{k} N(y) dy \quad 
\forall \, k \in \mathbb N. 
\]
Notice that the moments of $N_{\e}$ relate to the moments of $q_\e$ in the following way
 \beq
 \label{link-M-tildeM}
  M_{\e,1}(t)=\e \widetilde M_{\e,1}(t),\qquad  M_{\e,k}^c(t)=\e^k \widetilde M_{\e,k}^c(t), \quad \text{for all  $k\in \mathbb{N}_2$}.
\eeq

Finally, we have the multiplier pertubation of the Gaussian $h$ of $N_{\e}$ so that $N_{\e} = hG.$ This latter unknown belongs to $L^2(\mathbb R,G(x)dx)$ and will be expanded in terms of Hermite polynomials.

The Hermite coefficients $(\alpha_k)_{k\in \mathbb N_0}$   of the multiplier perturbation $h$  are related to the central moments of its associated solution $N_{\e}$. Indeed,  we recall that $H_0=1$, $H_1=x$ and hence
\[
\alpha_0 = \int_{\mathbb R} {h(x)}G(x){\rm d}x =\widetilde M_0= 1, \qquad
\alpha_1 = \int_{\mathbb R} x {h(x)}G(x){\rm d}x=\widetilde M_1.
\]
Higher coefficients measure the distance between $N_{\e}$ to the Gaussian profile in a certain sense given in Lemma \ref{link-Hermite-moments} below.  In this statement $\sigma_{k}$ refers {to} Gaussian moments \eqref{eq_sigmak}. 

\begin{lemma}
  Let $h \in L^2(\mathbb R,G(x) dx)$ and $N = hG.$
 \label{link-Hermite-moments}
 We have,  with obvious notations to denote moments:
\[
 \alpha_k=\f{1}{\sqrt{k!}}\big(\widetilde M_{k}^c-\sigma_k-R(k)\big), \quad \forall \, k \geq 2,
\]
 with 
\[
 \begin{array}{rl}
  R(k)&=\sum_{l=0}^{k-1} \begin{pmatrix} k \\ l  \end{pmatrix}\widetilde M_{1}^{k-l}(-1)^{k-l}\sum_{j=0}^l\f{l!}{(l-j)!\sqrt{j!}}\alpha_j \sigma_{l-j}\\
  &+\sum_{j=1}^{k-1} \f{k!}{(k-j)!\sqrt{j!}}\alpha_j\sigma_{k-j}
  .
  \end{array}
\]
 \end{lemma}
 
\begin{proof}
 Given $k \geq 2,$  we compute
 $$
\begin{aligned}
 \widetilde M_{k}^c&=\int_\R (x-\widetilde M_{1})^kN(x)dx\\
 &=\sum_{l=0}^k \widetilde M_{1}^{k-l}(-1)^{k-l} \begin{pmatrix} k \\ l  \end{pmatrix}  \int_\R   x^lN(t,x)dx\\
 &=\sum_{l=0}^k \widetilde M_{1}^{k-l}(-1)^{k-l}\begin{pmatrix} k \\ l  \end{pmatrix}  \sum_{j=0}^{+\infty} \alpha_j\int_\R x^l H_j(x) G(x)dx.
\end{aligned}
 $$
 We next notice that,   whatever the value of $l,j$,
 $$
\begin{aligned}
 \int_\R x^l H_j(x) G(x)dx&=\int_\R \f{x^l (-1)^j}{\sqrt{j!} }\p_x^jG(x)dx
  \\
  &=\mathds{1}_{j\leq l}\int_\R \f{l(l-1)\dots (l-j+1)}{\sqrt{j!} }x^{l-j}G(x)dx\\
  &=\mathds{1}_{j\leq l}\f{l!}{(l-j)!\sqrt{j!}}\sigma_{l-j}.
 \end{aligned}
  $$
 We deduce that
  $$
 \begin{aligned}
 \widetilde M_{k}^c= &\sum_{l=0}^k \widetilde M_{1}^{k-l}(-1)^{k-l}\begin{pmatrix} k \\ l  \end{pmatrix}  \sum_{j=0}^{l} \f{l!}{(l-j)!\sqrt{j!}}\alpha_j\sigma_{l-j}\\
 =&\sigma_k+\sqrt{k!} \,\alpha_k +\sum_{l=0}^{k-1}\widetilde M_{1}^{k-l}(-1)^{k-l}\begin{pmatrix} k \\ l  \end{pmatrix}  \sum_{j=0}^{l} \f{l!}{(l-j)!\sqrt{j!}}\alpha_j\sigma_{l-j}\\
 &+ \sum_{j=1}^{k-1} \f{k!}{(k-j)!\sqrt{j!}}\alpha_j\sigma_{k-j}.
 \end{aligned}
 $$
  This concludes the proof.
 \qed 
 \end{proof}

 \subsection{Rephrasing our main results with Hermite-polynomial expansions.}
 
 In order to obtain our  results on the local existence,  uniqueness and stability of a steady state, it will be convenient to work with the formulation of the problem given in 
  \eqref{eq_Nepsform}. Such local results will be given working with  a certain set, defined as follows, for given $\varepsilon >0$ and $C >0$,
\[
\mathcal U^{(\varepsilon)}(C) := \left\{ N = \sum_{k=0}^{\infty} \alpha_k H_k G   \text{ s.t. } \alpha_0=1 \quad |\alpha_1| \leq C  \varepsilon 
\quad 
 \| \overline{\alpha}_2 \|_{\ell^2(\mathbb N_2)} \leq  C \varepsilon^2
\right\}.
\]
This set represents a small neighborhood of  $G$ in $L^2(\R,G^{-1}(x)dx)$. In particular  the gap between the central moments of the elements of $\mathcal U^{(\varepsilon)}(C)$ and the central moments of the Gaussian $G$ is small. One can  indeed verify using Lemma \ref{link-Hermite-moments} that for any $N\in \mathcal U^{(\e)}(C)$, the central moments of $N$ satisfy, for some positive constants $C_k$,
\[
 \widetilde M_0=1,\quad |\widetilde M_1|\leq C \e ,\qquad |\widetilde M_{k}^c-\sigma_k |\leq C_k\e^2.
\]
Moreover, one can   verify that  $ \overline N_\e\in \mathcal U^{(\e)}(C)$, is equivalent with, ${\overline q_\e(x)=\frac{1}{\e}\overline N_\e(\frac{ x}{\e})} \in  L^2(\R,G^{-1}_\e(x)dx)$ and 
\[
\left\| \overline q_\e(y)-G_\e(y)\left(1+\f{  M_{\e,1}}{\e^2}y\right) \right\|_{L^2(\R,G_\e^{-1}(x)dx)}\leq C\e^2,\qquad   M_{\e,0}=1,\qquad |   M_{\e,1}|\leq C\e^2.
\]
%

  We obtain existence and uniqueness of  stationary solutions to \fer{eq_Nepsform_steady} in the set ${\mathcal U}^{(\e)}(C)$.
\begin{theorem} \label{thm_stat}
Assume \eqref{as:m-extremum}-\eqref{as:D3m-polynom}. There exists  ${\overline{C}} >0$ such that, for all $C>{\overline{C}}$ and $0<\varepsilon<\e_0(C)  $ sufficiently small,  there is a unique stationary solution $\overline N_{\varepsilon} $ to {\eqref{eq_Nepsform_steady}} in $\mathcal U^{(\varepsilon)}(C)$.
\end{theorem}
  
 Working in $\mathcal U^{(\varepsilon)}(C)$ does not guarantee that our solution is positive. This latter property yields from the stability analysis since we recall \cite{JG.MH.SM:23} that positive initial data yield positive solutions  to the time-dependent problem. Hence, our second main result concerns the stability of such stationary solutions. In this analysis, we restrict again to solutions in the form $hG$ with $h \in L^2(\mathbb R,G(x)dx)$, or equivalently with $N \in L^2(\mathbb R,G^{-1}(x)dx)$.  We prove:

 \begin{theorem} \label{thm_longtime}
Assume \eqref{as:m-extremum}-\eqref{as:D3m-polynom}. Then we have two alternatives:
\begin{itemize}
\item[(i)] 
{Assume that $m^{''}(0) >0$.} {Then, for any $C>0$ and $\varepsilon\leq \e_0(C)$ sufficiently small}, any  $N_0  \in \mathcal U^{(\varepsilon)}({{C}})$ yields a unique global solution $N_{\varepsilon}= h_{\varepsilon} G$ to \eqref{eq_Nepsform} with $h_{\varepsilon} \in  C([0,\infty);L^2(\mathbb R;G(x)dx))$   which satisfies, for some constants $A_1$ and $A_2$,
\[
  \left\| \dfrac{N_{\varepsilon}(t,\cdot) -  \overline N_{\varepsilon}(\cdot)}{G} \right\|_{L^2(\mathbb R,G(x){\rm d}x)}\leq A_1\exp(- A_2 \e^2t) \left\| \dfrac{N_{\varepsilon,0}(\cdot) -  \overline N_{\varepsilon}(\cdot)}{G} \right\|_{L^2(\mathbb R,G(x){\rm d}x)}.
\]

\item[(ii)]  {Assume that  $m^{''}(0) <0$ and that $m$ is even. For any $C >0$ and $\varepsilon$ small}, {any even initial data $N_0 \in \mathcal U^{(\varepsilon)}(C)$} yields a unique global solution $N_{\varepsilon}= h_{\varepsilon} G$ to \eqref{eq_Nepsform} with $h_{\varepsilon} \in  C([0,\infty);L^2(\mathbb R;G(x)dx))$  which satisfies, for some constants $A_1$ and $A_2$,
\[
  \left\| \dfrac{N_{\varepsilon}(t,\cdot) -  \overline N_{\varepsilon}(\cdot)}{G} \right\|_{L^2(\mathbb R,G(x){\rm d}x)}\leq A_1\exp(-A_2t) \left\| \dfrac{N_{\varepsilon,0}(\cdot) -  \overline N_{\varepsilon}(\cdot)}{G} \right\|_{L^2(\mathbb R,G(x){\rm d}x)}.
\]

 \end{itemize}
 Consequently, {in these two cases,} the steady solution given in Theorem \ref{thm_stat} is positive.
 \end{theorem} 
 We point out again the discrepancy between the rate of convergence we obtain in item (ii) and the convergence rate in item (i). This is again due to the fact that item (i) allows non-zero first coefficient $\alpha_1$ in contrast with item (ii).
Notice that Theorems \ref{thm:local-steady} and \ref{thm_longtime_q} then follow from Theorems \ref{thm_stat} and \ref{thm_longtime}.

\subsection{ Hermite-polynomial expansion of source term $m$}

We note that $m_{\varepsilon} \in L^2(\mathbb R,G(x){\rm d}x)$ under the sole condition that it increases polynomially at infinity. This entails that $\overline{m}^{(\varepsilon)}_0 =(m_{k}^{(\e)})_{k\geq 0}\in \ell^2(\mathbb N)$ whatever $\varepsilon >0.$
The behavior of this sequence when $\varepsilon \to 0$ is the content of the following computation:
\begin{lemma} \label{lem_m}
There exists a constant $K_m \in (0,\infty)$ for which:
\[
\|m_{\varepsilon}\|_{L^2(\mathbb R,G(x){\rm d}x)} \leq K_m \varepsilon^2, \quad \forall \, \varepsilon \in (0,1).
\]
in particuliar $|m_k^{(\varepsilon)}| \leq K_m \varepsilon^2$  for all  $k \in \mathbb N_0.$ In case $k=1$ 
we have even the better estimate: 
\[
|m^{(\varepsilon)}_1| \leq K_m \varepsilon^3, \qquad  \forall \, \varepsilon \in (0,1).
\]
\end{lemma}

\begin{proof}
Recalling that $m(0) = m'(0) =0$  and \eqref{as:D2m-polynom}, we have:
\[
\begin{aligned}
\|m_{\varepsilon}\|^2_{L^2(\mathbb R,G(x){\rm d}x)}  & =  \int_{\mathbb R} |m_{\varepsilon}(x)- m_{\varepsilon}(0)|^2 G(x) {\rm d}x \\
& \leq \varepsilon^4 \int_{\mathbb R} {x^4} \left[  \int_0^{1} (1  -t) m_{\varepsilon}^{''}(t\varepsilon x) {\rm d}t\right]^2 G(x) {\rm d}x\\
& \leq  {4\varepsilon^{4}A_m^2} \int_{\mathbb R} x^4 {(1+ |x|^{p+1})^{2}}G(x){\rm d}x \\
& \leq C(p,A_m) \varepsilon^4 .
\end{aligned}
\]   
Eventually, we have proven that there exists a constant $K_m$ independent of $\varepsilon$ for which: 
\begin{equation} \label{eq_m_2}
\|m_{\varepsilon}\|_{L^2(\mathbb R,G(x){\rm d}x)}  \leq K_m \varepsilon^2\,, \quad \forall \, \varepsilon \in (0,1).
\end{equation}

\medskip

Regarding $m_1^{(\varepsilon)}$ we have a little better. Indeed, we recall that we computed {$H_1(x) = x.$} Hence, 
\[
m_1^{(\varepsilon)} = {  \int_{\mathbb R} (m_{\varepsilon}(x)- m_{\varepsilon}(0)) x G(x){\rm d}x}.
\]
However,  we have $m'(0) = 0$  so that, using the taylor expansion of $m$ in 
$0$ to order $3$ there holds:
\[
m_{\varepsilon}(x)  =  m^{''}(0) \dfrac{\varepsilon^2 x^2}{2} 
+ \varepsilon^3  x^3 \int_0^{1} \dfrac{(1 -t)^2}{2} m^{(3)}(\varepsilon x  t) {\rm d}t 
\]
and  
\[
m_1^{(\varepsilon)}  = {  \dfrac{m^{''}(0) \varepsilon^2}{2} \int_{\mathbb R}  x^3 G(x){\rm d}x + \varepsilon^3 \int_{\mathbb R} x^4  \int_0^{1} \dfrac{(1 -t)^2}{2} m^{(3)}(\varepsilon x t) {\rm d}t  G(x){\rm d}x}
\]
  By symmetry, the first term on the right-hand side vanishes and we conclude {using assumption \eqref{as:D3m-polynom}} that:
\begin{equation}
|m_1^{(\varepsilon)}| \leq K_m \varepsilon^3,  \qquad \forall \, \varepsilon \in (0,1).
\end{equation}
\qed \end{proof}

{\bf Remark.}
{\em{By generalizing the argument,  we can prove a similar  bound for  arbitrary coefficients  $m_k^{(\varepsilon)}$ provided that the successive derivatives of $m$
enjoy a polynomial growth condition similar to (H2).  Indeed,  given $k \in \mathbb N$ we use the Taylor expansion of $m_{\varepsilon}$ in $0$ to order $k-1$ with integral remainder.  We  obtain that there exists a polynomial $P_k$ of degree $\leq k-1$ such that:
\[
m_{\varepsilon}(x) = P_k(\varepsilon x) + \varepsilon^{k} x^k  \int_0^1 \dfrac{(1-t)^{k-1}}{(k-1)!} m^{(k)}(\varepsilon x t)  {\rm d}t 
\] 
where the polynomial growth of $m^{(k)}$ entails that:
\[
\left | \int_0^1 \dfrac{(1-t)^{k-1}}{(k-1)!} m^{(k)}(\varepsilon x t)  {\rm d}t \right| 
\leq A^{(k)}_m (1+ |x|)^{p_k}. 
\]
Since $P_k$ is of degree less than $(k-1),$ we have:
\[
(P_k(\varepsilon x), H_k)_G = 0
\]
and :
\[
|m_k^{(\varepsilon)}| \leq A^{(k)}_m \varepsilon^k \int_{\mathbb R}  (1+|x|)^{p_k} H_k G \leq  A^{(k)}_m \varepsilon^{k} \left[ \int_{\mathbb R} (1+|x|)^{2p} G(x){\rm d}x\right]^{\frac 12}.
\]
Eventually, we conclude that there exists a constant {$C_m^{(k)}$ independent of   $\varepsilon$} for which:
\begin{equation} \label{eq_F_infini}
|m_k^{(\varepsilon)}| \leq C_m^{(k)} \varepsilon^{k} \qquad \forall \, k \geq 1 \quad \forall \, \varepsilon < 1.
\end{equation}
Nevertheless, the constant $C_m^{(k)}$ depends a priori on $k$ and we should enforce a more stringent assumption to get uniform bounds in $k$ (which would be useless below anyway).  
}
}

For technical purpose below, we provide here also another estimate related to 
$m_{\varepsilon}.$
\begin{lemma} \label{lem_normmHk}
Let $l \in \mathbb N.$ There exists $K_l$  such that:
\[
\|m_{\varepsilon}{H_l}\|_{L^2(\mathbb R,G(x){\rm dx})} \leq K_l \varepsilon^2,\quad \forall \, \varepsilon \in (0,1).
\] 
\end{lemma}
\begin{proof}
Let $l \in \mathbb N.$ Arguing as in the previous proof, we have:
\begin{align*}
\|m_{\varepsilon} {H_l}\|^2_{L^2(\mathbb R,G(x){\rm dx})} 
& = \int_{\mathbb R} |m(\varepsilon x)|^2 |H_l(x)|^2 G(x){\rm d}x \\
& \leq  K_m \varepsilon^4  \int_{\mathbb R} |x|^2 (1+|x|^{{p+1}})^2|H_l(x)|^2 G(x){\rm d}x 
\end{align*}
The last integral is bounded by comparing the decay of $H_l$ and $G$ at infinity.
\qed \end{proof}

\section{The concentration property of the steady states: the proof of Theorem \ref{thm:concentration}-(i)} \label{sec:conc1}

 We recall that we treat from now on equation \eqref{eq_qepsform} with $r= 1$ and $\alpha = \varepsilon > 0.$ Our proof splits in two phases. Firstly, we show that the variance and higher moments of a possible solution have the right scaling. This enables {us} to prove that $M_{\e,1}$ has the right scaling and the expected expansion for $M_{e,2}.$ With this technical material at-hand, we provide a proof for Theorem \ref{thm:concentration}-(i).

\subsection{Preliminaries.}
As mentioned in introduction, an important ingredient in the proof of Theorem \ref{thm:concentration}-(i) is to show that the phenotypic variance $M_{\e,2}$ is small. We  prove indeed that:
\begin{proposition}
\label{lem:moments-m-II}
Assume   \fer{As:m},   and let $\overline q_\e  \in L^1(\mathbb R, (1+x^{2l})dx),$ for all $l \in \mathbb N,$  {be a solution of} \fer{eq_qepsform}, with its central moments defined in \fer{def-moments}. \\
(i) We have
\beq
\label{est-Me2}
 M_{\e,2}^c\leq \e^2.
\eeq
(ii) Assume additionally \fer{As:q}. Then, for all $k\geq 2$, there exists a positive constant $C_k$, independent of $\e$, such that
\beq
\label{bound-Mk}
M_{\e,k}^{|c|}:=\int_\R |x-M_{\e,1}|^k\overline q_\e(x)dx\leq C_k\e^k.
\eeq
\end{proposition}
\begin{proof}
 We start with the proof of item (i). {This proof relies strongly on  Assumption \fer{As:m}, that implies that $0$ is the  only admissible extremum point satisfying \fer{as:m-extremum}.}\\ 
We multiply \fer{eq_qepsform} by $(x-M_{\e,1})$ and integrate to obtain
\beq
\label{eq-xmq}
 \int_\R xm(x)\overline q_\e(x)dx= M_{\e,1}\int_\R m(x)\overline q_\e(x)dx .
\eeq
We next multiply \fer{eq_qepsform} by $(x-M_{\e,1})^2$ and integrate to obtain, using the equality above,
\[
\begin{aligned}
\left(\int_\R(x-M_{\e,1})^2 T_\e[\overline q_\e](x)-M_{\e,2}^c\right) &=\int_\R (x-M_{\e,1})^2m(x)\overline q_\e(x)dx\\
&-\left(\int m(x)\overline q_\e(x)dx\right)\, \left(\int (x-M_{\e,1})^2\overline q_\e(x)dx\right)
\\
&= \int_\R x^2m(x)\overline q_\e(x)dx-\left(\int m(x)\overline q_\e(x)dx\right)\, \left(\int x^2\overline q_\e(x)dx\right)\\
&-2M_{\e,1}\int xm(x)\overline q_\e(x)dx+2M_{\e,1}^2\int m(x)\overline q_\e(x)dx\\
&= \int_\R x^2m(x)\overline q_\e(x)dx-\left(\int_\R m(x)\overline q_\e(x)dx\right)\, \left(\int_\R x^2\overline q_\e(x)dx\right).
\end{aligned}
\]
One can verify that
\[
\int_\R(x-M_{\e,1})^2 T_\e[\overline q_\e](x)dx=\f{\e^2}{2}+{\frac{M_{\e,2}^c}{2}},
\]
and hence
{\beq
\label{eq:M2c-II}
\begin{array}{rl}
\left(\f{\e^2}{2}-\f{M_{\e,2}^c}{2}\right) &=\int_\R (x-M_{\e,1})^2m(x)\overline q_\e(x)dx
 -\left(\int m(x)\overline q_\e(x)dx\right)\, \left(\int (x-M_{\e,1})^2\overline q_\e(x)dx\right)
\\&= \int_\R x^2m(x)\overline q_\e(x)dx-\left(\int_\R m(x)\overline q_\e(x)dx\right)\, \left(\int_\R x^2\overline q_\e(x)dx\right).
\end{array}
\eeq}
 In order to prove {\fer{est-Me2}} it is hence enough to show that  the right-hand side is positive, that is:
$$
\left(\int_\R m(x)\overline q_\e(x)dx\right)\, \left(\int_\R x^2\overline q_\e(x)dx\right) \leq \int_\R x^2m(x)\overline q_\e(x)dx.
$$
{This is where Assumption \fer{As:m} plays an important role. To prove the inequality above,} we introduce the following probability density (recall that $m \geq 0$ in this section):
$$
p_\e(x)=\dfrac{m(x)\overline{q}_\e(x)}{\int_\R m(y)\overline q_\e(y)dy}, \quad \forall \, x \in \mathbb R.
$$
We have,  because of \eqref{eq-xmq} regarding the second identity:
\beq
\label{eq:M1pqII}
\int_\R p_\e(x)dx=\int_\R \overline q_\e(x)dx=1,\qquad \int_\R xp_\e(x)dx=\int_\R x\overline q_\e(x)dx=M_{\e,1}.
\eeq
We also notice that
$$
\f{p_\e(x)}{\overline q_\e(x)}=\frac{m(x)}{\int_\R m(y)\overline q_\e(y)dy}.
$$
 Inequality \fer{bound-int-mq} entails that we can set $ v := \int_\R m(y)\overline q_\e(y)dy \in (0,1)$ into  \fer{monotony-m}. We obtain that there exist $x_\ast,\,x^\ast\in (0,+\infty]$ such that
\beq
\label{pqx*}
p_\e(x)\leq \overline q_\e(x),\quad \text{in $[-x_\ast,x^\ast]$}\quad \text{and}\quad\overline  q_\e(x)\leq p_\e(x),\quad \text{in $(-\infty,-x_\ast]\cup [x^\ast,+\infty$).}
\eeq
We show that  these properties imply that
$$
\int_\R x^2\overline q_\e(x)dx\leq \int_\R x^2p_\e(x)dx.
$$
Let's suppose the contrary. Then, we can assume, without loss of generality up to  exchanging the role of $\R^+$ by $\R^-$ in the following arguments, that
$$
\int_{\R^+} x^2p_\e(x)dx< \int_{\R^+} x^2\overline q_\e(x)dx.
$$
We rewrite this inequality as below
$$
 \int_{x^\ast}^{+\infty} x^2\big(p_\e(x)-\overline q_\e(x) \big)dx< \int_{0}^{x^\ast} x^2\big(\overline  q_\e(x)-p_\e(x)\big)dx.
$$
Using \fer{pqx*} we obtain that the terms in the integrals above are nonnegative.   Let's suppose that $x^\ast\neq +\infty$.   We deduce that
$$
x^\ast \int_{x^\ast}^{+\infty} x \big(p_\e(x)-\overline q_\e(x)\big)dx\leq x^\ast\int_{0}^{x^\ast} x\big(\overline q_\e(x)- p_\e(x)\big)dx,
$$
$$
(x^\ast)^2 \int_{x^\ast}^{+\infty}  \big(p_\e(x)-\overline q_\e(x)\big)dx\leq (x^\ast)^2\int_{0}^{x^\ast} \big(\overline q_\e(x)- p_\e(x)\big)dx.
$$
 Since $x^* >0,$ the two inequalities entail respectively:
$$
\int_{\R^+} xp_\e(x)dx< \int_{\R^+} x\overline q_\e(x)dx,\qquad \int_{\R^+} p_\e(x)dx< \int_{\R^+} \overline q_\e(x)dx.
$$
Note that these inequalities hold also trivially if $x^\ast=+\infty$. We next use   \fer{eq:M1pqII} to find that
\beq
\label{eq-pqx-II}
\int_{\R^+}xp_\e(-x)dx<\int_{\R^+} x\overline q_\e(-x)dx,\qquad \int_{\R^+} p_\e(-x)dx> \int_{\R^+} \overline q_\e(-x)dx.
\eeq
We recall that
$$
p_\e(-x)\leq \overline q_\e(-x),\quad \text{for all $0\leq x\leq x_\ast$},\quad\overline q_\e(-x)<p_\e(-x),\quad \text{for all $x> x_\ast$}.
$$
Note that in view of \fer{eq-pqx-II},   $x_\ast$ may not be equal to $+\infty$ nor equal to $0$.  We then obtain that
$$
\int_{0}^{x_\ast} xp_\e(-x)dx+\int_{x_\ast}^{+\infty} xp_\e(-x)dx< \int_{0}^{x_\ast} x\overline q_\e(-x)dx+\int_{x_\ast}^{+\infty} x\overline q_\e(-x)dx
$$
that we rewrite as below
$$
\int_{x_\ast}^{+\infty} x\big(p_\e(-x)-\overline q_\e(-x)\big)dx<  \int_{0}^{x_\ast} x\big(\overline q_\e(-x)-p_\e(-x)\big)dx.
$$
Notice again that similarly to above the terms in the integrals above are nonnegative. We deduce that
$$
 x_\ast\int_{x_\ast}^{+\infty}\big(p_\e(-x)-\overline{q}_\e(-x)\big)dx< x_\ast\int_{0}^{x_\ast} \big(\overline q_\e(-x)-p_\e(-x)\big)dx.
$$
We deduce that
$$
\int_{\R^+} p_\e(-x)dx <\int_{\R^+} \overline q_\e(-x)dx.
$$
This is in contradiction with the second equation in \fer{eq-pqx-II}. This conclude the proof of (i).

\medskip

 We turn now to the proof of (ii).  This proof relies on the analysis of the equation on $M_{\e,2l}^c$, which includes a dissipative term.     Assumption \fer{As:q} and the  fact that $m$ is nonnegative, together with a fine analysis of the terms coming from the reproduction term lead to the result.\\
  Assumption \fer{As:q} implies that there exists $l_0>0$ such that 
\[
\int m(x)\overline{q}_\e(x)dx<1 -\f{1}{2^{2l_0+1}}-\f\delta 2.
\]
 We first prove that, for all $l\geq \max(l_0,2)+1$,
$$
M_{\e,2l}^c\leq C_{2l}\e^{2l}.
$$
To this end, we  multiply  \fer{eq_qepsform} by $(x-M_{\e,1})^{2l}$ and integrate with respect to $x$ to obtain that 
$$
\begin{aligned}
\int_\R\int_\R\int_\R (x-M_{\e,1})^{2l}\Gamma_\e\big(x-\f{y_1+y_2}{2}\big)q_\e(y_1)q_\e(y_2)dy_1dy_2dz -{M_{\e,2l}^c}\\
=\int_\R (x-M_{\e,1})^{2l}m(x)q_\e(x)dx-\left(\int_\R m(x)q_\e(x)dx\right) M_{\e,2l}^c.
\end{aligned}
$$
We deduce, using \fer{As:m}, that
\beq
\label{eq:Mk}
\begin{aligned}
\int_\R\int_\R\int_\R (x-M_{\e,1})^{2l}\Gamma_\e\big(x-\f{y_1+y_2}{2}\big)\overline{q}_\e(y_1)\overline{q}_\e(y_2)dy_1dy_2dz &\geq \left(1 -\int_\R m(x)\overline{q}_\e(x)dx\right)M_{\e,2l}^c\\
&\geq \left(\f{1}{2^{2l_0+1}}+\f\delta2\right)M_{\e,2l}^c.
\end{aligned}
\eeq
We next provide an approximation of the l.h.s. $A_1$ of the inequality above.
Let's denote by $\e^{2i}\widetilde \sigma_{2i}$, the $2i$'s moment of the normal distribution $\Gamma_\e$. We then obtain by expanding little by little:
\[
\begin{aligned}
A_1&:=\int_\R\int_\R\int_\R (x-M_{\e,1})^{2l}\Gamma_\e\big(x-\f{y_1+y_2}{2}\big)q_\e(y_1)q_\e(y_2)dy_1dy_2dx\\
&=\int_\R\int_\R\int_\R \Big(x-\f{y_1+y_2}{2}+ \f{y_1+y_2}{2}-M_{\e,1}\Big)^{2l}\Gamma_\e\big(x-\f{y_1+y_2}{2}\big)q_\e(y_1)q_\e(y_2)dy_1dy_2dx\\
 &= \sum_{i=0}^{l}   \begin{pmatrix}2l  \\ 2i  \end{pmatrix}\f{\e^{2i}\widetilde \sigma_{2i}}{2^{2l -2i}}\int_\R\int_\R  \Big(  y_1-M_{\e,1}+ y_2-M_{\e,2}\Big)^{2l-2i}  q_\e(y_1)q_\e(y_2)dy_1dy_2 \\
 & = \sum_{i=0}^{l}   \begin{pmatrix} 2l \\ 2i  \end{pmatrix}\f{\e^{2i}\widetilde{\sigma}_{2i}}{2^{2l -2i}}\sum_{k=0}^{2l -2i}
   \begin{pmatrix}2l -2i \\ k  \end{pmatrix} M_{\e,k}^c M_{\e,2l -2i-k}^c.\\
\end{aligned}
\]
  We gather in this expression, the smaller and larger moments. For this, we note that $M_{\e,0}^c=1$, $M_{\e,1}^c=0$  and that $\widetilde{\sigma}_0=1$ and we consider that
\begin{itemize}
\item for $i=l:$
\[
 \begin{pmatrix} 2l \\ 2i  \end{pmatrix}\f{\e^{2i}\widetilde{\sigma}_{2i}}{2^{2l -2i}}\sum_{k=0}^{2l -2i}
   \begin{pmatrix}2l -2i \\ k  \end{pmatrix} M_{\e,k}^c M_{\e,2l -2i-k}^c = \e^{2l}\tilde{\sigma}_{2l},
\]
\item for $i=l-1$ 
\[
 \begin{pmatrix} 2l \\ 2i  \end{pmatrix}\f{\e^{2i}\widetilde{\sigma}_{2i}}{2^{2l -2i}}\sum_{k=0}^{2l -2i}
   \begin{pmatrix}2l -2i \\ k  \end{pmatrix} M_{\e,k}^c M_{\e,2l -2i-k}^c =  
   \begin{pmatrix} 2l \\ 2l-2  \end{pmatrix} \e^{2l-2}\widetilde{\sigma}_{2l-2}\f{M_{\e,2}^c}{2},
\]
\item in the terms $i=0,...,l-2$,
\end{itemize}
\[
 \sum_{k=0,1,2l-2i-1,2l-2i}  \begin{pmatrix}2l -2i \\ k  \end{pmatrix} M_{\e,k}^c M_{\e,2l -2i-k}^c
    = 2M_{\e,2l -2i}^c.
\]
We infer that
\beq \label{eq-A1}
\begin{aligned}
A_1 &= \e^{2l}\widetilde{\sigma}_{2l}+ \begin{pmatrix} 2l \\ 2l-2  \end{pmatrix}\e^{2l-2}\widetilde{\sigma}_{2l-2}\f{M_{\e,2}^c}{2}\\
&+\sum_{i=0}^{l-2}   \begin{pmatrix} 2l \\ 2i  \end{pmatrix}\e^{2i}\widetilde{\sigma}_{2i}\f{1}{2^{2l -2i}}\sum_{k=2}^{2l -2i-2}
   \begin{pmatrix}2l -2i \\ k  \end{pmatrix} M_{\e,k}^c M_{\e,2l -2i-k}^c\\
   &+\f{1}{2^{2l-1}}M_{\e,2l}^c+\sum_{i=1}^{l-2}   \begin{pmatrix} 2l \\ 2i  \end{pmatrix}\f{\e^{2i}\widetilde{\sigma}_{2i}}{2^{2l -2i-1}} 
   M_{\e,2l -2i}^c.
\end{aligned}
\eeq
We next notice that, using an interpolation argument, for all $j$ such that $2\leq j\leq 2l-2$, we have
$$
M_{\e,j}^c\leq \big( M_{\e,2}^c \big)^{\f{2l-j}{2l-2}}  \big( M_{\e,2l}^c \big)^{\f{ j-2}{2l-2}}.
$$
We deduce, thanks to \fer{est-Me2}, that
$$
M_{\e,j}^c\leq \e^{\f{2l-j}{l-1}}  \big( M_{\e,2l}^c \big)^{\f{ j-2}{2l-2}}.
$$
We use this inequality in \fer{eq-A1} to obtain that
$$
\begin{aligned}
A_1&\leq \e^{2l}\widetilde{\sigma}_{2l}+ \begin{pmatrix} 2l \\ 2l-2  \end{pmatrix}\f{\widetilde{\sigma}_{2l-2}\e^{2l }}{2}\\
&+\sum_{i=0}^{l-2}   \begin{pmatrix} 2l \\ 2i  \end{pmatrix}\f{\e^{2i}\widetilde{\sigma}_{2i}}{2^{2l -2i}}\sum_{k=2}^{2l -2i-2}
   \begin{pmatrix}2l -2i \\ k  \end{pmatrix}  \e^{\f{2l+2i}{l-1}}  \big( M_{\e,2l}^c \big)^{\f{ 2l-2i-4}{2l-2}}  \\
   &+\f{1}{2^{2l-1}}M_{\e,2l}^c+\sum_{i=1}^{l-2}   \begin{pmatrix} 2l \\ 2i  \end{pmatrix} \e^{2i}\widetilde{\sigma}_{2i}\f{1}{2^{2l -2i-1}} 
  \e^{\f{2i}{l-1}} \big( M_{\e,2l}^c \big)^{\f{ l-i-1}{l-1}}\\
 & \leq \e^{2l}\left( \widetilde{\sigma}_{2l}+  \f{l(2l-1)}{2}\widetilde{\sigma}_{2l-2}\right)+\f{1}{2^{2l-1}}M_{\e,2l}^c\\
&+\sum_{i=0}^{l-2}   \begin{pmatrix} 2l \\ 2i  \end{pmatrix} \f{\widetilde{\sigma}_{2i}}{2^{2l -2i}} \left(\sum_{k=2}^{2l -2i-2}
   \begin{pmatrix}2l -2i \\ k  \end{pmatrix}  \right) \e^{\f{2l+2il}{l-1}}  \big( M_{\e,2l}^c \big)^{\f{ 2l-2i-4}{2l-2}}  \\
   &+\sum_{i=1}^{l-2}   \begin{pmatrix} 2l \\ 2i  \end{pmatrix} \f{\widetilde{\sigma}_{2i}}{2^{2l -2i-1}} 
  \e^{\f{2il}{l-1}} \big( M_{\e,2l}^c \big)^{\f{ l-i-1}{l-1}}.
   \end{aligned}
$$
  We next remark that $(2l-2i-4)/(2l-2) \in (0,1)$ for all $i \in \{0,l-2\}$ (resp. $(l-i-1)/l-1 \in (0,1)$ for all $i \in \{1,\ldots,l-2$), hence we may then the Young's inequality in the last two lines above to find  that, for arbitrary $c_0 >0,$ there is a constant $\tilde{K}(l,c_0)$ depending on $l,$ the $(\widetilde{\sigma}_{2i})_{i=1,\ldots,l}$  and $c_0$ such that:
\[
A_1 \leq \e^{2l} \tilde{K}(l,c_0) + \left(\dfrac{1}{2^{2l-1}} + c_0 \right) M_{\e,2l}^{c}.
\]
%
Choosing next $c_0$ small enough we obtain  a constant $K_{2l}$ depending only on $l,$ possibly in a complicated way, such  that
\beq
\label{final-Es-A1}
  A_1\leq K_{2l}\e^{2l}+\left(\f{1}{2^{2l-1}}+\f{\delta}{4}\right)M_{\e,2l}^c.
\eeq
 Plugging \fer{final-Es-A1} in the left-hand side of \fer{eq:Mk} and using $l\geq l_0+1$ we deduce that
$$
 M_{\e,2l}^c\leq \f{4 K_{2l}}{\delta}\e^{2l}.
$$
Note that using H{\"o}lder inequality, the inequality above implies that
$$
M_{\e,k}^{|c|}\leq C_k\e^k,
$$
for all  $2 {\leq} k \leq l$ and thus for all $k \geq 2$ since $l$ is arbitrary large.
\qed

\end{proof}

%

\bigskip

%
%

We next prove the following Lemma. 
\begin{lemma}
\label{lem:M1e}
Let's assume \fer{As:m}--\fer{As:m2}--\fer{As:q}. Then, for $\e\leq \e_0$ small enough, we have
$$
 |M_{\e,2}^c-\e^2|\leq C\e^4,\qquad M_{\e,1}\leq C\e,
$$
with $C$ a constant that does not depend  on   $\e$.
\end{lemma}

\begin{proof}
{In order to obtain the   precise estimate on $M_{\e,2}^c$ in Lemma \ref{lem:M1e}, we use again \fer{eq:M2c-II}, but this time we provide a more precise approximation of the r.h.s using a Taylor expansion on $m$ and the previous estimates \fer{est-Me2}--\fer{bound-Mk}.} We recall that \fer{eq:M2c-II} reads
$$
\left(\f{\e^2}{2}-\f{M_{\e,2}^c}{2}\right)=\int_\R (x-M_{\e,1})^2m(x)\overline{q}_\e(x)dx-M_{\e,2}^c\int_\R m(x)\overline{q}_\e(x)dx.
$$
We next use a Taylor expansion on $m$ as follows
\beq
\label{T1}
m(x)=m(X)+(x-X)m'(X)+(x-X)^2r_1^m[X](x), 
\eeq
with
$$
r_1^m[X](x)=\int_0^1 (1-\sigma) m''(X+\sigma(x-X))d\sigma \text{ satisfying } |r_1^m[X](x)|\leq C_m , \quad 
\forall \, x \in \mathbb R,
$$
by \eqref{As:m2}. This leads to
\[
\begin{aligned}
\int_\R (x-M_{\e,1})^2m(x)\overline{q}_\e(x)dx&=m(M_{\e,1})\int_\R (x-M_{\e,1})^2 \overline{q}_\e(x)dx+m'(M_{\e,1})\int_\R (x-M_{\e,1})^3 \overline{q}_{\e}(x)dx\\
&+\int_\R (x-M_{\e,1})^4 r_1^m[M_{\e,1}](x)dx\\
&=m(M_{\e,1})M_{\e,2}^c+m'(M_{\e,1})M_{\e,3}^c+\int_\R (x-M_{\e,1})^4 r_1^m[M_{\e,1}](x) \overline{q}_{\e}(x)dx.
\end{aligned}
\]
We then obtain thanks to \fer{bound-Mk} and Assumption \fer{As:m2}    that
$$
\left|\int_\R (x-M_{\e,1})^2m(x)\overline{q}_{\e}(x)dx-m(M_{\e,1})M_{\e,2}^c-m'(M_{\e,1})M_{\e,3}^c\right|\leq C\e^4.
$$
Similarly, we compute using \fer{T1}
$$
\int_\R m(x)\overline{q}_{\e}(x)dx=m(M_{\e,1})+\int_\R (x-M_{\e,1})^2 r_1^m[M_{\e,1}](x)\overline{q}_{\e}(x)dx,
$$
and again thanks to \fer{bound-Mk} and Assumption \fer{As:m2}   we obtain that
$$
\left|M_{\e,2}^c \int_\R m(x){\overline q}_\e(x)dx-m(M_{\e,1})M_{\e,2}^c\right|\leq C\e^4.
$$
 Introducing these inequalities in \eqref{eq:M2c-II}  and applying \eqref{bound-Mk} with $k=3$, we  conclude that 
\beq
\label{eq:M2c-correct}
\left| \f{\e^2}{2}-\f{M_{\e,2}^c}{2}\right| \leq {|m'(M_{\e,1})M_{\e,3}^c|+C\e^4 \leq C  \e^3}. 
\eeq
We next rewrite the equality \fer{eq-xmq} as below
$$
\int_\R (x-M_{\e,1})m(x) {\overline q}_\e(x)dx=0.
$$
Using again \fer{T1} we obtain that 
$$
m'(M_{\e,1})M_{\e,2}^c+\int_\R (x-M_{\e,1})^3 r_1^m[M_{\e,1}](x){\overline{q}}_{\e}(x)dx=0.
$$
We deduce again that
$$
 \left| m'(M_{\e,1})M_{\e,2}^c\right|\leq C\e^3.
$$
Since $|M_{\e,2}^c-\e^2| \leq C\e^3$, we find that  for $\e$ sufficiently small
\beq
\label{est-mM1}
 |m'(M_{\e,1})|\leq C\e.
\eeq
We next notice that
$$
\int_\R m(x){\overline q}_\e(x)dx=m(M_{\e,1})+O(\e^2),
$$
which implies that for $\e$ small enough
$$
m(M_{\e,1})<1 - \delta/2 .
$$
This inequality together with \fer{est-mM1} and Assumption \fer{As:m2} implies that
$$
|M_{\e,1}| \leq C\e.
$$
We also deduce using \fer{eq:M2c-correct} that
$$
|M_{\e,2}^c -\e^2| \leq C \e^4.
$$
\qed
\end{proof}

\subsection{The proof of Theorem \ref{thm:concentration}--(i).}  
We divide the proof in three parts : (0) integrability properties of $\overline{q}_{\e},$ (i) the estimate on $M_{\e,k}^c$, (ii) the estimate on $|M_{\e,1}|$.

\medskip

(0) {\bf Integrability properties of $\overline{q}_{\e}.$} By definition, if $\overline{q}_{\e}$ is a solution of \eqref{eq_qepsform}, we have:
\[
\left(1 + m(x) - \int_{\mathbb R} m(y) q_{\e}(y)dy\right)\overline{q}_{\e} = T_e[\overline{q}_{\e}].
\] 
We note here that combining assumption \eqref{As:q} with \eqref{As:m2}, we have 
\[
\left(1 + m(x) - \int_{\mathbb R} m(y) q_{\e}(y)dy\right) \geq c(1+x^2).
\]
so that:
\begin{equation} \label{eq_qtemp}
\overline{q}_{\e}(x) \leq \dfrac{C}{1+x^2} |T_{\e}[q_{\e}](x)| \quad \forall \, x \in \mathbb R.
\end{equation}
At this point, we obtain our result by induction on $\ell.$ We already have the property for $\ell=1$ by assumption.  If the property holds for $l\in \mathbb N,$ then we may apply Proposition \ref{prop_contT} to obtain that $T_{\e}[\overline{q}_{\e}] \in L^1((1+x^{2})^{l} dx).$ Plugging in \eqref{eq_qtemp} entails that $\overline{q}_\e \in L^{1}((1+x^2)^{(l+1)}dx).$  

\medskip

(i) {\bf The estimate on $M_{\e,k}^c$.} We prove the estimate on $M_{\e,k}^c$ by {induction}. The estimate is already proved for $k=2$. We assume that the estimate holds for all $k\leq k_0-1$. We prove it for $k=k_0$  following similar arguments as in the proof of Lemma \ref{lem:M1e} on the estimate on $M_{\e,2}^c$.

 We multiply \fer{eq_qepsform} by $(x-M_{\e,1})^k$ and integrate to obtain
\beq
\label{eq:Mek}
\begin{aligned}
\int_\R\int_\R\int_\R (x-M_{\e,1})^{k}\Gamma_\e\big(x-\f{y_1+y_2}{2}\big){\overline q}_\e(y_1){\overline q}_\e(y_2)dy_1dy_2dz - M_{\e,2l}^c\\
=\int_\R (x-M_{\e,1})^{k}m(x){\overline q}_\e(x)dx-\left(\int_\R m(x){\overline q}_\e(x)dx\right) M_{\e,k}^c.
\end{aligned}
\eeq
Let's first estimate the right hand side. Using \fer{T1}, evaluated at $X=M_{\e,1}$, we write
$$
\int_\R (x-M_{\e,1})^{k}m(x){\overline q}_\e(x)dx=m(M_{\e,1})M_{\e,k}^c+m'(M_{\e,1})M_{\e,k+1}^c+
\int_\R (x-M_{\e,1})^{k+2}r_1^m[M_{\e,1}](x){\overline q}_\e(x)dx,
$$
and hence, since $|M_{\e,1}|\leq C\e$ and $|M_{\e,j}^c| \leq C_j\e^j$,
$$
\left|\int_\R (x-M_{\e,1})^{k}m(x){\overline q}_\e(x)dx-m(M_{\e,1})M_{\e,k}^c\right|\leq D_k\e^{k+2}.
$$
We obtain similarly
$$
\int_\R  m(x){\overline q}_\e(x)dx=m(M_{\e,1})+m'(M_{\e,1})M_{\e,1}+
\int_\R (x-M_{\e,1})^{2}r_1^m[M_{\e,1}](x){\overline q}_\e(x)dx,
$$
and hence 
$$
\left|\int_\R  m(x){\overline q}_\e(x)dx- m(M_{\e,1})\right|\leq D_0\e^2.
$$
We deduce that, up to changing the constant $D_k$,
\beq
\label{est-lhs}
\left|\int_\R (x-M_{\e,1})^{k}m(x){\overline q}_\e(x)dx-(\int_\R m(x){\overline q}_\e(x)dx) M_{\e,k}^c\right|\leq D_k\e^{k+2}.
\eeq
Recall that $\sigma_k$ is the $k$-th order moment of $G$ and $\e^k\widetilde \sigma_k$ is the $k$-th order moment of $\Gamma_\e$.
 Note that
\beq
\label{sigma-vs-tilde}
 \sigma_k=2^{k/2}\widetilde\sigma_k=
\begin{cases}
{\frac{k!}{2^{k/2}(k/2)!}} & \text{for $k$ even},\\
0&\text{for $k$ odd}.
\end{cases}
\eeq
For the l.h.s. of \eqref{eq:Mek}, we compute
$$
\begin{aligned}
&\int_\R\int_\R\int_\R (x-M_{\e,1})^{k_0}\Gamma_\e\big(x-\f{y_1+y_2}{2}\big){\overline q}_\e(y_1){\overline q}_\e(y_2)dy_1dy_2dx\\
&=\int_\R\int_\R\int_\R \left(x-\f{y_1+y_2}{2} + \f{y_1+y_2}{2}-M_{\e,1}\right)^{k_0}\Gamma_\e\left(x-\f{y_1+y_2}{2}\right){\overline q}_\e(y_1){\overline q}_\e(y_2)dy_1dy_2dx\\
&=\sum_{i=0}^{k_0}  \begin{pmatrix} k_0 \\ i  \end{pmatrix}  \f{\e^i\widetilde{\sigma_i}}{2^{k_0-i}}\sum_{j=0}^{k_0-i} \begin{pmatrix} k_0-i \\ j \end{pmatrix}  
M_{\e,j}^cM_{\e,k_0-i-j}^c.
\end{aligned}
$$
We next use the assumption of induction to deduce that
$$
\begin{aligned}
& \int_\R\int_\R\int_\R (x-M_{\e,1})^{k_0}\Gamma_\e\big(x-\f{y_1+y_2}{2}\big)\overline{q}_\e(y_1)\overline{q}_\e(y_2)dy_1dy_2dx\\
&=\f{1}{2^{k_0-1}}M_{\e,k_0}^c+\e^{k_0}\sum_{i=1}^{k_0}  \begin{pmatrix} k_0 \\ i  \end{pmatrix}   \widetilde{\sigma_i}\f{1}{2^{k_0-i}}\sum_{j=0}^{k_0-i} \begin{pmatrix} k_0-i \\ j \end{pmatrix}  
  \sigma_j   \sigma_{k_0-i-j}
  \\
& + \e^{k_0} \f{1}{2^{k_0}}\widetilde \sigma_0\sum_{j=1}^{k_0-1} \begin{pmatrix} k_0 \\ j \end{pmatrix}  
  \sigma_j   \sigma_{k_0-j}
  +O(\e^{k_0+2}).
\end{aligned}
$$
 In these sums, we note that $j$ and $k_0-j$ (resp. $i,j$ and $k_0-i-j$) cannot be simultenaously even if $k_0$ is odd.  As a direct consequence, since the odd central moments of a Normal distribution vanish, if $k_0$ is odd we obtain that
$$
\int_\R\int_\R\int_\R (x-M_{\e,1})^{k_0}\Gamma_\e\big(x-\f{y_1+y_2}{2}\big)\overline{q}_\e(y_1)\overline{q}_\e(y_2)dy_1dy_2dx=\f{1}{2^{k_0-1}}M_{\e,k_0}^c+O(\e^{k_0+2}),
$$
and hence 
$$
M_{\e,k_0}^c=O(\e^{k_0+2}). 
$$
Let's now assume that $k_0$ is even and hence there exists $l_0$ such that $k_0=2l_0$. We have, using again the equality above, \fer{sigma-vs-tilde} and the fact that the odd central moments of a Normal distribution vanish,
$$
\begin{aligned}
&\int_\R\int_\R\int_\R (x-M_{\e,1})^{2l_0}\Gamma_\e\big(x-\f{y_1+y_2}{2}\big)\overline{q}_\e(y_1)\overline{q}_\e(y_2)dy_1dy_2dx-\f{1}{2^{2l_0-1}}M_{\e,2l_0}^c\\
&=\e^{2l_0}\f{1}{2^{2l_0}}\sum_{i=1}^{l_0}  \begin{pmatrix} 2l_0 \\ 2i  \end{pmatrix}  {{\sigma}_{2i}2^i} \sum_{j=0}^{l_0-i } \begin{pmatrix} 2l_0-2i \\ 2j \end{pmatrix}  
 \sigma_{2j}  \sigma_{2l_0-2i-2j} \\
&+\e^{2l_0} \f{1}{2^{2l_0}}    {\sigma}_{0} \sum_{j=1}^{l_0-1 } \begin{pmatrix} 2l_0  \\ 2j \end{pmatrix}  
  \sigma_{2j}  \sigma_{2l_0 -2j}+O(\e^{k_0+2}).
\end{aligned}
$$
We next use \fer{sigma-vs-tilde} to obtain that
$$
\begin{pmatrix} 2l_0 \\ 2i  \end{pmatrix}\begin{pmatrix} 2l_0-2i \\ 2j \end{pmatrix}    \sigma_{2i} 
  \sigma_{2j}   \sigma_{2l_0-2i-2j}=  \sigma_{2l_0}\begin{pmatrix}  l_0 \\  i  \end{pmatrix}\begin{pmatrix} l_0- i \\  j \end{pmatrix} .
$$
We deduce that
$$
\begin{aligned}
& \int_\R\int_\R\int_\R (x-M_{\e,1})^{2l_0}\Gamma_\e\big(x-\f{y_1+y_2}{2}\big)q_\e(y_1)q_\e(y_2)dy_1dy_2dx-\f{1}{2^{2l_0-1}}M_{\e,2l_0}^c\\
&=\e^{2l_0}  \sigma_{2l_0} \f{1}{2^{{2 l_0}}}\sum_{i=1}^{l_0}  \begin{pmatrix}  l_0 \\  i  \end{pmatrix} {2}^i\sum_{j=0}^{l_0-i} \begin{pmatrix} l_0-i \\ j \end{pmatrix}   \\
& +\e^{2l_0}  \sigma_{2l_0}\f{1}{2^{2l_0}}      \sum_{j=1}^{l_0-1 } \begin{pmatrix} l_0  \\ j \end{pmatrix}  
 +
 O(\e^{k_0+2})\\
& =\e^{2l_0}  \sigma_{2l_0} \f{1}{2^{ {2l_0}}}\left[{2^{l_0}(2^{l_0}-1)+2^{l_0}-2}\right]+O(\e^{k_0+2})=\e^{2l_0}  \sigma_{2l_0}(1-\f{1}{2^{2l_0-1}})+O(\e^{k_0+2}).
\end{aligned}
$$
We then combine \fer{eq:Mek}, \fer{est-lhs} and the equation above to obtain that
$$
M_{\e,k_0}^c=\e^{k_0}  \sigma_{k_0}  +O(\e^{k_0+2}). 
$$


(ii) {\bf The estimate on $|M_{\e,1}|$.} Similarly to the proof of Lemma \ref{lem:M1e} we use the following equality 
$$
\int_\R (x-M_{\e,1})m(x)\overline{q}_\e(x)dx=0.
$$
We next use a higher order Taylor expansion for $m(x)$ around $M_{\e,1}$
$$
m(x)=m(M_{\e,1})+m'(M_{\e,1})(x-M_{\e,1})+m''(M_{\e,1})(x-M_{\e,1})^2+r_2^m [M_{\e,1}](x)(x-M_{\e,1})^3.
$$
Combining the equalities above, we obtain
$$
m'(M_{\e,1})M_{\e,2}^c+m''(M_{\e,1})M_{\e,3}^c=O(\e^4).
$$
From the estimate in step (i) we deduce that
$$
m'(M_{\e,1})M_{\e,2}^c=O(\e^4).
$$
From Lemma \ref{lem:M1e} we deduce that
$$
m'(M_{\e,1})=O(\e^2),
$$
and hence 
$$
M_{\e,1}=O(\e^2). 
$$
 \qed

  \section{The proof of Theorem \ref{thm:concentration}--(ii)}
  \label{sec:conc2}

 In this section we prove Theorem \ref{thm:concentration}--(ii). Let $\overline q_\e$ be a steady solution to \fer{eq_qepsform} and $\overline N_\e(x)=\e\overline q_\e(\e x)$. We will prove that 
 \beq
 \label{est-Ne-G-L2}
\left\| \overline N_\e(x)-G(x) (1+{\widetilde M}_{\e,1}\,x) \right\|_{L^2(\R,G^{-1}(x)dx)}\leq C\e^2,
 \eeq
 which implies \fer{est-qe-Ge-L2}. To this end, we will use the expansion of $\overline N_\e$ in terms of Hermite polynomials, that is $N_\e=h_\e G$, with $h_\e\in L^2(\R,G(x)dx)$ that we can expand {as follows}
\[
h_\e(x)=\sum_{i=0}^{+\infty} \alpha_i H_i(x).
\]
Since $\alpha_0=1$ and $\alpha_1=\widetilde M_{\e,1}$, $H_0(x)=1$, $H_1(x)=x$, proving \fer{est-Ne-G-L2} is equivalent with showing
\beq
\label{est-sum-alpha}
\sum_{k=2}^{+\infty} \alpha_k^2\leq C\e^4.
\eeq

  We split the proof of Theorem \ref{thm:concentration}-(ii) into two steps. We first obtain separately informations on moments of $h_{\e}.$ We postpone the proof of Theorem \ref{thm:concentration}-(ii) to the last part of this section.

\subsection{Crude bounds on moments of $h_{\e}$}

From item (i), we already have \fer{est-Mek} on the moments of $\overline{q}_{\e}$. We translate at first this estimate in terms of the Hermite coefficients $(\alpha_{k})_{k \in \mathbb N}$ in the lemma below.
We emphasize that such estimates are interesting by themselves, but they hold for fixed $k$ only so that we will need a different analysis afterwards to complete our result.

\begin{lemma} \label{lem:alphakpointwise}
For all $k\geq 1$, there exists a constant $C_k$ such that
 $$
 |\alpha_1|\leq C_1\e, \qquad |\alpha_k|\leq C_k \e^2. 
 $$
 \end{lemma}
 \begin{proof}
  We first notice thanks to \fer{est-Mek} and \fer{link-M-tildeM} that
 \[
| \alpha_1|=|\widetilde M_{\e,1}|\leq C_1\e.
\]
 We next prove the result for $k\geq 2$   by induction and using Lemma \ref{link-Hermite-moments}.
 Note that, for $k=2$,
 $$
 \alpha_2=\f{1}{\sqrt{2}}\big(\widetilde M_{\e,2}^c-\sigma_2-R(2)\big),\qquad R(2)=\widetilde M_{\e,1}^2-2\widetilde M_{\e,1}\alpha_1=-\widetilde M_{\e,1}^2.
 $$
Using again \fer{est-Mek} and \fer{link-M-tildeM} we deduce that $|\alpha_2|\leq C_2\e^2$.
We now suppose that, for all $2\leq k\leq k_0$, $|\alpha_k| \leq C_k \e^2$. we prove the result for $k=k_0+1$. In view of Lemma \ref{link-Hermite-moments}, \fer{est-Mek} and \fer{link-M-tildeM} it is enough to prove that $|R(k_0+1)| \leq C_{k_0+1} \e^2$. 

\medskip

Since for all $2\leq k\leq k_0$, $|\alpha_k| \leq C_k \e^2$ and since $|\widetilde M_{\e,1}| \leq C_1 \e$, we deduce that
\[
 \begin{aligned}
 R(k_0+1)= &  - (k_0+1) \widetilde M_{\e,1} \left( k_0 \alpha_1\sigma_{k_0-1}+\alpha_0\sigma_{k_0}\right)
 + (k_0+1) \alpha_1\sigma_{k_0}+O(\e^2) \\
 =  & (k_0+1) (\alpha_1 - \widetilde{M}_{\e,1}) \sigma_{k_0} + O(\e^2)
\end{aligned}
\]
 where $O(\e^2)$ stands for a quantity dominated by $C_{k_0} \e^2.$
 We next use that $\alpha_1=\widetilde M_{\e,1}$, to deduce that $| R(k_0+1)| \leq C_{k_0+1} \e^2.$
 This ends the proof.
 \qed
 \end{proof}

\medskip

   To compute $\ell^2$-norm of the coefficients $(\alpha_k)_{k\in \mathbb N},$ we will need the following lemma that yields a control on the fourth moments of the solution.
 \begin{lemma}
 \label{N-L2-bound}
Assume \fer{As:m}, \fer{As:m2} and \fer{cond-q/G-bound}. Then, there exists a positive constant $A_3$ such that, for $\e$ small enough,
 \beq
 \label{bound x4NL2}
 \int_\R x^4  \f{\overline N_\e^2(x)}{G(x)}dx\leq A_3.
 \eeq
 \end{lemma}
 
\begin{proof}
We recall that
\[
{\overline q}_\e=\frac{T_\e[\overline{q}_\e]}{1+m-\int_\R m(y)\overline{q}_\e(y)dy.}
\]
Assumption \fer{As:m2} and Lemma \ref{lem:M1e}  imply that 
\[
\int_{\R} m(y)\overline{q}_\e(y)dy\leq C_m\int_{\mathbb R} y^2\overline{q}_\e(y)dy \leq C\e^2.
\]
Therefore, for $\e$ small enough, $\int_\R m(y)\overline{q}_\e(y)dy\leq 1/2$. 
Using this property and since $m(x)\geq 0$ we obtain that, for $\e$ small enough,
\beq
\label{comp-q-T}
\overline{q}_\e\leq 2 T_\e[\overline{q}_\e],
\eeq
We deduce from {\fer{comp-q-T} and the definition of $T_{\e}$} that
\[
0 \leq \overline{q}_\e(x)\leq \f{2}{\e\sqrt{\pi}}\int\int \exp\left( -\f{(x-\f{y_1+y_2}{2})^2+{\delta'} y_1^2+{\delta'} y_2^2}{\e^2}\right)\overline{q}_\e(y_1)e^{\f{{\delta'} y_1^2}{\e^2}}\overline{q}_\e(y_2)e^{\f{{\delta'} y_2^2}{\e^2}} dy_1dy_2.
\]
One can verify that
 \[
 -x^2-\f{(y_1+y_2)^2}{4}+x(y_1+y_2)-{\delta'}(y_1^2+y_2^2)\leq -{\f{2{\delta'}}{1+2{\delta'}}}x^2 \quad 
 \forall (x,y_1,y_2) \in \mathbb R^3,
 \]
 and hence 
\beq
\label{ineq-tail}
\overline{q}_\e(x)\leq \f{2}{\e \sqrt{\pi}} e^{-\frac{\f{2{\delta'}}{1+2{\delta'}}x^2}{\e^2}}\Big( \int_{\mathbb R} q_\e(y_1)e^{\f{{\delta'} y_1^2}{\e^2}}dy_1\Big)^2\leq
\f{2{A_1^2}}{\e \sqrt{\pi}}e^{ -\frac{\f{2{\delta'}}{1+2{\delta'}}x^2}{\e^2}}.
\eeq
{If ${\delta'}\geq 1/2$ then, $\f{2{\delta'}}{1+2{\delta'}}\geq 1/2$ and hence, we deduce from the assumed integral bound   that there exists a constant $A_2$ independent of $\e$ such that we have the pointwise bound:
 \beq
 \label{tail-qe}
\overline q_\e(x)\leq \f{A_2}{\e}e^{-\f{3x^2}{8\e^2}}, \quad \forall \, x \in \mathbb R.
 \eeq
 If we are in the case ${\delta'}<1/2,$ we notice that \fer{ineq-tail} implies that for any constant ${\delta''}<\f{2{\delta'}}{1+2{\delta'}}$, there exists a constant $A_1'$ such that
\[
\int_{\mathbb R} \overline q_\e(y) e^{\f{{\delta''} y^2}{\e^2}}dy\leq A_1'.
\]
Moreover, the mapping ${\delta'} \mapsto \f{2{\delta'}}{1+2{\delta'}} - {\delta'}$ is positive on $[0,1/2]$ and vanishes in $0$ and $1/2$ only.  Consequently, we can iterate the computations above a finite number of times to obtain that there exists a constant $A_2$ for which 
 \[
\overline{q}_\e(x)\leq \f{A_2}{\e}e^{-\f{3 x^2}{8\e^2}}, \quad \forall \, x \in \mathbb R.
 \]
 }
 The inequality \fer{bound x4NL2} then follows thanks to $\overline N_\e(x)=\e\overline q_\e(\e x)$.
 \qed
\end{proof} 
 
 {
 \subsection{Main proof}
 }
 We are now ready to prove \fer{est-sum-alpha}. 
   We will   find an estimate of type 
  $$
  \sum_{k=2}^L |\alpha_k|^2\,\leq C\e^4,\qquad \text{for $\e\leq \e_0$, with $\e_0$ small enough},
  $$
and $C$ independent of $L$. We will prove this property by induction. {Since  this property is already true for $L=2$ (see Lemma \ref{lem:alphakpointwise}), w}e assume that, for $L\geq 2$,
\begin{equation} \label{eq_rec}
  \sum_{k=2}^L |\alpha_k|^2\,\leq C\e^4,\qquad \text{for $\e\leq \e_0$, with $\e_0$ small enough},
\end{equation}
  with $C$ and $\e_0$ to be chosen later. We then prove that
  $$
    \sum_{k=2}^{L+1} |\alpha_k|^2\,\leq C\e^4.
    $$
    For $k=2,\cdots L+1$, we multiply \fer{eq-alpha} by $\alpha_k$ and sum over $k$ to obtain
$$
\begin{aligned}
    \sum_{k=2}^{L+1} \left( \dfrac{1}{2^{k-1}}  - 1  +(H_0,  m_{\varepsilon} H_0)_G   \right) \alpha_k^2 
    &=\sum_{k=2}^{L+1} \alpha_k\left( \sum_{l=2}^{\infty} \alpha_l H_l ,  m_{\varepsilon}  H_k\right)_G\\
    &+\sum_{k=2}^{L+1} \alpha_km_{k}^{(\varepsilon)}-\sum_{k=2}^{L+1}\f{\sqrt{k!}\alpha_k}{2^k} \sum_{l=1}^{k-1} \dfrac{\alpha_l \alpha_{k-l}}{\sqrt{l! (k-l)!}}\\
    &-\sum_{k=2}^{L+1}\sum_{l=1}^{\infty} \alpha_k^2  \alpha_l  m_l^{(\varepsilon)}+\sum_{k=2}^{L+1} \alpha_k\left(   H_1 ,  m_{\varepsilon}  H_k\right)_G \alpha_1.
    \end{aligned}
    $$
 We control the terms separately. {On the left-hand side, we have:}
  $$
  \sum_{k=2}^{L+1} \left( \dfrac{1}{2^{k-1}}  - 1  +(H_0,  m_{\varepsilon} H_0)_G   \right) \alpha_k^2\leq   \big(-\f{1}{2}+(H_0,  m_{\varepsilon} H_0)_G  \big)\sum_{k=2}^{L+1}\alpha_k^2.
  $$
  Moreover,   thanks to Assumptions \fer{As:m}, \fer{As:m2} and \fer{T1} we have
  $$
  \int_\R m_\e(x) G(x)dx\leq C_m\e^2, 
  $$
and hence
   $$
  \sum_{k=2}^{L+1} \left( \dfrac{1}{2^{k-1}}  - 1  +(H_0,  m_{\varepsilon} H_0)_G   \right) \alpha_k^2\leq   \big(-\f{1}{2}+C_m\e^2  \big)\sum_{k=2}^{L+1}\alpha_k^2.
  $$
{We next split the right-hand side $RHS_1 + RHS_2+RHS_3 + RHS_4 + RHS_5$.
Concerning the first term, we have:}
 $$
 \begin{aligned}
| RHS_1| & \leq  \big|\sum_{k=2}^{L+1} \left( \sum_{l=2}^{\infty} \alpha_l H_l ,  m_{\varepsilon}  H_k\right)_G \alpha_k\big| \\
&= {\Big|} \int_\R m_\e(x) \left( \sum_{l=2}^{+\infty} \alpha_l H_l(x)\right)\left( \sum_{k=2}^{L+1} \alpha_k H_k(x)\right)G(x)dx{\Big|}  \\
&\leq \left( \int_\R m_\e^2(x)  \left( \sum_{l=2}^{+\infty} \alpha_l H_l(x)\right)^2G(x)dx\right)^{1/2}\left( \int_\R \left( \sum_{k=2}^{L+1} \alpha_k H_k\right)^2G(x)dx\right)^{1/2}\\
&= \left( \int_\R m_\e^2(x) \left( \sum_{l=2}^{+\infty} \alpha_l H_l(x)\right)^2G(x)dx\right)^{1/2}  \left(\sum_{k=2}^{L+1}\alpha_k^2\right)^{1/2}.
\end{aligned}
 $$
 We  notice that, using that $m''$ is uniformly bounded,
 $$
 \begin{aligned}
  \int_\R m_\e^2(x)  \left( \sum_{l=2}^{+\infty} \alpha_l H_l(x) \right)^2G(x)dx 
 & \leq C_3\e^4\int_\R x^4\left( \sum_{l=2}^{+\infty} \alpha_l H_l(x) \right)^2G(x)dx \\
 & \leq C_3\e^4\int_\R x^4\left( \f{\overline{N}_\e(x)}{G(x)}-\alpha_0H_0(x)-\alpha_1 H_1(x) \right)^2G(x)dx.
\end{aligned} 
 $$
We are now in a position to apply Lemma \ref{N-L2-bound}. We obtain that there exists a constant $B_1$ independent of $\e$ such that 
 $$
   \int_\R m_\e^2(x)  \left( \sum_{l=2}^{+\infty} \alpha_l H_l(x) \right)^2G(x)dx \leq B_1^2\e^4,
 $$
 and hence 
 $$
 |RHS_1| \leq  B_1\e^2\left(\sum_{k=2}^{L+1}\alpha_k^2\right)^{1/2}.
 $$
 We also compute, using Lemma \ref{lem_m},
 $$
 |RHS_2| \leq \left|\sum_{k=2}^{L+1} m_{k}^{(\varepsilon)}\alpha_k \right|\leq  \left(\sum_{k=2}^{L+1}\alpha_k^2\right)^{1/2}\left(\sum_{k=2}^{L+1} (m_{\varepsilon}, H_k)_G^2\right)^{1/2}\leq K_m\e^2{\left(\sum_{k=2}^{L+1}\alpha_k^2\right)^{1/2}.}
 $$
 We leave aside $RHS_3$ and we next compute
 $$
|RHS_4| ={\Big|} \sum_{k=2}^{L+1} \alpha_k^2 \, (m_\e,\sum_{l=1}^{\infty}\alpha_ l H_l)_G{\Big|} 
 =\big(\sum_{k=2}^{L+1}  \alpha_k^2\big) {\Big|} \int_\R m_\e(x)\big(\overline{N}_\e(x)-G(x))dx{\Big|} .
 $$
Using again Assumptions \fer{As:m}, \fer{As:m2} and \fer{T1} we have
 \[
  \int_\R m_\e(x)G(x)dx\leq C_m \e^2,\qquad \int_\R m_\e(x)\overline{N}_\e(x)dx\leq C_m\e^2\int_\R x^2 \overline{N}_\e(x)dx.
\]
 We deduce that there exists a constant $B_2$ such that
\[
|RHS_4| \leq B_2\e^2 \sum_{k=2}^{L+1} \alpha_k^2 .
\]
 For the last term of the r.h.s. we have
\[
 \begin{aligned}
| RHS_5| & \leq   | \alpha_1| \left|\sum_{k=2}^{L+1}  \left(   H_1 ,  m_{\varepsilon}  H_k\right)_G\alpha_k\right |\\
 &\leq   | \alpha_1|\left(\sum_{k=2}^{L+1}\alpha_k^2\right)^{1/2}\left(\sum_{k=2}^{L+1} \left(   H_1 m_{\varepsilon},    H_k\right)_G^2\right)^{1/2}\\
  &\leq  | \alpha_1|\left(\sum_{k=2}^{+\infty}\alpha_k^2\right)^{1/2} \| m_\e H_1\|_{L^2(\R,G(x)dx)}\\
  &\leq K_1C_1 \e^3 \left(\sum_{k=2}^{+\infty}\alpha_k^2\right)^{1/2},
  \end{aligned}
\]
 where we have used Lemma \ref{lem_normmHk}. 
 It remains only to control the cubic term $RHS_3$. This is the only term for the control of which we will use the induction assumption.  We have indeed \eqref{eq_rec} and consequently, for all $k$ such that $2\leq k\leq L$ we have $|\alpha_k|\leq C\e^2$. We also recall that $|\alpha_1|\leq C_1\e$. Enforcing that $C\geq C_1$ without restriction, we compute that:
 $$
 \begin{aligned}
| RHS_3| & \leq  \left|\sum_{k=2}^{L+1}\f{\sqrt{k!}\alpha_k}{2^k} \sum_{l=1}^{k-1} \dfrac{\alpha_l \alpha_{k-l}}{\sqrt{l! (k-l)!}}\right|  \\
 & \leq \f{1}{2\sqrt 2}|\alpha_2|\alpha_1^2 +C^2\e^3\sum_{k=3}^{L+1}\f{ { |} \alpha_k { |}}{2^k} \sum_{l=1}^{k-1} \sqrt{\begin{pmatrix}  k \\  l  \end{pmatrix}}\\
 & \leq \f{1}{2\sqrt 2}|\alpha_2|\alpha_1^2+C^2\e^3\sum_{k=3}^{L+1}\f{ \sqrt{k}\,  { |}\alpha_k { |}}{2^k}
 \left( \sum_{l=1}^{k-1}\begin{pmatrix}  k \\  l  \end{pmatrix}\right)^{1/2}\\
 &\leq \f{1}{2\sqrt 2}|\alpha_2|\alpha_1^2+C^2\e^3\sum_{k=3}^{L+1}\f{ \sqrt{k}\,  { |}\alpha_k { |}}{2^{k/2}}\\
 &\leq  \f{C_1^2}{2\sqrt 2}|\alpha_2|\e^2+C^2\e^3\left( \sum_{k=3}^{L+1}\alpha_k^2\right)^{1/2}
 \left( \sum_{k=3}^{+\infty} \f{ k}{2^{k}}\right)^{1/2}\\
 &\leq \f{ B_3}{\sqrt{2}}|\alpha_2|\e^2+C^2\f{B_3}{\sqrt{2}}\e^3\left( \sum_{k=3}^{L+1}\alpha_k^2\right)^{1/2}.
 \end{aligned}
 $$
 We next choose $\e_0$ such that $C^2\e_0\leq 1$. We deduce that for all $\e\leq \e_0$,
  $$
 |RHS_3| \leq 
   \f{B_3}{\sqrt{2}}\e^2\left(|\alpha_2| +\left( \sum_{k=3}^{L+1}\alpha_k^2\right)^{1/2}\right)\leq B_3\e^2 \left( \sum_{k=2}^{L+1}\alpha_k^2\right)^{1/2}.
  $$
  Combining the inequalities above we deduce that
  $$
\left(\f 12-(C_m+B_2)\e^2\right) \left(\sum_{k=2}^{L+1}\alpha_k^2\right)^{1/2}\leq \big( B_1+K_m +K_1C_1\e+B_3 \big)\e^2.
  $$
 Choosing $C$ large enough { -- wrt constants $B_1,K_m$ and $B_3$ that are independant of $L$ --} and $\e_0$ small enough { -- wrt $C_m,B_2,K_1,C_1$ that are also independant of $L$ --}  we obtain the result.


\section{Proof of Theorem \ref{thm_stat}}
We look now for a solution to {\eqref{eq_Nepsform_steady}} in the form ${\overline N_{\e}} = h_{\e}G$ 
where $h_{\e} \in L^2(\mathbb R,G(x){\rm d}x)$ and $\varepsilon$ is sufficiently small.
{Note that ultimately we are interested in a positive solution $N_\e$ corresponding to a population density. However, for the moment we focus on the proof of Theorem \ref{thm_stat} which does not require a positivity condition. The positivity of the steady solution will follow from Theorem \ref{thm_longtime}.}

 We recall that in Section \ref{sec:Hermite}, we transformed \eqref{eq_Nepsform_steady} into a problem on $\ell^2(\mathbb N)$ by writing:
\[
h = \sum_{k=0}^{\infty} \alpha_k H_k.
\]
Noticing that {$\overline{N}_{\e}$ being of unit mass} translates into $\alpha_0= 1,$  we infered the infinite dimensional system of equations \eqref{eq-alpha} in terms of $(\alpha_{k})_{k \geq 1}.$ 
We want to construct a unique solution to this system such that:
 \[
|\alpha_1| \leq C\e \qquad |\overline{\alpha}_2| \leq C \e^2,
 \] 
 with a $C$ sufficiently large and an $\e$ sufficiently small. Below we assume $C$ and $\e$ 
 are given and we only point out restrictions on these quantities for our result to hold true. 
 To avoid incompatibility requirements, we introduce other notations $K_m$ for a constant depending on $m$ and $C$ and $K$ for a constant independent of relevant parameter.

\medskip

To prepare computations we provide a preliminary analysis of \eqref{eq-alpha} in order to design a resolution method.
Indeed, in case $k=1,$ we have $2^{k-1} - 1 = 0.$ Hence, the factor of $\alpha_1$ in the left-hand side of \eqref{eq-alpha} is $(H_0,m_{\e}H_0)$ that is of order $\varepsilon^2,$ while it will remain of order $1$ for $k\geq 2.$ It is then not possible to use this term as a pivot to compute $\alpha_1$ while it is possible for $\alpha_k$ ($k \geq 2$). This remark motivates the following splitting to solve \eqref{eq-alpha}.   
 
 \medskip
 
Firstly, assuming $\overline{\alpha}_2$ is given, we infer that $\alpha_1$ solves:
 \begin{multline} \label{eq-alpha1}
  m_{1}^{(\varepsilon)} \alpha_1^2 +  \left[(H_0, m_{\varepsilon} H_0)_G   - (H_1,m_{\varepsilon} H_1)_G + \sum_{l=2}^{\infty} \alpha_l m_l^{(\varepsilon)}\right]  \alpha_1 
  \\ = m_1^{(\varepsilon)} + \sum_{l=2}^{\infty} \alpha_l \left( H_l ,  m_{\varepsilon}  H_1\right)_G .
\end{multline}
{ 
A central technical result in our analysis then reads:
\begin{lemma} \label{lem_alpha1}
Let $C >0,$ there exists a  constant $K_m$ depending only on $m$ such that, for $\varepsilon$ sufficiently small (depending on $C$) the following properties hold true: 
\begin{itemize}
\item[(i)] Given $\overline{\alpha}_2 \in B_{\ell^2(\mathbb N_2)}(0,C\varepsilon^2)$  there is a unique  $\alpha_1 \in  (-K_{m}\varepsilon,K_m\varepsilon)$ solution to \eqref{eq-alpha1}.
\item[(ii)] The induced mapping $\overline{\alpha}_2 \to \alpha_1[\overline{\alpha}_2]$ that is defined on $B_{\ell^2(\mathbb N_2)}(0,C\varepsilon^2)$ is $K_{m}$-Lipschitz.
\end{itemize}
\end{lemma}

We remark that, in terms of $\alpha_1,$ the equation \eqref{eq-alpha1} is quadratic if $m_{1}^{(\varepsilon)} \neq 0$ while it degenerates into a linear equation when $m_{1}^{(\varepsilon)} = 0.$ The linear case   arises when $m$ is even which we assume in our analysis when $m$ has a maximum in $0.$ However, in this latter context, our approach restricts to even solutions for which $\alpha_1$ vanishes by assumption.  We postpone the proof of Lemma \ref{lem_alpha1} to a further subsection.
}
\medskip

\medskip

Secondly, assuming  that $\alpha_1$ is given, the system of equations \eqref{eq-alpha} for $k \geq 2$ reads: 
\[
\overline{\mathcal L}[\overline{\alpha}_2] = \overline{m}_2^{(\varepsilon)} - \overline{\mathcal Q}[\overline{\alpha}_2]
\]
where,  $\overline{\mathcal L}$ and $\overline{\mathcal Q}$ are defined by
\[
\begin{aligned}
\mathcal L_k[\overline{\alpha}_2] & :=  \left( \dfrac{1}{2^{k-1}}  - 1  + (H_0, m_{\varepsilon} H_0)_G  \right) \alpha_k  - \left( \sum_{l=2}^{\infty} \alpha_l H_l ,  m_{\varepsilon} H_k\right)_G \\
\mathcal{Q}_k[\overline{\alpha}_2] & := \dfrac{\sqrt{k!}}{2^k}   \sum_{l=1}^{k-1} \dfrac{\alpha_l \alpha_{k-l}}{\sqrt{l! (k-l)!}} + \alpha_k \sum_{l=1}^{\infty} \alpha_l (H_l,m_{\varepsilon}H_0)_{G}  \\
& \qquad -  \alpha_1(H_1,m_{\varepsilon}H_k)_G.
\end{aligned}
\]
for $k \geq 2.$ {Our rewriting of the set of equations for $k \geq 2$ relies on the remark that we have in our system source terms depending on $\overline{m}^{(\varepsilon)},$ linear terms in $\overline{\alpha}_2$ and bilinear terms in $\overline{\alpha}_1.$ We point out that, despite $\alpha_1$ is {\em a priori} a fixed quantity so that the associated terms could be considered as source terms or linear terms, we treat them as nonlinearities since {in} what follows, we assume that $\alpha_1 = \alpha_1[\overline{\alpha}_2].$}
 
 For legibility, we dropped the $\varepsilon$-dependencies for 
$\overline{\mathcal L},$ $\overline{\mathcal Q}.$ 
Key results concerning these mappings are the following two lemmas. Concerning $\mathcal L,$ we have:
\begin{lemma} \label{lem_invL}
Let $\varepsilon$ be sufficiently small.  For  arbitrary $\overline{Q}_2 \in \ell^2(\mathbb N_2)$ there is a unique $\overline{\alpha}_2 {:= {\cal L}^{-1}[\overline{Q}_2]} \in \ell^2(\mathbb N_2)$ satisfying
\begin{itemize}
\item  $\overline{\mathcal L}[\overline{\alpha}_2] = \overline{Q}_2$
\item $|\overline{\alpha}_2| \leq 2^{k_0+2} |\overline{Q}_2|$ {with $k_0 \in \mathbb N$ to be made precise depending on $m$ only.}
\end{itemize} 
\end{lemma} 
{Assuming that $\alpha_1 = \alpha_1 [\overline{\alpha}_2]$ in the formulas defining $\mathcal Q,$} we obtain:
\begin{lemma} \label{lem_Q}
There exists a constant $K_m$
depending only on $m$ such that,  for $\varepsilon$ sufficiently small, depending on $C,$ the mapping $\overline{\mathcal Q}[\cdot]$   is well-defined on $B_{\ell^2(\mathbb N_2)}(0,C\e^2)$ with 
\begin{itemize}
\item $\overline{\mathcal Q}[\bar{0}]| \leq  K_m \varepsilon^3 $
\item $\overline{\alpha}_1 \mapsto \overline{\mathcal Q}[\overline{\alpha}_1]$ is $K_m \varepsilon$-Lipschitz on $B_{\ell^2(\mathbb N_2)}(0,C\e^2)$.
\end{itemize}
\end{lemma}

Again the proofs of these technical results are postponed to a further subsection. We explain at first how they 
entail Theorem \ref{thm_stat}.

\subsection{Proof of Theorem \ref{thm_stat}.}

Let us denote by the same symbol $K_m$ the max of the $K_m$ constructed in Lemma \ref{lem_Q}, Lemma \ref{lem_alpha1} and Lemma \ref{lem_m}. Let then ${\overline C} = 2^{k_0+3} K_m$ and fix {$C> \overline{C},$ $\varepsilon$ sufficiently small} so that {Lemma \ref{lem_invL}} and Lemma \ref{lem_Q} hold true.

\medskip

{\it Existence.} We obtain existence of a solution by a fixed-point argument. Let $\overline{\mathcal T} : B_{\ell^2(\mathbb N_2)}(0,C \varepsilon^2) \to \ell^2(\mathbb N_2)$ be defined by:
\[
\overline{\mathcal T}[\overline{\alpha}_2] =\overline{\mathcal L}^{-1}[\overline{m}_2^{(\varepsilon)} - \overline{\mathcal Q}[\overline{\alpha}_2]]
\]
Let  $\overline{\alpha}_2 \in B_{\ell^2(\mathbb N_2)}(0,C\e^2).$ Applying Lemma \ref{lem_Q}, we have
\[
|\overline{\mathcal Q}[\overline{\alpha}_1]]| \leq {K_m (C+1)\e^3}  ,
\]
while Lemma \ref{lem_m} ensures that $|\overline{m}^{(\varepsilon)}| \leq {K_m} \varepsilon^2.$
Combining with {Lemma \ref{lem_invL}}, we obtain that:
\[
|\overline{\mathcal T}[\overline{\alpha}_2]| \leq 2^{k_0+2} \left({ K_m \varepsilon^2 + K_m (C+1)\e^3 }\right).
\]
Consequently, choosing $\e$ sufficiently small, we infer that $\overline{\mathcal T}$ maps $B_{\ell^2(\mathbb N_2)}(0,C \varepsilon^2)$ into itself. 
  
 \medskip
  
Furthermore,  given $\overline{\beta}^{(1)} $ and $\overline{\beta}^{(2)}$  in $B_{\ell^2(\mathbb N_2)}(0,C \varepsilon^2),$ applying successively the Lipschitz properties of $\overline{\mathcal Q}$ and the boundedness of {the right inverse $\mathcal L^{-1},$} we infer:
\[
\| \overline{\mathcal T}[\bar{\beta}^{(2)}] - \overline{\mathcal T}[\bar{\beta}^{(1)}]  \|_{\ell^2(\mathbb N_1)} \leq 2^{k_0+2} K_m \varepsilon  \|\bar{\beta}^{(1)} -\bar{\beta}^{(2)}\|_{\ell^2(\mathbb N_1)}.
\]
In particular, $\overline{\mathcal T}$ is a contraction up to restricting the size of $\varepsilon.$
We conclude that $\overline{\mathcal T}$ admits a unique fixed-point on $B_{\ell^2(\mathbb N_2)}(0,C \varepsilon^2)$ that we denote $\overline{\alpha}^{(\e)}_1 = (\alpha_{k}^{(\e)})_{k\geq 2}.$ 
We denote the corresponding $\alpha_1^{(\e)} = \alpha_1[\overline{\alpha}_2^{\e}].$

\medskip

We set now:
\[
h_{\e} = H_0 + \sum^{\infty}_{k=1} \alpha_k^{(\e)} H_k 
\]
By construction $h_{\e} \in L^2(\mathbb R, G(x)dx).$ Furthermore, for arbitrary $k \geq 0$ we can define:
\[
{\gamma_ k} = (m_{\e}h_\e, H_k)_G = (H_0,m_{\e} H_k)_G + \sum_{l=1}^{\infty} \alpha_l^{(\e)}(H_l,m_{\e} H_k)_{G}\,, 
\]
We can then recast the $k$-th equation of \eqref{eq-alpha} into:
\begin{equation} \label{eq_betak}
{\gamma_k} = \left(\tilde{T}[h_{\e}] - \left( 1 - \int_{\mathbb R} m_{\e}(x)h_{\e}(x)G(x) dx\right) h_{\e}  , H_k\right)_G  \qquad \forall \, k \geq 1.
\end{equation}
{In particular, thanks to  Lemma \ref{lem:T-cont},  we have $({\gamma_k})_{k \in \mathbb N} \in \ell^2(\mathbb N),$ and then $\sum_{k} {\gamma_k} {H_k} \in L^2(\mathbb R,G(x) dx).$ 
Let denote $\pi = {(\sum_{k} {\gamma_k} H_k - m_{\e} h_\e)G}.$ Since $m_{\e}$ grows polynomially, we have that  $x \mapsto \pi(x) \exp({\eta} |x|) \in L^1(\mathbb R)$ for all ${\eta} >0$ so that the Fourier transform $\hat{\pi}$ of $\pi$ {is} holomorphic on $\mathbb C.$ By construction, all derivatives of $\hat{\pi}$ vanish in the origin showing that $\hat{\pi}$ and thus $\pi$ vanish. Identifying Hermite coefficients, we conclude that:
\[
(\tilde{T}[h_{\e}](x) - h_{\e}(x)) - \left(m_{\e}(x) - \int_{\mathbb R} m_{\e}h_{\e} G(x)dx \right) h_{\e}(x)= 0 \qquad { \forall \, x \in \mathbb R}.
\]
}  

\medskip

{\it Uniqueness.} Let $N_{\e}$ be a solution to \eqref{eq_Nepsform} in $\mathcal U^{(\e)}(C).$ We can then construct $h_{\e}$ and the associated $\overline{\alpha}^{(\e)} \in \ell^2(\mathbb N).$ By assumption 
we have that $\overline{\alpha}_2^{(\e)} \in B_{\ell^2(\mathbb N_2)}(0,C\e^2).$ Multiplying the equation for $h_{\e}$ with $H_k,$ we infer that $\overline{\alpha}^{(\e)}$ satisfies \eqref{eq-alpha} so that 
$\overline{\alpha}_1^{(\e)}$ is indeed a fixed point of the above $\overline{\mathcal T}$ that is a contraction. This ends the proof.  

\medskip

In the following subsections, we give proofs for the technical lemmas : Lemma \ref{lem_invL}, Lemma 
\ref{lem_alpha1} and Lemma \ref{lem_Q}. 

\bigskip

\subsection{Proof of Lemma \ref{lem_invL}}

Let $\overline{Q}_2 \in \ell^2(\mathbb N_2)$ and remark that we are interested in solving the system:
\begin{equation} \label{eq_L}
\left( \dfrac{1}{2^{k-1}}  - 1  +(H_0, m_{\varepsilon} H_0)_G   \right) \alpha_k 
 - \left( \sum_{l=2}^{\infty} \alpha_l H_l ,  m_{\varepsilon}  H_k\right)_G  = Q_k
 \quad \forall \, k \geq 2,
\end{equation}
where the unknown is $\overline{\alpha}_2 \in \ell^2(\mathbb N_2).$

\begin{remark}
{
In this system,  one could expect to approximate the left-hand side by the diagonal operator with coefficients $(1/2^{k-1}-1)_{k\geq 2}$ that is  the operator obtained by deleting the $\varepsilon$-dependent terms. However, we lack a uniform bound of the perturbation in terms of $\varepsilon.$ We should then use this remark for small $k$ only and use 
that the second part of the left-hand side corresponds to the operator 
$h \mapsto m_{\varepsilon} h$ that is controlled by assumption \eqref{as:m-extremum}. Thanks to this assumption, we can introduce a number $k_0 \in \mathbb N$ (depending on $m$ but not on $\varepsilon$) such that:
\begin{equation} \label{eq_defk0}
m_{\varepsilon} + 1 \geq \dfrac{1}{2^{k_0+1}} \quad \forall \, \e >0.
\end{equation}
}
\end{remark}

{\em Uniqueness.} Firstly, we investigate uniqueness of a $\overline{\alpha}_2 \in \ell^2(\mathbb N_2)$ such that $\mathcal L[\overline{\alpha}_2] = \overline{Q}_2.$ By difference, this amounts to show the unique solution with $\overline{Q}_2 = 0$ is $\overline{\alpha}_2 = 0.$ 

So let's suppose $\overline{\alpha}_2 \in \ell^2(\mathbb N_2)$ is a solution with $\overline{Q}_2 = 0.$ Multiplying by $\alpha_k$ each $k$-equation and summing over $k,$ we obtain:
\begin{equation} \label{eq_uniq0}
\sum_{k=2}^{\infty} \left( \dfrac{1}{2^{k-1}}  - 1  +(H_0, m_{\varepsilon} H_0)_G   \right) \alpha_k^2 
 - \sum_{k=2}^{\infty}\left( \sum_{l=2}^{\infty} \alpha_l H_l ,  m_{\varepsilon}  H_k\right)_G\alpha_k  = 0. 
\end{equation}
Concerning the second term, we have:
\begin{multline*}
\sum_{k=2}^{\infty}\left( \sum_{l=2}^{\infty} \alpha_l H_l ,  m_{\varepsilon}  H_k\right)_G\alpha_k \\
\begin{aligned}
& = 
\int_{\mathbb R} m_{\e}(x) |\sum_{l=2}^{\infty} \alpha_l H_l(x)|^2 G(x) dx   \\
& =  \sum_{l=2}^{{k_0+1}}\sum_{l'=2}^{{k_0+1}} \alpha_l \alpha_{l'} \int_{\mathbb R} m_{\e}  H_l(x) H_{l'}(x) G(x) dx  + 2\sum_{l=2}^{{k_0+1}}\sum_{l'={k_0+2}}^{\infty} \alpha_l \alpha_{l'} \int_{\mathbb R} m_{\e}  H_l(x) H_{l'}(x) G(x) dx  \\
& \quad  + \int_{\mathbb R} m_{\e}(x) | \sum_{l={k_0+2}}^{\infty} \alpha_l H_l(x)|^2 G(x) dx ,
\end{aligned}
\end{multline*}
where
\[
\int_{\mathbb R} m_{\e} (x)| \sum_{l={k_0+2}}^{\infty} \alpha_l H_l(x)|^2 G(x) dx  \geq  - \left( 1 - \dfrac{1}{2^{ k_0+1}} \right) \sum_{l={k_0+2}}^{\infty} |\alpha_l|^2.
\]
{Combining} a Cauchy Schwarz inequality and Lemma \ref{lem_m}
 we have:
\begin{multline*}
\sum_{l=2}^{{k_0+1}}\sum_{l'={k_0+2}}^{\infty} \alpha_l \alpha_{l'}  \int_{\mathbb R} m_{\e}(x)  H_l(x) H_{l'}(x) G(x) dx
\\
\begin{aligned}
& \geq -  \left( \int_{\mathbb R} m_{\e}^{2}(x) |\sum_{l=2}^{{k_0+1}} \alpha_l H_l(x)|^2 G(x) dx  \right)^{\frac 12} \left( \sum_{l={k_0+2}}^{\infty} |\alpha_l|^2 \right)^{\frac 12}\\
& \geq -  \e^2 K_m \left( \sum_{l=2}^{{k_0+1}} |\alpha_l|^2  \right)^\frac 12  \left( \sum_{l={k_0+2}}^{\infty} |\alpha_l|^2 \right)^{\frac 12} \\
& \geq - \e^2 \frac{K_m}{2} \sum_{l=2}^{{k_0+1}} |\alpha_l|^2 - \e^2 \frac{K_m}{2} \sum_{l={k_0+2}}^{\infty} |\alpha_l|^2. 
\end{aligned}
\end{multline*}
with a constant $K_m$ depending only on $m$ (also through $k_0$). {Similarly, we obtain}
\begin{multline*}
\sum_{l=2}^{{k_0+1}}\sum_{l'=2}^{{k_0+1}} \alpha_l \alpha_{l'} \int_{\mathbb R} m_{\e}(x)  H_l(x) H_{l'}(x) G(x) dx\\
\begin{aligned}
& \geq -  \left( \int_{\mathbb R} m_{\e}(x)^{2} |\sum_{l=2}^{{k_0+1}} \alpha_l H_l(x)|^2 G(x) dx  \right)^{\frac 12} \left( \sum_{l=2}^{{k_0+1}} |\alpha_l|^2 \right)^{\frac 12}\\
& \geq -  \e^2 K_m \left( \sum_{l=2}^{{k_0+1}} |\alpha_l|^2  \right). 
\end{aligned}
\end{multline*}
Eventually, we conclude that
\begin{equation} \label{eq_uniq1}
- \sum_{k=2}^{\infty}\left( \sum_{l=2}^{\infty} \alpha_l H_l ,  m_{\varepsilon}  H_k\right)_G\alpha_k
\leq \left( 1 + K_{m} \e^2 - \dfrac{1}{2^{ k_0+1}} \right) \sum_{l={k_0+2}}^{\infty} |\alpha_l|^2 +{2} K_m \e^2 \sum_{l=2}^{{k_0+1}} |\alpha_l|^2.
\end{equation}
As for the first term {of \fer{eq_uniq0}}, we obtain similarly:
\begin{multline} \label{eq_uniq2}
\sum_{k=2}^{\infty} \left( \dfrac{1}{2^{k-1}}  - 1  +(H_0, m_{\varepsilon} H_0)_G   \right) \alpha_k^2  \\ 
\leq - \left(\dfrac{1}{2} - K_m \e^2  \right) \sum_{k=2}^{{k_0+1}} |\alpha_l|^2 - \left( 1 - \dfrac{1}{2^{{k_0+2}}} - K_m \e^2 \right) \sum_{k={k_0+2}}^{\infty} |\alpha_l|^2.
\end{multline}
Plugging \eqref{eq_uniq1} and \eqref{eq_uniq2} into \eqref{eq_uniq0}, taking small $\e,$ we infer that : 
\[
\dfrac{1}{4} \sum_{l=2}^{{k_0+1}} |\alpha_l|^2  + \dfrac{1}{4}\dfrac{1}{2^{{k_0+1}}} \sum_{l={k_0+2}}^{\infty}  |\alpha_l|^2  \leq 0,
\]
that is $\bar{\alpha}_2 = 0$

\medskip

{\em Existence.} We proceed now with constructing a candidate $\overline{\alpha}_2 = \mathcal L^{-1}\overline{Q}_2$ by finite-rank approximations. Let $K \geq k_0+1$ {be} arbitrary large. We look for $\overline{\alpha}_1^{(K)}
 = (\alpha^{(K)}_2,\ldots,\alpha^{(K)}_K) \in \mathbb R^{K-1}$ solving:
 \begin{equation} \label{eq_LK}
\left( \dfrac{1}{2^{k-1}}  - 1  +(H_0, m_{\varepsilon} H_0)_G   \right) \alpha^{(K)}_k 
 - \left( \sum_{l=2}^{K} \alpha^{(K)}_l H_l ,  m_{\varepsilon}  H_k\right)_G  = Q_k
 \quad \forall \, k = 2,\ldots,K,
\end{equation}
This turns out te be an invertible linear system. Indeed, let {us} denote by $\mathbb M^{(K)}$ the matrix 
implicitly involved by this system. For any $\overline{\beta} = (\beta_1,\ldots,\beta_k)$ we have
\[
\begin{aligned}
\overline{\beta} \cdot \mathbb M^{(K)} \cdot \overline{\beta}^{\top} = 
\sum_{k=2}^{K} \left[ \left( \dfrac{1}{2^{k-1}}  - 1  +(H_0, m_{\varepsilon} H_0)_G   \right) \beta_k 
 - \left( \sum_{l=2}^{K} \beta_l H_l ,  m_{\varepsilon}  H_k\right)_G \right] \beta_k
\end{aligned}
\]   
The restrictions on $\varepsilon$ required in the previous uniqueness proof entail with similar computations that (when $K \geq k_0$)
\[
\overline{\beta} \cdot \mathbb M^{(K)} \cdot \overline{\beta}^{\top}  \geq \dfrac{1}{4} \sum_{l=2}^{{k_0+1}} |\beta_l|^2 + \dfrac{1}{2^{k_0+2}} \sum_{l={k_0+2}}^{K} |\beta_l|^2.
\]
This entails that $Ker(\mathbb M^{(K)}) = \{0\}.$ We recover that there is a unique $\overline{\alpha}_2^{(K)}$ solution to \eqref{eq_LK}. Below we denote also $\overline{\alpha}_2^{(K)}$ the trivial extension of this solution and that is an element of $\ell^{2}(\mathbb N_2).$

\medskip

For arbitrary $K \geq k_0+2$, we observe again that (thanks again to our previous restrictions on $\varepsilon$):
\[
\dfrac{1}{2^{k_0+2}} |\overline{\alpha}_{2}^{(K)} |^2  \leq \overline{\alpha}_{2}^{(K)} \cdot \mathbb M^{(K)} \cdot \overline{\alpha}_{2}^{(K)} = \sum_{k=2}^{K} Q_k \alpha_k^{(K)}
\]
and then by a Cauchy-Schwarz ineuquality:
\[
|\overline{\alpha}_{2}^{(K)} | \leq 2^{k_0+2} |\overline{Q}_2| \qquad \forall \, K \geq k_0+2.
\]
We can then extract a weakly converging limit $\overline{\alpha}_2.$ Since the left-hand side of 
\eqref{eq_L} is continuous for the weak topology for arbitrary $k \geq 2,$ we obtain that $\overline{\alpha}_2$ is a candidate  to be $\mathcal L^{-1} \overline{Q}_2.$ Furthemore, it satisfies the required control 
\[
|\overline{\alpha}_2| \leq 2^{k_0+2} |\overline{Q}_2|.
\]
This ends the proof.

\subsection{Proofs of Lemmas \ref{lem_alpha1} and \ref{lem_Q}}

We analyze successively {the $\alpha_1$ equation and the mapping $\overline{\mathcal Q}.$ This order is chosen}
since $\alpha_1$ is involved in the definition of $\overline{\mathcal Q}$. 

\medskip

\begin{proof}[Proof of Lemma \ref{lem_alpha1}]

{As indicated above, we} consider the case $m_{1}^{(\varepsilon)} \neq 0$ only.  The case $m_{1}^{(\varepsilon)} =0$ yields with similar (but simpler) arguments.

\medskip

{{\em Proof of item (i).} Let  $\overline{\beta}_2 \in \ell^2({\mathbb N_2})$ with ${\|\overline{\beta}_2\|_{\ell^2(\mathbb N_2)}} \leq C \e^2$ and set 
\[
\begin{aligned}
D^{(\varepsilon)}&  := (H_0,  m_{\varepsilon}  H_0)_G   - (H_1,m_{\varepsilon} H_1)_G, \\
m_{\beta} & :=({\overline{\beta}_2}, \overline{m}^{(\varepsilon)})_{G} = \sum_{l=2}^{\infty} \beta_l m_{l}^{(\varepsilon)} ,\\
 \tilde{m}_{\beta} & := (\sum_{l=2}^{\infty} \beta_l H_l ,m_{\varepsilon}H_1)_{G}.
\end{aligned}
\]
With such notations equation \eqref{eq-alpha1} with associated with $\beta_2$ reads:
\begin{equation} \label{eq_alpha12}
m_{1}^{(\varepsilon)} \alpha_1^{2} + \left( D^{(\varepsilon)} + m_{\beta} \right)\alpha_1 = m_{1}^{(\varepsilon)} + \tilde{m}_{\beta}.
\end{equation}
In this equation, we notice that, by a standard 
integral convergence argument, we have that:
\begin{equation} \label{eq_Deps}
\begin{aligned}
D^{(\varepsilon)}&  =  \int_{\mathbb R} (1-|H_1(x)|^2) (m_{\varepsilon} - m_{\varepsilon}(0)) G(x){\rm d}x \\
& =   \varepsilon^2 \dfrac{m''(0)}{2} \int_{\mathbb R} ( 1- |H_1(x)|^2)x^2 G(x){\rm d}x + {O(\e^3)},   
\end{aligned}
\end{equation}
{where  the coefficient of $\e^2$ above is nonzero thanks to assumption \eqref{as:m-extremum} and since $\int_{\mathbb R} ( 1- |H_1(x)|^2)x^2 G(x){\rm d}x=\int_{\mathbb R} ( 1- x^2)x^2 G(x){\rm d}x\neq 0$.} On the other hand, the analysis of source terms
in Lemma \ref{lem_m} and Lemma \ref{lem_normmHk} entails:
\[
\begin{aligned}
| m_{\beta} |  & \leq \|\overline{\beta}_2\|_{\ell^2({\mathbb N_2})} \|m^{(\varepsilon)} \|_{L^2(\mathbb R,G(x){\rm d}x)} 
\leq K_m\varepsilon^2 \|\overline{\beta}_2\|_{\ell^2({\mathbb N_2})},
\\
|\tilde{m}_{\beta}| &  \leq  \|\overline{\beta}_2\|_{\ell^2({\mathbb N_2})}  \|m^{(\varepsilon)}  H_1 \|_{L^2(\mathbb R,G(x){\rm d}x)} 
\leq  K_m \varepsilon^2 \|\overline{\beta}_2\|_{\ell^2({\mathbb N_2})} ,\\
| m_{1}^{(\varepsilon)} | &  \leq K_m \varepsilon^3.
\end{aligned}
\]
Thus, {using $\|\overline{\beta}_2\|_{\ell^2(\mathbb N_2)}\leq C\e^2$}, we infer   that :
\[
\left[ D^{(\varepsilon)} + m_{\beta} \right]^2 + 4 m_1^{(\varepsilon)} ( m_1^{(\varepsilon)} + \tilde{m}_{\beta}) = \varepsilon^4 \left(\dfrac{m''(0)}{2} \int_{\mathbb R} ( 1- |H_1(x)|^2)x^2 G(x){\rm d}x\right)^{2} + O(\varepsilon^5),
\]
so that \eqref{eq_alpha12} admits two real roots for small $\varepsilon:$
\[
\begin{aligned}
\alpha_1  &=  - \dfrac{D^{(\varepsilon)}+ m_{\beta} }{2m_1^{(\varepsilon)}} 
  + \dfrac{{\Bigl( \left[ D^{(\varepsilon)} + m_{\beta} \right]^2 + 4 m_1^{(\varepsilon)} ( m_1^{(\varepsilon)} + \tilde{m}_{\beta})\Bigr)^{\frac 12}}}{2m_1^{(\varepsilon)}}\\
\tilde{\alpha}_1  &=  - \dfrac{D^{(\varepsilon)}+ m_{\beta} }{2m_1^{(\varepsilon)}} 
  - \dfrac{{\Bigl( \left[ D^{(\varepsilon)} + m_{\beta} \right]^2 + 4 m_1^{(\varepsilon)} ( m_1^{(\varepsilon)} + \tilde{m}_{\beta})\Bigr)^{\frac 12}}}{2m_1^{(\varepsilon)}}. 
\end{aligned}
\]
Introducing 
\[
\Delta_{\beta} := \dfrac{4 (m_1^{(\varepsilon)} ( m_1^{(\varepsilon)} + \tilde{m}_{\beta})) }{\left[ D^{(\varepsilon)} + m_{\beta} \right]^2}
\]
we remark that:
\[
\alpha_1 = \dfrac{(D^{(\varepsilon)} + m_{\beta})}{2m_1^{\varepsilon}} \left( \sqrt{1 + \Delta_{\beta}} - 1 \right)  \qquad 
\tilde{\alpha}_1 ={-} \dfrac{(D^{(\varepsilon)} + m_{\beta})}{2m_1^{\varepsilon}} \left( \sqrt{1 + \Delta_{\beta}} + 1 \right).
\]
With similar arguments as previously, we get that:
\[
(D^{(\varepsilon)} + m_{\beta}) = \dfrac{m^{''}(0)}{2} \int_{\mathbb R} ( 1- |H_1(x)|^2)x^2 G(x){\rm d}x \, \varepsilon^2 + O(\varepsilon^3)  
\]
while $|m_{1}^{\varepsilon}| \leq K_m \e^3$ and $\Delta_{\beta} \leq K_1 \varepsilon^2$ with a constant $K_1$ depending on $C.$
Since  the square root is $1$-lipschitz on $[1/2,3/2],$ we conclude:
\[
|\alpha_1| \leq  \dfrac{|D^{(\varepsilon)} + m_{\beta}|}{2|m_1^{\varepsilon}|} |\Delta_{\beta}| \leq 2\dfrac{  | m_1^{(\varepsilon)} + \tilde{m}_{\beta}| }{|D^{(\varepsilon)} + m_{\beta}| }   \leq K_m \varepsilon.
\]
while 
\[
|\tilde{\alpha}_1| \geq \dfrac{C_m}{\e}.
\]
We have then one possible root that is $\alpha_1.$

{\em Proof of item (ii).} 
}
Let {us} now consider $\overline{\beta}^{(1)}$ and $\overline{\beta}^{(2)}$ in $B_{\ell^2(\mathbb N_2)}(0,C\e^2).$
 We have, with obvious notations (and applying the linearity of $\overline{\beta} \to( m_{\beta},\tilde{m}_{\beta})$)
\[
\begin{aligned}
| \alpha_1^{(1)} - \alpha_1^{(2)}| & \leq
\dfrac{|m_{\beta}^{(1)} - m_{\beta}^{(2)}|}{2m_1^{\varepsilon}} \left( \sqrt{1 + \Delta^{(1)}_{\beta}} - 1 \right)  
+ 
\dfrac{(D^{(\varepsilon)} + m^{(2)}_{\beta})}{2m_1^{(\varepsilon)}} \left( \sqrt{1 + \Delta_{\beta}^{(1)}} - \sqrt{1 + \Delta_{\beta}^{(2)}} \right)   \\
& \leq K_{m} \varepsilon  \|\overline{\beta}^{(1)} - \overline{\beta}^{(2)} \|_{\ell^2(\mathbb N_2)}  +  
\dfrac{(D^{(\varepsilon)} + m^{(2)}_{\beta})}{2m_1^{(\varepsilon)}} |\Delta_{\beta}^{(1)} - \Delta_{\beta}^{(2)}|\\
\end{aligned}
\]
where:
\begin{multline*}
\dfrac{(D^{(\varepsilon)} + m^{(2)}_{\beta})}{2m_1^{(\varepsilon)}} |\Delta_{\beta}^{(1)} - \Delta_{\beta}^{(2)}| \\
\begin{aligned}
& \leq 2 
\left[ \dfrac{ | \tilde{m}_{\beta}^{(1)} - \tilde{m}_{\beta}^{(2)}|}{(D^{(\varepsilon)} + m^{(2)}_{\beta})} 
 + \dfrac{(|m_1^{(\varepsilon)}| + |\tilde{m}_{\beta}^{(2)}|) (2|D^{(\varepsilon)}| + |m_{\beta}^{(1)}| + |m_{\beta}^{(2)}|)}{|D^{(\varepsilon)} + m^{(2)}_{\beta}|(D^{(\varepsilon)} + m^{(1)}_{\beta})^2} |m_{\beta}^{(1)} - m_{\beta}^{(2)}| 
\right]\\
& \leq K_{m} \|\overline{\beta}^{(1)} - \overline{\beta}^{(2)} \|_{\ell^2(\mathbb N_2)} .
\end{aligned}
\end{multline*}
This concludes the proof. 
\qed  \end{proof}

\medskip

With these computations, we analyse now the behavior of $\overline{\mathcal Q}.$  

\medskip

\begin{proof}[Proof of Lemma \ref{lem_Q}]
Firstly, we fix $\varepsilon$ small so that the content of the previous proposition is satisfied.  Since $\overline{\mathcal Q}$ is made of  component multiplications and $\alpha_1,$ the formula defining its components are well-defined on $B_{\ell^2(\mathbb N_2)}(0,C\e^2).$ To prove that the result lies in $\ell^2(\mathbb N_2)$ we restrict to compute the value in $\overline{0}$ and lipschitz-properties of the mapping.  

\medskip

We have:
\[
\mathcal Q_k[\bar{0}] = \alpha_1[\bar{0}] (m_{\varepsilon} H_1, H_k)_G
\]
Hence, combining Lemma \ref{lem_m} and the previous lemma yields:
\[
|\overline{\mathcal Q}[\bar{0}]| \leq 
|\alpha_1[\bar{0}]| \|m_{\varepsilon} H_1\|_{L^2(\mathbb R,G(x){\rm d}x)}  \leq K_m \varepsilon^3.
\]

To compute the lipschitz  properties, we extract the bilinear part from the nonlinear part due to $\alpha_1.$ For $\overline{\beta} \in B_{\ell^2(\mathbb N_2)}(0,C\e^2)$ we split $\mathcal Q_k[\bar{\beta}]  = \mathcal  Q^{bl}_k[\bar{\beta},\bar{\beta}] + \mathcal R_k[\bar{\beta}]$ where:
\[
\mathcal R_k[\bar{\beta}] =  \f{1}{2\sqrt{2}} \alpha_1[\bar{\beta}]^2 \delta_{k=2} +  \dfrac{\sqrt{k!}}{2^k} \alpha_1[\bar{\beta}] \beta_{k-1} \delta_{k\geq 3} + \beta_k \alpha_1[\bar{\beta}] (H_1,m_{\varepsilon})_{G}
  - \alpha_1[\bar{\beta}](H_1,m_{\varepsilon} H_k)_G.
\]
Combining the lipschitz-properties of $\alpha_1$ and the size of $m_{\varepsilon}$ in $L^2(\mathbb R,G(x){\rm d}x)$ we
obtain:
\[
\|\overline{\mathcal R}[\bar{\beta}^{(1)}] - \overline{\mathcal R}[\bar{\beta}^{(2)}] \|_{\ell^2(\mathbb N_1)}
\leq  K_m \varepsilon \|\overline{\beta}^{(1)} -\overline{\beta}^{(2)}\|_{\ell^2(\mathbb N_2)}\,,
\quad \forall \, (\beta^{(1)},\beta^{(2)}) \in [B_{\ell^2(\mathbb N)}(0,C\e^2)]^2.
\]
While with similar arguments as previously, we have, for arbitrary
$\bar{\beta}^{(1)},\bar{\beta}^{(2)}$ in $\ell^{2}(\mathbb N_2):$
\[
\begin{aligned}
\|\overline{\mathcal Q}^{bl}[\bar{\beta}^{(1)},\bar{\beta}^{(2)}]\|_{\ell^2(\mathbb N_2)}^2
& \leq \sum_{k=2}^{\infty} \dfrac{k!}{4^k} \left| \sum_{l=2}^{k-2} \dfrac{\beta_l^{(1)} \beta_{k-l}^{(2)} }{\sqrt{l! (k-l)!}} \right|^2  + \sum_{k=2}^{\infty} |\sum_{l=2}^{\infty} \beta_k^{(1)} \beta_l^{(2)} (H_l,m_{\varepsilon}H_0)_G|^2\\
& \leq \sum_{k=2}^{\infty} \dfrac{1}{2^k} \sum_{l=2}^{k-2}|\beta^{(1)}_l|^2 |\beta^{(2)}_{k-l}|^2 + \sum_{k=2}^{\infty} |\beta^{(1)}_k|^2 \|\beta^{(2)}\|^2_{\ell^2(\mathbb N_2)} \|m_{\varepsilon}  H_0\|^2_{L^2(\mathbb R,G(x){\rm d}x)}\\
& \leq (1+ K_m \varepsilon^2)) \|\bar{\beta}^{(1)}\|_{\ell^2(\mathbb N_2)}^2 \|\bar{\beta}^{(2)}\|_{\ell^2(\mathbb N_2)}^2.
\end{aligned} 
\]

This concludes the proof.
\qed \end{proof}

\section{Proof of Theorem \ref{thm_longtime}}
We proceed with the anlysis of the time-dependent problem \eqref{eq_Nepsform}. Similarly to the stationary problem, {we use the fact that the problem reads more simply when formulated in terms of  the unknown} $h=N_{\e}/G.$ 
Hence, we rewrite the problem to be tackled:
\begin{equation} \label{eq_hevolution}
\left\{
\begin{aligned}
\partial_t h&  = (\tilde{T}[h] - h) -    \left( m_{\varepsilon}   -  \int_{\mathbb R} m_{\varepsilon}(x) h(x)G(x){\rm d}x \right) h, \\
h(0,\cdot) &= h_0,
\end{aligned}
\right.
\end{equation}
and we analyze the potential large-time properties of a solution to this problem. 
Our approach {to} Theorem \ref{thm_longtime} follows the classical scheme. Firstly, we analyze a Cauchy theory for \eqref{eq_hevolution} with a given initial data $h_0 \in L^2(\mathbb R,G(x) dx).$ We obtain existence and uniqueness of a solution for short times. In a second step, we show that for initial data close to the stationary state $\bar{h}_{\e} = N_{\e}/G,$ we have a uniform bound for the solution that entails this solution is global in time. We end-up the section by a proof of the stability {estimates.}

Before going into more precise statements and their associated proofs, we remind that we tackled a problem similar to \eqref{eq_Nepsform} in \cite{JG.MH.SM:23} up to normalizing the mass of the density. We proved that non-negative continuous initial data with enough bounded moments yield a unique solution that is continuous (in time and space) and has bounded moments (comparable to the initial data). Furthermore, this solution remains positive for all times. The main challenge in the first part of this section is to prove that we can construct a continuous solution that remains bounded in the weighted-$L^2$ space. 
The uniqueness result that we mention here entails that, if we assume further that the initial data is positive and continuous, the solution we construct is the unique one we constructed in our previous paper so that it remains positive as long as it exists.

\subsection{Cauchy theory for \eqref{eq_hevolution}}
In this part, we build a Cauchy theory on the basis of the following a priori estimate. Let $h$
be a solution on $(0,T)$. {We perform a dot-product between \eqref{eq_hevolution} and $h$} (according to the $L^2(\mathbb R,G(x) dx)$ scalar product). We obtain:
\begin{multline*}
\dfrac{1}{2} \dfrac{\textrm{d}}{\textrm{d}t} \|h(t,\cdot)\|^2_{L^2(\mathbb R,G(x) dx)} + 
\int_{\mathbb R} (1+m_{\e}(x)) |h(t,x)|^2 G(x) dx 
 \\
 = \left( \tilde{T}[h],h \right)_{G}  + \int_{\mathbb R} m_{\e}(x) h(t,x) G(x) dx \|h(t,\cdot)\|_{L^2(\mathbb R,G(x)dx)}^2.
\end{multline*}
We recall that we assume $1+m_{\e} \geq 0.$ Since $\tilde{T}$ is bilinear continuous ({see Lemma \ref{lem:T-cont}}), this entails via a standard Gronwall inequality that {the
$L^2(\mathbb R,G(x)dx)$-norm of the solution may only double on short times (see {\bf Section \ref{subsec_existence}} for related computations) so that:}
\begin{align}
& h \in L^{\infty}(0,T; L^2(\mathbb R,G(x) dx))  \label{eq_born1}\\
& \sqrt{1+m_{\e}}h \in L^2(0,T; L^2(\mathbb R,G(x)dx)) \label{eq_born2}. 
\end{align}
We point out that, with such a regularity, we have in particular
\[
 (\tilde{T}[h] - h) -    \left( m_{\varepsilon}   -  \int_{\mathbb R} m_{\varepsilon}(x) h(x)G(x){\rm d}x \right) h \in L^2((0,T) ; L^1(\mathbb R,G(x)dx)) \subset L^{1}_{loc}((0,T) \times \mathbb R).
\]
This entails that we may require 
\begin{equation} \label{eq_hevoldist}
\partial_t h  = (\tilde{T}[h] - h) -    \left( m_{\varepsilon}   -  \int_{\mathbb R} m_{\varepsilon}(x) h(x)G(x){\rm d}x \right) h \text{ in $\mathcal D'((0,T) \times \mathbb R)$},
\end{equation}
so that $\partial_t h \in L^2((0,T); L^1(\mathbb R,G(x)dx)).$ We have then {$h \in C([0,T] ; L^1(\mathbb R,G(x)dx))$ so  that initial conditions can be} matched in the sense:
\begin{equation} \label{eq_hevolinit}
h(0,\cdot) = h_0 \text{ in $L^1(\mathbb R,G(x)dx)$}
\end{equation}
Eventually, we give the following definition for our solutions:
\begin{definition} \label{def_cauchy}
Given $h_0 \in L^2(\mathbb R, G(x){\rm d}x)$ and $T >0,$ we call solution to \eqref{eq_hevolution} on $(0,T)$ any $h : (0,T) \times \mathbb R \to \mathbb R$ satisfying regularity requirements \eqref{eq_born1}-\eqref{eq_born2} and equation \eqref{eq_hevolution} in the sense of \eqref{eq_hevoldist}-\eqref{eq_hevolinit}.
\end{definition}

Our main result concerning the Cauchy theory reads:
\begin{theorem} \label{thm_cauchy}
Let $h_0\in L^2(\mathbb R, G(x){\rm d}x) {\cap C_b(\mathbb R)}$ be non-negative and such that 
\[
\|h_0\|_{L^2(\mathbb R,G(x)dx)} \leq M.
\] 
Then there exists $T_{M} >0$ depending only on $M$ such that, for arbitrary $T < T_M,$ there is a unique solution to \eqref{eq_hevolution} on $(0,T).$ {Moreover, this solution is non-negative.}
\end{theorem}

What remains of this subsection is devoted to the proof of this theorem. We split this proof into two parts. Firstly, we obtain uniqueness relying on our previous contribution \cite{JG.MH.SM:23}, we postpone the existence proof to the end of this subsection. 

\subsubsection{Proof of Theorem \ref{thm_cauchy}. Uniqueness}
Let $h_0 \in L^2(\mathbb R,G(x)dx) \cap C_b(\mathbb R)$ be non-negative and $h$ any solution that we may construct on $(0,T)$ for some $T>0$. In particular, for all $k \geq 1,$ we have $n_0 := h_0G \in X_k$ where:
\[
X_k := \left\{ n \in L^1(\mathbb R) , \quad  \int_{\mathbb R} (1+ |x|)^k n(x) dx < \infty \right\}.
\]
and $n = hG \in L^{\infty}((0,T); X_k).$ Seeing \eqref{eq_Nepsform} as a differential system in $N_{\e}$ with source term $T_{1}[N_{\e}]$
 we infer that:
\[
\begin{aligned}
n(t,x)  =&  n_0(x)\exp\left[-t (1+m_{\e}(x)) + \int_0^{t}\int_{\mathbb R} m_{\e}(y)n(s,y) {\rm d}y {\rm d}s \right] \\
& + \int_0^{t} \exp\left[ - (t-s) (1+m_{\e}(x)) + \int_{s}^{t} \int_{\mathbb R} m_{\e}(y)n(u,y) {\rm d}y {\rm d}u\right] T_1[n(s,\cdot)] {\rm d}s. 
\end{aligned}
\]
At this point, we remark that $m_{\e}$ is continuous with $(1+m_{\e}) \geq 0$. Furthermore, by standard convolution arguments, we have that $T_1[n] \in L^{\infty}(0,T ; C_b(\mathbb R)).$ This entails that ${n} \in C_b([0,T] \times \mathbb R)$ solves \eqref{eq_Nepsform} in  $\mathcal D'((0,T))$ for all $x \in \mathbb R.$  We can then argue that \cite[Proposition 1]{JG.MH.SM:23} entails there is at most one such solution {that is positive}.

\subsubsection{Proof of Theorem \ref{thm_cauchy}. Existence}
\label{subsec_existence}
To obtain existence of a solution in the sense of {Definition \ref{def_cauchy}}, we proceed with a Galerkin method. For this, we plug that $h(t,x) = \sum_{k} \alpha_k(t) H_k(x)$ in \eqref{eq_hevolution}
and transform this system into the infinite differential system:
\begin{equation} \label{eq_alpha_evolution}
\left\{
\begin{aligned}
 \dot{\alpha}_k &=  \dfrac{\sqrt{k!}}{2^{k}} \sum_{l=0}^{k} \dfrac{\alpha_l \alpha_{k-l}}{\sqrt{l!(k-l)!}}- \alpha_k - \left(\sum_{l=0}^{\infty} \alpha_l H_l  \; ,  \; m_{\varepsilon} (H_k - \alpha_k H_0)  \right)_{G}    \qquad \forall \, k \geq 1,\\
 {\alpha}_k(0) &= \alpha_k^0
\end{aligned}
\right.
\end{equation}
where $\alpha_0 = 1$ is constant with time. 

\medskip

Fix $K \geq 1$ arbitrary large. By a standard Cauchy-Lipschitz argument, there exists  $T_K >0$ and a unique $(\alpha^{(K)}_k)_{k=1,\ldots,K} \in C^1([0,T_K])$ satisfying:
\[
\left\{
\begin{aligned}
 \dot{\alpha}^{(K)}_k &=  \dfrac{\sqrt{k!}}{2^{k}} \sum_{l=0}^{k} \dfrac{\alpha_l^{(K)} \alpha_{k-l}^{(K)}}{\sqrt{l!(k-l)!}}- \alpha_k^{(K)} - \left(\sum_{l=0}^{K} \alpha_l^{(K)} H_l  \; ,  \; m_{\varepsilon} (H_k - \alpha_k^{(K)} H_0)  \right)_{G}    \\
 {\alpha}^{(K)}_k(0) &= \alpha_k^0
\end{aligned}
\right.
\qquad \forall \, k =1\ldots,K,
\] 
with $\alpha_0^{(K)} =1.$ Let {us} denote $h^{(K)} = \sum_{k=0}^{K} \alpha^{(K)}_k H_k.$ Dot-multiplying this latter differential system with $(\alpha_k^{(K)})_{k=1,\ldots,K},$ we infer that:
\[
\begin{aligned}
\dfrac{1}{2}\dfrac{\textrm{d}}{\textrm{d}t} \left[\|h^{(K)}\|_{L^2(\mathbb R,G(x)dx)}^2 \right]
+ \|h^{(K)}\|_{L^2(\mathbb R,G(x)dx)}^2 = 
\sum_{k=0}^{K} \sum_{l=0}^{k}\dfrac{\sqrt{k!}}{2^{k}}  \dfrac{\alpha_l^{(K)} \alpha_{k-l}^{(K)} \alpha_k^{(K)}}{\sqrt{l!(k-l)!}}\\
 - ({\mathbb P}^{(K)} m_{\e}h^{(K)},h^{(K)})_{G} + \int_{\mathbb R} m_{\e}(y) h^{(K)}(t,y) G(y) {\rm d}y \|h^{(K)}\|_{L^2(\mathbb R,G(x)dx)}^2
\end{aligned}
\]
where ${\mathbb P}^{(K)}$ is the projection on the $K$-first modes in the Hermite expansion. Since this is a bounded symmetric mapping, we have: 
\[
({\mathbb P}^{(K)} m_{\e}h^{(K)},h^{(K)})_{G} = \int_{\mathbb R} m_{\e}(y) |h^{(K)}(t,y)|^2 G(y){\rm d}y.
\]
Eventually, we infer that:
\begin{multline*}
 \dfrac{1}{2}\dfrac{\textrm{d}}{\textrm{d}t} \left[\|h^{(K)}\|_{L^2(\mathbb R,G(x)dx)}^2 \right]
+ \int_{\mathbb R} (1+m_{\e}(y))|h^{(K)}(t,y)|^2 G(y){\rm d}y
\\
= 
\sum_{k=0}^{K} \sum_{l=0}^{k}\dfrac{\sqrt{k!}}{2^{k}}  \dfrac{\alpha_l^{(K)} \alpha_{k-l}^{(K)} \alpha_k^{(K)}}{\sqrt{l!(k-l)!}} +  \int_{\mathbb R} m_{\e}(y) h^{(K)}(t,y) G(y) {\rm d}y \|h^{(K)}\|_{L^2(\mathbb R,G(x)dx)}^2.
\end{multline*}
{To control} the right-hand side, we recall that $\tilde{T}$ defines a bounded bilinear mapping on $L^2(\mathbb R,G(x) dx)$ ({see Lemma \ref{lem:T-cont}}) so that there is a constant $C_{bl}$ independent of $K$ for which:
\[
\sum_{k=0}^{K} \sum_{l=0}^{k}\dfrac{\sqrt{k!}}{2^{k}}  \dfrac{\alpha_l^{(K)} \alpha_{k-l}^{(K)} \alpha_k^{(K)}}{\sqrt{l!(k-l)!}} \leq C_{bl} \|h^{(K)}\|_{L^2(\mathbb R,G(x) dx)}^3. 
\]
Similarly, by a standard Cauchy-Schwarz inequality:
\[
\int_{\mathbb R} m_{\e}(y) h^{(K)}(t,y) G(y) {\rm d}y \|h^{(K)}\|_{L^2(\mathbb R,G(x)dx)}^2 \leq \|m_{\e}\|_{L^2(\mathbb R,G(x)dx)} \|h^{(K)}\|^3_{L^2(\mathbb R,G(x)dx)}.
\]
Eventually, we obtain that:
\begin{multline} \label{eq_energy}
 \dfrac{1}{2}\dfrac{\textrm{d}}{\textrm{d}t} \left[\|h^{(K)}\|_{L^2(\mathbb R,G(x)dx)}^2 \right]
+ \int_{\mathbb R} (1+m_{\e}(y))|h^{(K)}(t,y)|^2 G(y){\rm d}y \\
\leq  \left( C_{bl} + \|m_{\e}\|_{L^2(\mathbb R,G(x)dx)} \right) \|h^{(K)}\|_{L^2(\mathbb R,G(x) dx)}^3.
\end{multline}
Consequently, with a classical blow-up alternative argument, we obtain that, under the further assumption that $\|h_0\|_{L^2(\mathbb R,G(x)dx)} \leq M,$ we can build a time $T_{M}$ depending only on 
$M$ such that $T_K = T_M$ and $\|h^{(K)}(t,\cdot)\|_{L^2(\mathbb R,G(x)dx)} \leq 2M$ for all $t \leq T_M.$ Integrating a last time \eqref{eq_energy} and using this control on the norm of $h^{(K)}$ we obtain also that:
\[
\int_0^{T_M} \int_{\mathbb R} (1+m_{\e}(y))|h^{(K)}(t,y)|^2 G(y){\rm d}y \leq M +  \left( C_{bl} + \|m_{\e}\|_{L^2(\mathbb R,G(x)dx)} \right) 8{T_M M^3}.
\]

\medskip

We can let now $K$ tend to infinity in this family of approximate solutions. For legibility we drop the index $M$ and assume these solutions are defined on $(0,T).$ With the computation above, we have that $h^{(K)}$ is bounded in $L^{\infty}(0,T;L^2(\mathbb R,G(x)dx))$ with $\sqrt{1+m_{\e}} h^{(K)}$ bounded in $L^2((0,T); L^2(\mathbb R,G(x)dx)).$ We can thus extract a subsequence, that we do not relabel, for which 
\[
\begin{aligned}
& h^{(K)} \rightharpoonup h && \text{in $L^{\infty}(0,T;L^2(\mathbb R,G(x)dx)) - w*$} \\
& \sqrt{1+m_{\e}}h^{(K)} \rightharpoonup \mathcal D && \text{ in    $L^2((0,T); L^2(\mathbb R,G(x)dx))-w$}.
\end{aligned}
\]
We note here that the mapping $h \mapsto \sqrt{1+m_{\e}}h$ is linear continuous $L^2(\mathbb R,G(x)dx) \to L^1(\mathbb R,G(x)dx).$ Because of the weak convergence of $h^{(K)},$ we conclude that $\mathcal D = \sqrt{1+m_{\e}}h.$
For fixed $k \in \mathbb N,$ we can then pass to the limit in the $k$ equation satisfied by $\alpha^{(K)}_k$ for $K \geq k.$ We infer that $h = \sum_{k} \alpha_k H_k$ with 
\[
\left\{
\begin{aligned}
 \dot{\alpha}_k &=  \dfrac{\sqrt{k!}}{2^{k}} \sum_{l=0}^{k} \dfrac{\alpha_l \alpha_{k-l}}{\sqrt{l!(k-l)!}}- \alpha_k - \left(\sum_{l=0}^{\infty} \alpha_l H_l  \; ,  \; m_{\varepsilon} (H_k - \alpha_k H_0)  \right)_{G}    \\
 {\alpha}_k(0) &= \alpha_k^0
\end{aligned}
\right.
\qquad \forall \, k \in \mathbb N.
\] 
In other words, we obtain that:
\[
\left\{
\begin{aligned}
\partial_t (h,H_k)_G &= (\tilde{T}[h],H_K) - (\sqrt{(1+ m_{\e})}h,\sqrt{1+m_{\e}}H_k) + \int_{\mathbb R}{m_\e(y)} h(\cdot,y) G(y){\rm dy} (h,H_k) \\
(h(0,\cdot),H_k) &= (h_0,H_k) 
\end{aligned}
\right.
\qquad \forall \, k \in \mathbb N.
\]
 {Notice that}
\[
C^{\infty}_c (\mathbb R) \subset  
\{\varphi \in L^2(\mathbb R,G(x)dx) \text{ s.t. } \sqrt{1+m_{\e}}\varphi \in L^2(\mathbb R,G(x)dx)\}
\]
in which the $(H_k)_{k\geq 1}$ are dense (since $m$ grows at most polynomially).  {Moreover, $\varphi\in C^{\infty}_c(\mathbb R)$ if and only if $\varphi(\cdot)/G(\cdot)\in C^{\infty}_c(\mathbb R)$.} We infer that this latter equation extends to all test functions in $C^{\infty}_c(\mathbb R).$ That means \eqref{eq_hevolution} holds in $\mathcal D'((0,T) \times \mathbb R)$
with an initial condition in $L^1(\mathbb R,G(x)dx).$ This ends the existence proof.

\subsection{Orbital stability of stationary states}
\label{sec:orbital-stabilty}

In the previous proof, we obtained local-in-time existence of a solution for arbitrary initial data. We obtain now that, for data sufficiently close to the Gaussian, the solution is global and remains close to the Gaussian. 
We start with a statement in the stable case:

\begin{proposition} \label{prop_orbstability_stable}
Let us assume that $m^{''}(0) >0.$ Given $\overline{C}_0 >0$ there exists ${\overline{C}_1}$ and ${\overline{C}_2}$ 
depending on $\overline{C}_0$ and $m$ such that the following property holds true.

Given $\varepsilon$ sufficiently small and a non-negative initial data $h_0 \in L^2(\mathbb R,G(x){\rm d}x)) \cap C_b(\mathbb R)$ such that 
\[
|\alpha_1^0| \leq \overline{C}_0 \varepsilon \qquad \|\overline{\alpha}_2^{0}\| \leq \overline{C}_0 \varepsilon^2 ,
\] 
there is a unique global solution $h$ to \eqref{eq_hevolution} that satisfies moreover:
\begin{equation} \label{eq_boundglobal1}
|\alpha_1(t)| \leq {\overline{C}_1} \varepsilon \qquad \|\overline{\alpha}_2(t)\| \leq {\overline{C}_2} \varepsilon^2  \qquad \forall \, t \geq 0.
\end{equation}
\end{proposition}

{In this statement and from now on, we drop the index $\ell^2(\mathbb N_2)$ in the norm of sequences with index $2$ for legibility.} In the unstable case, we have the following:
\begin{proposition} \label{prop_orbstability_unstable}
Let us assume that $m^{''}(0) <0$   {and that $m$ is even}. Given $\overline{C}_0 >0$ there exists ${\overline{C}_2}$ 
depending on $\overline{C}_0$ and $m$ such that the following property holds true.

Given $\varepsilon$ sufficiently small and a non-negative even initial data $h_0 \in L^2(\mathbb R,G(x){\rm d}x)) \cap C_b(\mathbb R)$ such that 
\[
{\alpha_1^0=0,}  \qquad \|\overline{\alpha}_2^{0}\| \leq \overline{C}_0 \varepsilon^2 ,
\]
there is a unique global solution $h$ to \eqref{eq_hevolution} that satisfies moreover:
\begin{equation} \label{eq_boundglobal2}
{\alpha_1(t)=0,}   \qquad \|\overline{\alpha}_1(t)\| \leq {\overline{C}_2} \varepsilon^2  \qquad \forall \, t \geq 0.
\end{equation}
\end{proposition}

What remains of this section is devoted to a unified proof of both statements. Since we constructed a local-in-time Cauchy theory for \eqref{eq_hevolution} in the space $L^2(\mathbb R,G(x) dx) \cap C(\mathbb R)$ that propagates that $h(t,\cdot) \geq 0,$ we focus herein only on a proof of the bounds \eqref{eq_boundglobal1} (resp. \eqref{eq_boundglobal2}). For this, we fix $C_0 >0$ and we consider either that $m$ satisfies $m^{''}(0) >0$ (stable case) or that $m^{''}(0) < 0$ and is even (unstable even case). We fix a non-negative initial data $h^0$ enjoying the bound:
\[
|\alpha_1^0| \leq \overline{C}_0 \varepsilon \qquad \|\overline{\alpha}_2^{0}\| \leq \overline{C}_0 \varepsilon^2 ,
\]
and that is moreover even in the unstable case. We assume a priori that, for some $T>0$, {which may depend on $\e$,} we have a unique solution on  $h$  on $(0,T)$ in the sense of {\bf Definition \ref{def_cauchy}} satisfying moreover \eqref{eq_boundglobal1} (resp. \eqref{eq_boundglobal2}) up to time $T$ for some ${\overline{C}_1}$ and ${\overline{C}_2}$ to be fixed. We prove in the following arguments that, under some restrictions on ${\overline{C}_2}$ and $\varepsilon$ (independent of the solution and $T$) we have that 
\begin{equation} \label{eq_continuation}
|\alpha_1(t)| \leq {\overline{C}_1}\varepsilon/2 \qquad  \|\overline{\alpha}_2(t)\| \leq {\overline{C}_2} \varepsilon^2/2 \qquad \forall \, t \in [0,T].
\end{equation}
Our result then follows via a standard continuation argument relying on our Cauchy theory.
We split now the proof in two parts corresponding to the two bounds to be obtained.

\medskip

From now on, we use similar notations as in the stationary case. Namely, we rewrite 
\eqref{eq_alpha_evolution} as a differential system on $\ell^2(\mathbb N):$ 
\begin{equation} \label{eq_alpha_evolution3}
\dfrac{\textrm{d}\overline{\alpha}_1}{{\textrm d}t}
= \overline{m}^{(\varepsilon)}_1 + \mathcal L^{(\varepsilon)} [ \overline{\alpha}_1] + 
\mathcal Q^{(\varepsilon)} [\overline{\alpha}_1]. 
\end{equation}
where $\mathcal L^{(\varepsilon)}$ and $\mathcal Q^{(\varepsilon)}$ stand for the respective extensions to the exponent $k=1$
of the $\overline{\mathcal L}$ and $\overline{\mathcal Q}$ in the previous section:
\beq
\label{def:L}
\mathcal L^{(\varepsilon)} [\overline{\alpha}]=
\left[ 
\left( \dfrac{1}{2^{k-1}}  - 1  +(H_0,  m_{\varepsilon} H_0)_G   \right) \alpha_k 
 - \left( \sum_{l=1}^{\infty} \alpha_l H_l ,  m_{\varepsilon} H_k\right)_G \\
\right]_{k\in \mathbb N_1}
\eeq
and
\[
\mathcal Q^{(\varepsilon)} [\overline{\alpha}]=
 \left[ \dfrac{\sqrt{k!}}{2^k} \sum_{l=1}^{k-1} \dfrac{\alpha_l \alpha_{k-l}}{\sqrt{l! (k-l)!}}  + \sum_{l=1}^{\infty} \alpha_k  \alpha_l  m_l^{(\varepsilon)}\right]_{k\in\mathbb N_1}.
\]
Without restriction, we enforce the first condition that ${\overline{C}_2} \geq 1$
(so that powers of ${\overline{C}_2}$ grow with the exponent) and $\e \leq 1$ ({so that   the converse holds}). {We also enforce the condition $\overline C_2>\overline C_1$.  }
 
\medskip

\noindent{\em Part 1. Computing a bound on $\alpha_1(t).$}
We start with the unstable even case. Thanks to the symmetries of the operator $T$ (and thus $\tilde{T}$) and the symmetries of $m$ we note that $h_{s}(t,x) = h(t,-x)$ is also
a solution to \eqref{eq_hevolution} in the sense of {\bf Definition \ref{def_cauchy}} with the same initial data $h^0.$ By uniqueness, we obtain that $h(t,\cdot)$ is even for all $t \in [0,T].$ This entails that:
\[
\alpha_1(t) = \int_{\mathbb R} h(t,x) H_1(x) G(x) dx = \int_{\mathbb R} h(t,x) x G(x) dx = 0\,, \qquad \forall \, t \in [0,T]. 
\] 
We have our expected {property}.

\medskip

The stable case requires more details. We multiply the first equation of \eqref{eq_alpha_evolution3} with $\alpha_1$ 
\[
\dfrac{1}{2} \dfrac{\textrm{d}}{\textrm{d}t} |{\alpha}_1|^2  = {m}_1^{(\varepsilon)} \alpha_1
+ \mathcal L^{(\varepsilon)}_1 [ \overline{\alpha}_1] \alpha_1  + 
\mathcal Q^{(\varepsilon)}_1 [\overline{\alpha}_1]\alpha_1 .
\]
For the first term, we argue again that $|m_{1}^{(\varepsilon)}| \leq  K_m \varepsilon^3$ (see Lemma \ref{lem_m}) to yield that
\begin{equation}  \label{eq_RHS11}
|m^{(\varepsilon)}_1 \alpha_1| \leq \dfrac{\eta}{2}\varepsilon^2 |\alpha_1|^2  + \dfrac{1}{2\eta} K_m \varepsilon^4 ,
\end{equation}
for arbitrary $\eta >0.$ As for the second term,  we note the following explicit formula
\[
\mathcal L_1^{(\varepsilon)}[\overline{\alpha}_1]\alpha_1 = \left( (H_0,m_{\varepsilon}H_0)_G  - (H_1,m_{\varepsilon} H_1) \right)|\alpha_1|^2 - \left( \sum_{l=2}^{\infty}\alpha_l H_l,m_{\varepsilon}\alpha_1 H_1 \right).
\]
in which we already computed (see \eqref{eq_Deps}) and since {$m^{''}(0)=1$,} that
\[
\left| (H_0,m_{\varepsilon} H_0)_G - (H_1,m_{\varepsilon} H_1)_G  {-} \varepsilon^2   D_0 \right|
\leq K_m \varepsilon^3 
\]
where 
\[
D_0 = \int_{\mathbb R} (1-x^2)x^2 G(x){\rm d}x {<} 0.
\]  
On the other hand, {thanks to}  Lemma \ref{lem_normmHk}, we have a constant $K_m$ for which: 
\beq
\label{inequality-amH}
\begin{array}{rl}
\left| \left(\sum_{l=2}^{\infty}\alpha_l H_l,  \alpha_1 m_{\varepsilon}H_1 \right)_{G} \right|&
 {\leq |\alpha_1 | \|\sum_{l=2}^{\infty}\alpha_l H_l\|_{L^2(\mathbb R,G(x){\rm dx})}\|m_{\varepsilon}H_1\|_{L^2(\mathbb R,G(x){\rm dx})}}\\
&\leq K_m \varepsilon^2  \|\overline{\alpha}_2\| \,  |\alpha_1| \leq K_m \left( \varepsilon^3 |\alpha_1|^2 + {\overline{C}_2}^2 \varepsilon^5\right). 
\end{array}
\eeq
where we applied that $\|\overline{\alpha}_2\| \leq {\overline{C}_2} \varepsilon^2.$
Consequently, for small $\varepsilon, $ we have
\begin{equation} \label{eq_RHS12}
\mathcal L_1^{(\varepsilon)}[\overline{\alpha}_1]\alpha_1 \leq - \varepsilon^2 \dfrac{3{|D_0|}}{{4}}  |\alpha_1|^2  +  K_m {\overline{C}_2}^2  \varepsilon^5. 
\end{equation}
Finally, we observe that:
\[
\mathcal Q^{(\varepsilon)}_1 [\overline{\alpha}_1]\alpha_1  = \sum_{l=1}^{\infty} \alpha_1^2 \alpha_l m_{l}^{(\varepsilon)}.
\]
Applying a Cauchy-Schwarz inequality and Lemma \ref{lem_m}, 
we infer again: 
\[
\left| \sum_{l=1}^{\infty} \alpha_l m_{l}^{(\varepsilon)} \right| \leq \|\overline{\alpha}_1\|_{\ell^2(\mathbb N_1)} K_m \varepsilon^2 \leq K_m  {\overline{C}_2} \varepsilon^3,
\]
and
\begin{equation} \label{eq_RHS13}
\mathcal Q^{(\varepsilon)}_1 [\overline{\alpha}_1]\alpha_1   \leq K_m  {\overline{C}_2} \varepsilon^3 |\alpha_1|^2  .
\end{equation}
Choosing $\eta$ sufficiently small in \eqref{eq_RHS11} (depending only on $m$) and combining with \eqref{eq_RHS12}-\eqref{eq_RHS13} we infer that, for $\varepsilon$ sufficiently small (depending on $m$ and $\overline{C}_2$) there holds:
\begin{equation} \label{eq_alpha1l2}
\dfrac{\textrm{d}}{\textrm{d}t} |\alpha_1|^2  + {\dfrac{ {|D_0|}}{2}} \varepsilon^2 |\alpha_1|^2  \leq  K_m \varepsilon^4.
\end{equation}
After time integration, this entails that, as long as $|\alpha_1| \leq  {\overline{C}_1} \varepsilon$ and $\|\overline{\alpha}_2\| \leq {\overline{C}_2} \varepsilon^2$ we have:
\begin{equation} \label{eq_boundalpha1-final}
|{\alpha}_1(t)|^2 \leq |\alpha_1(0)|^2+ {\dfrac{4 K_m\varepsilon^2}{{|D_0|}}}.  
\end{equation}
We have the expected bound of \eqref{eq_continuation} {by choosing}
\begin{equation} \label{eq_condC1_1}
{\overline{C}_1}= 2 \left( \overline{C}_0^2 + {\dfrac{4K_m}{{|D_0|}} }\right)^{\frac 12}.
\end{equation}

\noindent{\em Part 2. Computing bounds on $\overline{\alpha}_2$.} In this second part, we do not distinguish the stable case and the unstable even case. We write {\em a priori} estimates directly for the full unknown $\overline{\alpha}_2.$ The reader should note that these estimates should be written at first for $(\alpha_2,\ldots,\alpha_K)$ with arbitrary large $K$ to recover our estimates letting $K \to \infty$.

\medskip

We multiply now equations for $\overline{\alpha}_2$ with $\overline{\alpha}_2.$  This yields:
\[
\dfrac{1}{2} \dfrac{\textrm{d}}{\textrm{d}t} \|\overline{\alpha}_2\|^2  = (\overline{m}^{(\varepsilon)}_2, \overline{\alpha}_2)
+ (\overline{\mathcal L}_2^{(\varepsilon)} [ \overline{\alpha}_1] ,\overline{\alpha}_2 ) + 
( \overline{\mathcal Q}_2^{(\varepsilon)} [\overline{\alpha}_1], \overline{\alpha}_2)
\]
with ever the same convention that $\overline{\mathcal L}_2^{(\varepsilon)}$ (resp. $ \overline{\mathcal Q}_2^{(\varepsilon)}$) regroups the components of $\mathcal L^{(\varepsilon)}$ (resp. $\mathcal Q^{(\varepsilon)}$) with index $k \geq 2$.
We compute again a bound for  the right-hand side under the assumption that $|\alpha_1| \leq  {\overline{C}_1} \varepsilon$ and $\|\overline{\alpha}_2\|  \leq  {\overline{C}_2} \varepsilon^2$  and $\varepsilon$ sufficiently small. 

\medskip

For the first term, we have again by a Cauchy-Schwarz inequality:
\begin{equation}  \label{eq_bRHS11}
| \sum_{k = 2}^{\infty} m^{(\varepsilon)}_k \alpha_k|  \leq \dfrac{\eta}{2} \|\overline{\alpha}_2\|^2 +  \dfrac{K_m}{2\eta} \varepsilon^4,
\end{equation}
for arbitrary $\eta \in (0,1).$ Concerning the second term, we note that:
\[
  (\overline{\mathcal L}_2^{(\varepsilon)} [ \overline{\alpha}_1] ,\overline{\alpha}_2 )  = \sum_{k=2}^{\infty} \left( \dfrac{1}{2^{k-1}}  - 1  +(H_0,  m_{\varepsilon} H_0)_G   \right) |\alpha_k|^2 
- \left(\sum_{l=1}^{\infty} \alpha_l H_l,  m_{\varepsilon} \sum_{k=2}^{\infty} \alpha_k H_k \right)_G .
\]
At this point, we note that, thanks to \eqref{as:m-extremum}, we have a $k_0 \geq 2$ such that:
\[
m_{\e} +1 \geq  \dfrac{1}{2^{k_0}}.
\] 
We split then correspondingly the sum with respect to the chosen exponent $k_0.$ We obtain:
\[
\begin{aligned}
 {(\overline{\mathcal L}_2^{(\varepsilon)} [ \overline{\alpha}_1] ,\overline{\alpha}_2 )}= & \sum_{k=2}^{{k_0+1}}  \left( \dfrac{1}{2^{k-1}}  - 1  +(H_0,  m_{\varepsilon} H_0)_G   \right) |\alpha_k|^2  - \left(\sum_{l=2}^{{k_0+1}} \alpha_l H_l,  m_{\varepsilon} \sum_{k=2}^{{k_0+1}} \alpha_k H_k \right)_G  \\
&  - 2  \left(\sum_{l=2}^{{k_0+1}} \alpha_l H_l,  m_{\varepsilon} \sum_{{k=k_0+2}}^{\infty} \alpha_k H_k \right)_G   -  \left( \alpha_1 m_{\varepsilon} H_1,   \sum_{k=2}^{\infty} \alpha_k H_k \right)_G  \\
& + 
 \sum_{{l=k_0+2}}^{\infty}  \left( \dfrac{1}{2^{k-1}}  - 1  +(H_0,  m_{\varepsilon} H_0)_G   \right) |\alpha_k|^2 -    \left(\sum_{{l=k_0+2}}^{\infty} \alpha_l H_l,  m_{\varepsilon} \sum_{{l=k_0+2}}^{\infty} \alpha_k H_k \right)_G.
\end{aligned}
\]
We denote with $T_1,T_2,T_3$ the three lines on the right-hand side of this latter identity. For the last term, we remark that, by introducing 
${h}_{k_0} =  \sum_{{l=k_0+2}}^{\infty} \alpha_l H_l,$ there holds:
\[
\begin{aligned}
T_3  = & \left( \dfrac{1}{2^{k_0+1}}   + (H_0,m_{\varepsilon} H_0)_G \right) \int_{\mathbb R} |h_{k_0}|^2 G(x){\rm d}x - \int_{\mathbb R}(1+ m_{\varepsilon} )|h_{k_0}|^2 G(x){\rm d}x\\
 \leq & \left(  - \dfrac{1}{2^{k_0+1}} + (H_0,m_{\varepsilon} H_0)_G \right) \int_{\mathbb R} |h_{k_0}|^2 G(x){\rm d}x ,
\end{aligned}
\]
where we used the bound from below on $(1+m_{\e})$ to pass from the first to the second line. When $\varepsilon$ is chosen sufficiently small (depending on $m$), the above computations entail that:
\begin{equation} \label{eq_T3}
T_3 \leq - \dfrac{1}{2^{k_0+2}} \sum_{{k={k_0+2}}}^{\infty} |\alpha_k|^2.
\end{equation}
Concerning $T_1,$ we apply again Lemma \ref{lem_normmHk} to yield that:
\[
|(H_0,m_{\varepsilon} H_0)_G| + \sum_{l=2}^{{k_0+1}} \sum_{k=2}^{{k_0+1}} |(H_l,m_{\varepsilon} H_k)_G| \leq K_m \varepsilon^2.
\] 
Consequently, we obtain the bound:
\[
T_1 \leq - \dfrac{1}{2}\sum_{k=2}^{{k_0+1}}   |\alpha_k|^2  + K_m \varepsilon^2 \sum_{k=2}^{{k_0+1}}   |\alpha_k|^2,
\]
and, choosing $\varepsilon$ sufficiently small (with a threshold depending on $m$),  
\begin{equation} \label{eq_T1}
T_1 \leq - \dfrac{1}{4}\sum_{k=2}^{{k_0+1}}   |\alpha_k|^2 .
\end{equation}
Finally,  we have: 
\[
|T_2|  \leq  2 \left| ( \sum_{l21}^{{k_0+1}} \alpha_l  m_{\varepsilon} H_l, \sum_{k=2}^{\infty} \alpha_k H_k )_G\right| + \left| \left( \alpha_1 m_{\varepsilon} H_1,   \sum_{k=2}^{\infty} \alpha_k H_k \right)_G\right| .
\]
Applying again Cauchy-Schwarz inequalities and Lemma \ref{lem_normmHk} for $k\leq {k_0+1}$ we infer:
\[
\begin{aligned}
 \left| ( \sum_{ l=2}^{{k_0+1}} \alpha_l  m_{\varepsilon} H_l, \sum_{k=2}^{\infty} \alpha_k H_k )_G\right|  & \leq 
K_m \varepsilon^2 \left( \sum_{ l=2}^{{k_0+1}} |\alpha_l|^2 \right)^{\frac 12} \left( \sum_{l={k_0+2}}^{\infty} |\alpha_l|^2\right)^{\frac 12} \\
& \leq {K_m \varepsilon^2 \|\overline{\alpha}_2\|^2},\\ 
 \left| \left( \alpha_1 m_{\varepsilon} H_1,   \sum_{k=2}^{\infty} \alpha_k H_k \right)_G\right|  & \leq  |\alpha_1| \|\overline{\alpha}_2\| \|m_{\varepsilon} H_1\|_{L^2(\mathbb R,G(x){\rm d}x)}  \leq K_m \varepsilon^2 |\alpha_1| \|\overline{\alpha}_2\| .
\end{aligned}
\] 
Combining the two latter estimates and introducing the {\em a priori} bounds 
$|\alpha_1| \leq {\overline{C}_1} \varepsilon,$   $\|\overline{\alpha}_2\| \leq {\overline{C}_2} \varepsilon^2$  {and $\overline C_1\leq \overline C_2$} we conclude  that: 
\begin{equation} \label{eq_T2}
|T_2|   \leq K_m  {\overline{C}_2}^2 \varepsilon^5 .
\end{equation}
Eventually, we combine \eqref{eq_T3}-\eqref{eq_T1}-\eqref{eq_T2} to yield:
\begin{equation} \label{eq_bRHS12}
 (\overline{\mathcal L}_2^{(\varepsilon)} [ \overline{\alpha}_1] ,\overline{\alpha}_2 ) \leq  -\dfrac{1}{2^{k_0+2}}  \|\overline{\alpha}_2\|^2 + K_m {\overline{C}_2}^2 \varepsilon^5. 
\end{equation}

As for the nonlinear term, we have:
\[
( \overline{\mathcal Q}_2^{(\varepsilon)} [\overline{\alpha}_1], \overline{\alpha}_2) = \sum_{k=2}^{\infty}  \alpha_k \left[ \dfrac{\sqrt{k!}}{2^k} \sum_{l=1}^{k-1} \dfrac{\alpha_l \alpha_{k-l}}{\sqrt{l! (k-l)!}}  + \sum_{l=1}^{\infty} \alpha_k  \alpha_l  m_l^{(\varepsilon)}\right] ,
\]
that we split into $S_1+S_2$ where:
\[
S_1 =  \sum_{k=2}^{\infty} \sum_{l=1}^{k-1} \dfrac{1}{2^k} \sqrt{\begin{pmatrix} k \\ l \end{pmatrix}} \alpha_l \alpha_{k-l} \alpha_k  \,, 
\quad 
S_2 =  \sum_{l=1}^{\infty} \alpha_lm_{l}^{(\varepsilon)} \sum_{k=1}^{\infty} \alpha_k^2 .
\]
For $S_2,$ we use again Lemma \ref{lem_m} and the a priori bound on $|\alpha_1|$  {and $\|\overline \alpha_2\|$} to yield:
\[
|S_2| \leq\left(  \sum_{l=1}^{\infty} |\alpha_l|^2  \right)^{\frac 32} \left( \sum_{l=1}^{\infty} |m_l^{(\varepsilon)}|^2 \right)^{\frac 12} \leq 
\|m_{\e}\|_{L^2(\mathbb R,G(x)dx)} \left( |\alpha_1|^2+  {\|\overline \alpha_2\|^2} \right)^{\frac 32}
\]
and
\begin{equation} \label{eq_S2}
|S_2| \leq K_m {\overline{C}_2}^3\varepsilon^5 .
\end{equation}
Finally, for $S_1$ we use properties of binomials as in Appendix \ref{sec:appendix}, {Lemma \ref{lem:T-cont}:
\[
\begin{aligned}
| S_1 | & \leq \|\overline{\alpha}_2\| \left(  \sum_{k=2}^{\infty} \left[ \sum_{l=1}^{k-1} \dfrac{1}{2^k} \sqrt{\begin{pmatrix} k \\ l \end{pmatrix}} \alpha_l \alpha_{k-l} \right]^2  \right)^{\frac 12}  \leq  \|\overline{\alpha}_2\|  \left[  \sum_{k=2}^{\infty} \dfrac{1}{2^k} \sum_{l=1}^{k-1} \alpha_l^2 \alpha_{k-l}^2 \right]^{\frac 12}  \\
& \leq  \|\overline{\alpha}_1\|_{\ell^2(\mathbb N_1)}^2 \|\overline{\alpha}_2\|  . 
\end{aligned}
\]
In the regime prescribed by $\|\overline{\alpha}_2\| \leq {\overline{C}_2} \varepsilon^2,$ and   the already proven inequality \eqref{eq_boundalpha1-final} (that we have also trivially in the unstable even case), we conclude again that, 
\begin{equation} \label{eq_S1}
|S_1| \leq  {\overline{C}_2} \varepsilon^2 \left(|\alpha_1(0)|^2+ \dfrac{4 K_m\varepsilon^2}{{|D_0|}} +  {\overline{C}_2}^2 \varepsilon^4 \right).
\end{equation}
Finally,  combining \eqref{eq_S1} and \eqref{eq_S2} and choosing $\varepsilon$ sufficiently small (wrt $m$ and  {$\overline{C}_2$}):
\begin{equation} \label{eq_bRHS13}
|( \overline{\mathcal Q}_2^{(\varepsilon)} [\overline{\alpha_1}], \overline{\alpha}_2)| \leq 
 K_m  {\overline{C}_2}^3 \varepsilon^5 +   {\overline{C}_2} \varepsilon^2 \left(|\alpha_1(0)|^2+ \dfrac{4 K_m\varepsilon^2}{{|D_0|} } \right).
\end{equation}
Combining \eqref{eq_bRHS11} with $\eta$ sufficiently small (wrt $k_0$)  {and letting $K_\eta$ be a large enough constant such that $K_\eta\geq 3+1/(2\eta)$},  {\eqref{eq_bRHS12} and} \eqref{eq_bRHS13},  we conclude that: 
\begin{equation} \label{eq_balpha1l2}
\dfrac{\textrm{d}}{\textrm{d}t} \|\overline{\alpha}_2\|^2  + \dfrac{1}{2^{k_0+4}} 
\|\overline{\alpha}_2\|^2 \leq   {K_\eta}K_m \left(  {\overline{C}_2}^{3} \varepsilon^5 + \varepsilon^4 \right) +  {\overline{C}_2} \left(|\alpha_1(0)|^2+ \dfrac{4 K_m\varepsilon^2}{ {|D_0|}}\right) \varepsilon^2
 \end{equation}
and after time integration:
\[
\|\overline{\alpha}_2\|^2 \leq \|\overline{\alpha}_2(0)\|^2 +  2^{k_0+4}  \left(  {K_\eta} K_m \left(  {\overline{C}_2}^3 \varepsilon^5 + \varepsilon^4 \right) +   {\overline{C}_2} \left(|\alpha_1(0)|^2+ \dfrac{4 K_m\varepsilon^2}{ {|D_0|}}\right) {\e^2} \right)
\]
We have finally the expected bound of \eqref{eq_continuation} if: 
\begin{equation} \label{eq_condC1-2}
 {\overline{C}_2} \geq 2 \left[ \overline{C}_0^2 + 2^{k_0+4} \left(   {2K_\eta} K_m +  {\overline{C}_2} \left( \overline{C}_0^2 + \dfrac{4 K_m}{ {|D_0|}}\right) \right)\right]^{\frac 12}
\end{equation}
and $\varepsilon$ is chosen sufficiently small.

We observe that \eqref{eq_condC1_1} reduces to choosing ${\overline{C}_2}$
sufficiently large wrt $\overline{C}_0.$ In \eqref{eq_condC1-2}, we note that the right-hand side is sublinear in ${\overline{C}_2}$ so that \eqref{eq_condC1-2} is surely satisfied if ${\overline{C}_2}$ is chosen sufficiently large wrt $m$ and $\overline{C}_0.$  We can then match the two conditions with ${\overline{C}_2}$ sufficiently large. This ends up the proof.

\subsection{Asymptotic stability of stationary states (Proof of Theorem \ref{thm_longtime})}

We {conclude} this part with a proof of our main result, {Theorem \ref{thm_longtime}}. 
For this, we fix ${C}.$ Applying {\bf Proposition \ref{prop_orbstability_stable} } in the stable case, or {\bf Proposition \ref{prop_orbstability_unstable}} in the unstable even case, we obtain that, for {any} non-negative bounded and concentrated initial data $h_0,$ {there exists} a unique global solution $h(t,\cdot)$ to \eqref{eq_hevolution} that  {satisfies, since $\overline{C}_1\leq \overline C_2$},
\[
|\alpha_1(t)| \leq {\overline{C}_2} \varepsilon \qquad 
\|\overline{\alpha}_2(t)\| \leq {\overline{C}_2} \varepsilon^2 \qquad \forall \,t \geq 0.
\]  
Up to {increasing} the size of ${\overline{C}_2}$ and {restricting} the size of $\varepsilon,$ {{\bf Theorem \ref{thm_stat}} guarantees also}   the existence of a unique stationary solution $\bar{h}$ to \eqref{eq_hevolution}
in $L^2(\mathbb R,G(x){\rm d}x)$ whose coefficient in the Hermite basis satisfy:
\begin{equation} \label{eq_apriori}
|{\alpha}^{\infty}_1 | \leq {\overline{C}_2} \varepsilon \qquad 
\|\overline{\alpha}^{\infty}_2\| \leq {\overline{C}_2} \varepsilon^2.
\end{equation}
 We may then compute the difference $h(t) - \bar{h}$ and especially its coefficients in the Hermite basis $(\beta_{k})_{k \in \mathbb N}.$ {We recall that $\beta_0=0$ so that we restrict to the coefficients larger than $1$ in our computations.} {\bf Theorem \ref{thm_longtime}}
is then a consequence to the following statements. In the stable case we show:
\begin{proposition}
{Assume that $m^{''}(0)>0$. Then,} there exists $K^{(0)}$ and $K_m^{(0)} >0$ depending both only on $m$ such that:
\[
\|\overline{\beta}_1(t)\|_{\ell^2(\mathbb N_1)} \leq K^{(0)}\|\overline{\beta}(0)\|_{\ell^2(\mathbb N_1)} \exp(-K_m^{(0)}\varepsilon^2t) \qquad \forall \, t \geq 0. 
\]
\end{proposition}
Item \textit{(i)} of {\bf Theorem \ref{thm_longtime}} is a consequence to this statement. 
 Item {\textit{(ii)}} of {\bf Theorem \ref{thm_longtime}} is  a consequence to the following proposition:
\begin{proposition}
Assume that $m^{''}(0)<0$ and that $m$ is even. Then, for any even initial data there exists $K_m$ depending only on $m$ for which:
 \[
\|\overline{\beta}_1(t)\|_{\ell^2(\mathbb N_1)} \leq \|\overline{\beta}_1(0)\|_{\ell^2(\mathbb N_1)} \exp(-K_m t) \qquad \forall \, t \geq 0. 
\]
\end{proposition}

In what follows, we give again a unified proof of both propositions.
Substracting from the equation for $\overline{\alpha}_1$ the stationary solution $\overline{\alpha}_1^{(\infty)},$ we obtain that $\overline{\beta}_1$ is a solution to the differential system:
\begin{equation} \label{eq_beta_evol}
\dfrac{\textrm{d}\overline{\beta}_1}{{\textrm d}t}
= \overline{\mathcal L}^{(\varepsilon)} [ \overline{\beta}_1] + 
\overline{\mathcal Q}_{\infty}^{(\varepsilon)} [\overline{\beta}_1],
\end{equation}
with $\overline{\mathcal L}^{(\varepsilon)}$  defined {in \eqref{def:L}} and
\[
\overline{\mathcal Q}_{\infty}^{(\varepsilon)} [\overline{\beta}] = 
 \left[ \dfrac{\sqrt{k!}}{2^k} \sum_{l=1}^{k-1} \dfrac{\alpha_l \beta_{k-l} + \beta_l \alpha^{(\infty)}_{k-l}}{\sqrt{l! (k-l)!}}  + \sum_{l=1}^{\infty} (\alpha_k  \beta_l + \beta_k \alpha^{(\infty)}_l )   m_l^{(\varepsilon)}\right]_{k\in\mathbb N_1}.
\]
We show that the solutions to this system -- satisfying the {\em a priori} bounds \eqref{eq_apriori} -- decay exponentially to $0$ when $\varepsilon$ is sufficiently small.  

\medskip

In the unstable even case, we have directly by symmetry that $\beta_1 = 0.$ In the stable case, we multiply the first equation of \eqref{eq_beta_evol} by $\beta_1.$ We recover:
\[
\dfrac{1}{2} \dfrac{\textrm{d}}{{\textrm d}t} |\beta_1|^2 = { D_{\varepsilon} }|\beta_1|^2  - \left(\sum_{l=2}^{\infty}\beta_l H_l,  \beta_1  m_{\varepsilon}H_1 \right)_{G} + \sum_{l=1}^{\infty} (\beta_l \alpha_1 + \alpha_l^{(\infty)} \beta_1) \beta_1 m_{l}^{(\varepsilon)},
\]
where $D_{\varepsilon} ={(H_0,m_{\varepsilon} H_0)_G- (H_1, m_{\varepsilon} H_1)_{G} }$ 
satisfies {(recall $m^{''}(0)=1$)}
\[
|D_{\e} - \varepsilon^2 D_0 | \leq K_m \varepsilon^3, \qquad {D_0<0}.
\]
We estimate the remainder terms, using   Lemma \ref{lem_normmHk} and similarly to \eqref{inequality-amH},
\[
 \left| \left(\sum_{l=2}^{\infty}\beta_l H_l,  \beta_1  m_{\varepsilon}H_1 \right)_{G}\right| \leq K_m \varepsilon^2 |\beta_1|  \|\overline{\beta}_2\|
 \leq K_m \left( \varepsilon^3 |\beta_1|^2 + \varepsilon \|\overline{\beta}_2\|^2 \right)
\]
with {again the same convention for $\|\cdot\|$ and} $K_m$ depending only on $m.$ As for the last term, we use that $|\alpha_1| + \|\alpha_2^{(\infty)}\| \leq 2 \bar{C}_1 \varepsilon$ and {$|\overline{\alpha}_1| + \|\overline{\alpha}_2^{(\infty)}\| \leq 2 {\overline{C}_2} \varepsilon$.}}
This entails thanks to {Lemma \ref{lem_m} and} a Cauchy-Schwarz inequality that
\[
\begin{aligned}
\left|  \sum_{l=1}^{\infty} (\beta_l \alpha_1 + \alpha_l^{(\infty)} \beta_1) \beta_1 m_{l}^{(\varepsilon)} \right|
& \leq |\alpha_1| |\beta_1|   \left( \sum_{l=1}^{\infty} |\beta_l|^2 \right)^{\frac 12}  \left( \sum_{l=1}^{\infty}  |m_{l}^{(\varepsilon)}|^2 \right)^{\frac 12}  \\
& \qquad + |\beta_1|^2  \left( \sum_{l=1}^{\infty} |\alpha_l^{(\infty)}|^2 \right)^{\frac 12}   \left( \sum_{l=1}^{\infty}  |m_{l}^{(\varepsilon)}|^2 \right)^{\frac 12} \\
& \leq  {\overline{C}_2} K_m \varepsilon^3 \left(|\beta_1|^2 + \|\overline{\beta}_2\|^2 \right)
\end{aligned}
\]
We conclude that, for $\varepsilon$ sufficiently small,
\begin{equation} \label{eq_beta1_evol}
\dfrac{1}{2} \dfrac{\textrm{d}}{{\textrm d}t} |\beta_1|^2  + \dfrac{{|D_0|}}{2} \varepsilon^2 |\beta_1|^2 \leq  K_m  \varepsilon \|\overline{\beta}_2\|^2.
\end{equation}

\medskip

We obtain {a similar inequality on $\overline{\beta}_2$ in both stable and unstable even cases. 
To this end}, we multiply the equation on $\overline{\beta}_2$ by $\overline{\beta}_2.$ We obtain
\beq
\label{eq-beta1bar}
\begin{aligned}
\dfrac{1}{2}  \dfrac{\textrm{d}}{{\textrm d}t} \|\overline{\beta}_2\|^2&  =
 \sum_{k=2}^{\infty} \left( \dfrac{1}{2^{k-1}}  - 1  +(H_0,  m_{\varepsilon} H_0)_G   \right) |\beta_k|^2 
- \left(\sum_{l=1}^{\infty} \beta_l H_l,  m_{\varepsilon} \sum_{k=2}^{\infty} \beta_k H_k \right)_G  \\
& \quad  + 
\sum_{k=2}^{\infty} \left[ \dfrac{\sqrt{k!}}{2^k} \sum_{l=1}^{k-1} \dfrac{\alpha_l \beta_{k-l} + \beta_l \alpha^{(\infty)}_{k-l}}{\sqrt{l! (k-l)!}}  + \sum_{l=1}^{\infty} (\alpha_k  \beta_l + \beta_k \alpha^{(\infty)}_l )   m_l^{(\varepsilon)}\right] \beta_k.
\end{aligned}
\eeq
 
We compute the right-hand side of this system as previously.   We split the right-hand  into $L+NL$ where $L$ stands for the first line and $NL$ for the second one.  Introducing $k_0$ again, such that $(1+m_{\e}) \geq 2^{-k_0}$, we rewrite
$L = L_{(\leq)}+ L_{(\geq)} + L_{(c)}$, with 
\[
\begin{aligned}
L_{(\leq)} &= \sum_{k=2}^{{k_0+1}}   \left( \dfrac{1}{2^{k-1}}  - 1  +(H_0,  m_{\varepsilon} H_0)_G   \right) |\beta_k|^2 
- \left(\sum_{l=2}^{{k_0+1}} \beta_l H_l,  m_{\varepsilon} \sum_{k=2}^{{k_0+1}} \beta_k H_k \right)_G \\
L_{(\geq)} &=  \sum_{k={k_0+2}}^{\infty}   \left( \dfrac{1}{2^{k-1}}  - 1  +(H_0,  m_{\varepsilon} H_0)_G   \right) |\beta_k|^2 
- \left(\sum_{l={k_0+2}}^{\infty} \beta_l H_l,  m_{\varepsilon} \sum_{k={k_0+2}}^{\infty} \beta_k H_k \right)_G\\
L_{(c)} & =  - \left(\beta_1 H_1,  m_{\varepsilon} \sum_{k=2}^{\infty} \beta_k H_k \right)_G - 2 \left(\sum_{l={k_0+2}}^{\infty} \beta_l H_l,  m_{\varepsilon} \sum_{k=2}^{{k_0+1}} \beta_k H_k \right)_G.
\end{aligned}
\]
For the first and {second terms},  we infer as {in Section \ref{sec:orbital-stabilty}} that, for $\varepsilon$ sufficiently small,
\[
L_{(\leq)}  \leq  - \dfrac{1}{4} \sum_{k=2}^{{k_0+1}} |\beta_k|^2,  \qquad
L_{(\geq)}  \leq  - \dfrac{1}{2^{k_0+2}} \sum_{k={k_0+2}}^{\infty} |\beta_k|^2.
\]
As for $L_{(c)}$, we obtain, with similar computations as in the previous subsection, that
\[
| L_{(c)} |  \leq  K_m{\varepsilon^2} \left(   |\beta_1|  {\|\overline{\beta}_2\|} + \|\overline{\beta}_2 \|^2 \right) \leq  K_m \left( \varepsilon^3  |\beta_1|^2 + \varepsilon \|\overline{\beta}_2\|^2 \right).
\]
Eventually, we conclude that, for $\varepsilon$ sufficiently small,
\begin{equation} \label{eq_bbeta1_evollin}
L \leq - \dfrac{1}{2^{k_0+3}}\|\overline{\beta}_2\|^2 + K_m \varepsilon^3 |\beta_1|^2 .
\end{equation}
As for the nonlinear term, we write $NL= NL_1 + NL_2$, with
\[
NL_1 = \sum_{k=2}^{\infty} \dfrac{\sqrt{k!}}{2^k} \sum_{l=1}^{k-1} \dfrac{\alpha_l \beta_{k-l} + \beta_l \alpha^{(\infty)}_{k-l}}{\sqrt{l! (k-l)!}}  \beta_k\,,
\qquad
NL_2 = \sum_{k=2}^{\infty} \sum_{l=1}^{\infty} (\alpha_k  \beta_l + \beta_k \alpha^{(\infty)}_l )   m_l^{(\varepsilon)}  \beta_k\,.
\] 
Concerning $NL_2,$ {using the Cauchy-Schwarz inequality and Lemma \ref{lem_m}, we obtain}
\[
\begin{aligned}
|NL_2| &\leq \sum_{k=2}^{\infty}  |\alpha_k| |\beta_k| \sum_{l=1}^{\infty}  |\beta_l| | m_{l}^{(\varepsilon)}| + 
\sum_{k=2}^{\infty} |\beta_k|^2 \sum_{l=1}^{\infty} |\alpha_l^{(\infty)}| |m_{l}^{(\varepsilon)}| \\
& \leq \|\overline{\alpha}_2\| \|\overline{\beta}_2\| \|\overline{\beta}_1\|_{\ell^2(\mathbb N_1)} \|\overline{m}^{(\varepsilon)}_1\| + \|\overline{\beta}_2\|^2 \|\overline{\alpha}^{(\infty)}_1\|_{\ell^2(\mathbb N_1)} \|\overline{m}^{(\varepsilon)}_1\| \\
& \leq {\overline{C}_2} K_m \varepsilon^3  \left(  |\beta_1|^2  +   \|\overline{\beta}_2\|^2\right).
\end{aligned}
\]
As for $NL_1,$ we use again the Cauchy-Schwarz inequality and a classical binomial identity to yield
\[
\begin{aligned}
NL_1 & \leq \|\overline{\beta}_2\|\left( \left(  \sum_{k=2}^{\infty} \left[ \sum_{l=1}^{k-1} \dfrac{1}{2^k} \sqrt{\begin{pmatrix} k \\ l \end{pmatrix}} |\alpha_l ||\beta_{k-l}| \right]^2  \right)^{\frac 12}
+ \left(  \sum_{k=2}^{\infty} \left[ \sum_{l=1}^{k-1} \dfrac{1}{2^k} \sqrt{\begin{pmatrix} k \\ l \end{pmatrix}} |\beta_l ||\alpha_{k-l} ^{(\infty)}|\right]^2  \right)^{\frac 12}
\right)\\
&\leq  \|\overline{\beta}_2\| \left( \left(\sum_{k=2}^{\infty}  \f{1}{2^{k}}    \sum_{l=1}^{k-1} |\beta_{k-l}|^2|\alpha_l|^2\right)^{\f12}+\left(\sum_{k=2}^{\infty}  \f{1}{2^{k}}    \sum_{l=1}^{k-1} |\beta_{l}|^2|\alpha_{k-l}^{(\infty)}|^2\right)^{\f12}\right)\\
&\leq \f{1}{2} \|\overline{\beta}_2\|\|\overline \beta_1\|_{\ell^2(\mathbb N_1)}\left(\|\overline \alpha_1\|_{\ell^2(\mathbb N_1)}^2+\|\overline \alpha_1^{(\infty)}\|_{\ell^2(\mathbb N_1)}^2\right)^{\f12}\\
& \leq   C_2 \e
\|\overline{\beta}_2\|  (\beta_1^2+\|\overline \beta_2\|^2)^{\f12}
\\
&\leq \f{\overline C_2^2\e^2}{\eta}|\beta_1|^2+(\f\eta 4+\overline C_2 \e) \|\overline \beta_2\|^2\leq \f{\overline C_2^2\e^2}{\eta}|\beta_1|^2+\f\eta 2  \|\overline \beta_2\|^2,
\end{aligned}
\]
for arbitrary $\eta \in (0,1) $ {and $\e$ small enough}.
Eventually, we obtain that, for arbitrary fixed $\eta \in (0,1),$ we can find a constant $K_{\eta}$ such that, for $\varepsilon$ sufficiently small (depending on $\eta$),
\[
NL \leq   K_{\eta} \varepsilon^2 |\beta_1|^2 + {\eta} \|\overline{\beta}_2\|^2 .
\]
Chosing $\eta=1/2^{k_0+4},$ combining with \eqref{eq_bbeta1_evollin} and restricting the size of  $\varepsilon$ if necessary, we finally obtain a constant $\tilde{K}_m$ for which
\[
L+ NL \leq - \dfrac{1}{2^{k_0+4}} \|\overline{\beta}_2\|^2 + \tilde{K}_m \varepsilon^2 |\beta_1|^2  .
\] 
{Combining this with \eqref{eq-beta1bar} we obtain that}
\begin{equation} \label{eq_bbeta1_evol}
\dfrac{1}{2} \dfrac{\textrm{d}}{\textrm{d}t} \|\overline{\beta}_2\|^2 +  \dfrac{1}{2^{k_0+4}} \|\overline{\beta}_2\|^2 \leq  \tilde{K}_m \varepsilon^2 |\beta_1|^2  .
\end{equation}

\medskip

This completes the proof in the unstable even case since $\beta_1 \equiv 0$ and the latter inequality entails the expected exponential decay of $|\overline{\beta}_1|$. In the stable case,  we {multiply  \eqref{eq_beta1_evol}  with $4 \tilde{K}_m/{|D_0|}$ and do a summation with 
\eqref{eq_bbeta1_evol}, to obtain} 
\[
\dfrac{1}{2} \dfrac{\textrm{d}}{\textrm{d}t} \left[ \dfrac{4 \tilde{K}_m}{{|D_0|}} |\beta_1|^2  + \|\overline{\beta}_2\|^2  \right] 
 + \dfrac{1}{2^{k_0+4}} \|\overline{\beta}_2\|^2 + 2 \tilde{K_m} \varepsilon^2 |\beta_1|^2 \leq  \tilde{K}_m \varepsilon^2 |\beta_1|^2 + \dfrac{4 \tilde{K}_m}{{|D_0|}}  K_m \varepsilon \|\overline{\beta}_2\|^2. 
\]
Restricting again the size of $\varepsilon$ if necessary, we conclude that:
\[
\dfrac{1}{2} \dfrac{\textrm{d}}{\textrm{d}t} \left[ \dfrac{4 \tilde{K}_m}{{|D_0|}} |\beta_1|^2  + \|\overline{\beta}_2\|^2  \right] 
+ \dfrac{1}{2^{k_0+5}} \|\overline{\beta}_2\|^2 +  \tilde{K}_m \varepsilon^2 |\beta_1|^2 \leq  0,
\]
that implies again the expected exponential decay of  $\|\overline{\beta}_1\|_{\ell^2(\mathbb N_1)}$ with time.

\appendix
\section{The admissibility condition of the extremum point }
\label{sec:nonexistence}

In this section we prove that the second inequality in Assumption \fer{as:m-extremum} is indeed a necessary condition for the existence of a concentrated steady solution around $x_m$.
Let's consider a steady solution
$$
T_\e[q_\e]=q_\e(1+m(x)-\int m(x) q_\e(x)dx).
$$
We define
$$
\min_{x\in \R} m(x)=m_-.
$$
 We also let $\Omega_\eta$ be the following set
$$
\Omega_\eta :=\{ x\, |\, m_-\leq m(x)\leq m_-+1+\eta\},
$$
with $\eta>0$. We then claim that
$$
\int_{\Omega_\eta}q_\e(t,x)dx \geq \f{\eta}{1+\eta}.
$$
This inequality would imply that if $m(x_m)>m_-+1+\eta$, then there is no steady solution concentrated around $x_m$.

To prove this property we define 
$$
\overline m(x)=m(x)-m_- -1-\eta.
$$
Then, we have
$$
\Omega_\eta =\{ x\, |\,   -1-\eta\leq \overline  m(x)\leq 0\}
$$
We can re-write the equation on $q_\e$ as follows
$$
T[q_\e]=q_\e(1+\overline  m-\int \overline  m  q_\e dx) .
$$
From the positivity of $T(q_\e)$ and $q_\e$ we deduce that
\[
\int_\R \overline  m(y) q_\e(y) dy \leq 1+ \overline m (x).
\]
We then evaluate this at the minimum point of $ \overline m$ to obtain
$$
\int_\R \overline  m(y) q_\e(y) dy \leq -\eta  .
$$
We then use the fact that
$$
\int_{\Omega_\eta}  \overline  m(y) q_\e(y) dy \leq \int_\R\overline  m(y) q_\e(y) dy,
$$
to obtain that
$$
(-1-\eta)\int_{\Omega_\eta}  q_\e(y) dy\leq -\eta,
$$
which leads to the result.

{
\section{${T}_{\e}$ is continuous on $L^1((1+x^{2})^ldx)$.}
\label{sec:appendix}
In this appendix, we focus on the following proposition:
\begin{proposition} \label{prop_contT}
Let $\e >0$ and $l \in \mathbb N.$ The mapping $q \mapsto T_{\e}[q]$ is bilinear continuous on $L^1((1+x^{2l})dx).$
\end{proposition}

We provide a proof to make the paper self-contained. However, we point out that this result is a straightforward continuation of \cite[Lemma A.2]{JG.MH.SM:23}.

\medskip

\begin{proof}
We give a proof in case $\e =1.$ Since bilinearity is obvious, we focus on integrability properties. 
Let $l \in \mathbb N$ and $q \in L^1((1+x^{2})¨^ldx).$ Given $\ell \in \{0,\ldots,l\}$, we have for all $x \in \R$
\[
\begin{aligned}
& |x^{2\ell}T_1[q](x)|  = \int_{\mathbb R} \int_{\mathbb R} \left(x - \dfrac{y_1+y_2}{2}  +\dfrac{y_1+y_2}{2} \right)^{2\ell} \Gamma _1\left(x - \dfrac{y_1+y_2}{2} \right) |q(y_1)| |q(y_2)| dy_1dy_2 \\
&  = \sum_{k=0}^{2\ell} \sum_{i=0}^{k} \dfrac{\begin{pmatrix}2\ell \\ k \end{pmatrix} \begin{pmatrix} k \\ i \end{pmatrix}}{2^{{k}}} \int_{\mathbb R} \int_{\mathbb R} \left(x - \dfrac{y_1+y_2}{2} \right)^{2\ell-k} \Gamma_1 \left( x - \dfrac{y_1+y_2}{2} \right)  
y_1^{i} |q(y_1)| y_2^{k-i} |q(y_2)|  dy_1y_2 \\
\end{aligned}
\]
Here we note that for all instances we have $i \leq 2l$ and $k-i \leq 2l$ and that when $j \leq 2l$
we have $y^{j}q(y) \in L^1(\mathbb R)$ with :
\[
\int_{\mathbb R} |y^{j}q(y)| dy \leq C(j,\ell) \int_{\mathbb R} (1+y^{2l}) |q(y)| dy.
\] 
We can then apply Fubini theorem to yield that $x^{2\ell} T_1[q](x) \in L^1(\mathbb R)$ with:
\[
\int_{\mathbb R} x^{2\ell} |T_1[q](x)| dx \leq C(\ell,l) \left( \int_{\mathbb R} (1+y^{2l}) |q(y)| dy\right)^2.
\]
This ends the proof.
\end{proof}

}
\section{$\tilde{T}_1$ is continuous.}
\label{sec:appendixT}
{ Let $h_1,h_2\, \in L^2(\mathbb R,G(x){\rm d}x),$ with  
\[
h_1 = \sum_{k=0}^{\infty} \alpha_k H_k,\quad h_2 = \sum_{k=0}^{\infty} \beta_k H_k,
\]
such that $(\alpha_k)_{k \in \mathbb N} \in \ell^2(\mathbb N)$ and  $(\beta_k)_{k \in \mathbb N} \in \ell^2(\mathbb N)$. We define
\beq
\label{def:T-tilde}
\tilde{T}_1[h_1,h_2] = \sum_{k=0}^{\infty} \gamma_k H_k
\quad
\text{where}
\quad
\gamma_k = \dfrac{\sqrt{k!}}{2^{k}} \sum_{l=0}^{k} \dfrac{\alpha_l \beta_{k-l}}{\sqrt{l!(k-l)!}}\qquad
\forall \, k \in \mathbb N.
\eeq 
\begin{lemma}
\label{lem:T-cont}
The operator $\tilde{T}_1$ defines a bounded bilinear mapping $L^2(\mathbb R,G(x)dx)^2 \to L^2(\mathbb R,G(x)dx)$. Moreover, we have
\beq
\label{eq:T-tildeT}
\tilde{T}_1[h_1,h_2] = G^{-1} T_1[h_1G,h_2G] \qquad \forall \, (h_1,h_2) \in (L^2(\mathbb R,Gdx))^2.
\eeq
\end{lemma}}

 {\begin{proof}
Bilinearity is obvious. We prove that $\tilde{T}_1$ is a well-defined continuous mapping.  In particular,  we show that the right-hand side of \eqref{def:T-tilde} is converging in $L^2(\mathbb R, G{\rm d}x)$. We have
\[
\begin{aligned}
\sum_{k \in \mathbb N} |\gamma_k |^2  &  \leq \sum_{k=0}^{\infty} \dfrac{k!}{4^k}  \left( \sum_{l=0}^k \dfrac{|\alpha_l| |\beta_{k-l}| }{l! (k-l)!} \right)^{2}  \\
&   \leq \sum_{k=0}^{\infty} \dfrac{1}{4^k} \sum_{l'=0}^k \begin{pmatrix} k \\ l' \end{pmatrix} \sum_{l=0}^{k} \alpha_l^2 \beta_{k-l}^2.
\end{aligned}
\]
Up to extending the sequences $\alpha_k$ and $\beta_k$ by $\alpha_k = 0$ and $\beta_k=0$ if $k \in \mathbb Z \setminus \mathbb N,$ we obtain
\[
\begin{aligned}
\sum_{k \in \mathbb N} |\gamma_k|^2   \leq \sum_{k \in \mathbb Z} \sum_{l\in \mathbb Z} \alpha_l^2 \beta_{k-l}^2 \leq  \left( \sum_{k \in \mathbb Z} |\alpha_l|^2 \right)   \left( \sum_{k \in \mathbb Z} |\beta_l|^2 \right).
\end{aligned}
\]
Consequently, $\tilde{T}_1$ defines a bounded bilinear mapping $L^2(\mathbb R,G(x)dx)^2 \to L^2(\mathbb R,G(x)dx)$. The interested reader can also verify that the operator $(h_1,h_2)\to G^{-1}T_1[h_1G,h_2G]$ is a bilinear continuous mapping $L^2(\mathbb R,G(x)dx)^2 \to L^2(\mathbb R,G(x)dx)$. Finally notice that the equality \eqref{eq:T-tildeT} holds for $h_1=H_k$ and $h_2=H_l$, for any $(k,\ell) \in \mathbb N^2$, thanks to \eqref{tildeT1-Hermite}. The equality \fer{eq:T-tildeT} then follows from the continuity of both operators. \qed
\end{proof}}

{\section*{Acknowledgements}
This work is funded by the European Union ERC-2024-COG MUSEUM-101170884. Views and opinions expressed are however those of the authors only and do not necessarily reflect those of the European Union or the European Research Council Executive Agency (ERCEA). Neither the European Union nor the granting authority can be held responsible for them. }

\bibliographystyle{plain}

\end{document}